\documentclass[10pt]{amsart}
\usepackage{amssymb}
\usepackage{amsbsy}
\usepackage{amscd}
\usepackage[mathscr]{eucal}
\usepackage{verbatim}
\usepackage{version}
\usepackage{color}
\usepackage[matrix,arrow]{xy}
\usepackage{version}

\def\cal{\mathcal}
\def\Bbb{\mathbb}
\def\frak{\mathfrak}

\newenvironment{NB}{
\color{red}{\bf NB}. \footnotesize 
}{}
\excludeversion{NB}
\newenvironment{NB2}{
\color{blue}{\bf NB}. \footnotesize
}{}

\newcommand{\Pic}{\operatorname{Pic}}

\newcommand{\Supp}{\operatorname{Supp}}
\newcommand{\ch}{\operatorname{ch}}

\newcommand{\Coh}{\operatorname{Coh}}

\newcommand{\Ext}{\operatorname{Ext}}
\newcommand{\Hom}{\operatorname{Hom}}

\newcommand{\im}{\operatorname{im}}
\newcommand{\Aut}{\operatorname{Aut}}
\newcommand{\rk}{\operatorname{rk}}

\newcommand{\chr}{\operatorname{char}}

\newcommand{\NS}{\operatorname{NS}}
\newcommand{\coker}{\operatorname{coker}}
\newcommand{\td}{\operatorname{td}}

\newcommand{\Amp}{\operatorname{Amp}}

\newcommand{\Per}{\operatorname{Per}}

\newcommand{\alg}{\operatorname{alg}}

\newcommand{\Stab}{\operatorname{Stab}}

\newcommand{\Ima}{\operatorname{Im}}
\newcommand{\Rea}{\operatorname{Re}}
\newcommand{\pH}{{^p H}}

\font\b=cmr10 scaled \magstep5
\def\bigzerou{\smash{\lower1.7ex\hbox{\b 0}}}

\numberwithin{equation}{section}

\theoremstyle{plain}
 \newtheorem{thm}{Theorem}[section]
 \newtheorem{lem}[thm]{Lemma}
 \newtheorem{prop}[thm]{Proposition}
 \newtheorem{cor}[thm]{Corollary}
 
\theoremstyle{definition}
 \newtheorem{defn}[thm]{Definition}
 
\theoremstyle{remark}
 \newtheorem{rem}[thm]{Remark}
 \newtheorem{ex}[thm]{Example}
 
\begin{document}

\title{
Some remarks on 
Bridgeland stability conditions on K3 and Enriques surfaces}
\author{K\={o}ta Yoshioka}
\address{Department of Mathematics, Faculty of Science,
Kobe University,
Kobe, 657, Japan
}
\email{yoshioka@math.kobe-u.ac.jp}

\thanks{
The author is supported by the Grant-in-aid for 
Scientific Research (No.\ 26287007,\ 24224001), JSPS}


\begin{abstract}
We give some remarks on our papers with Minamide and Yanagida
on Bridgeland stability conditions.
We also give a remark on stability conditions on Enriques surfaces,
and give another proof of the projectivity
of  the coarse moduli spaces of semi-stable objects, which were obtained by Nuer.  
\end{abstract}

\maketitle

\section{Introduction}
In \cite{Br:stability}, Bridgeland introduced 
a very useful notion of stability condition on the derived category
${\bf D}(X)$ of coherent sheaves on a projective scheme $X$,
and showed that the set of stability conditions
$\Stab(X)$ has a structure of complex manifold.
A stability condition 
$\sigma=({\cal P}_\sigma,Z_\sigma)$
consists of an abelian category ${\cal P}_\sigma$
which is a heart of a $t$-structure of ${\bf D}(X)$
and a stability function
$Z_\sigma:{\bf D}(X)  \to {\Bbb C}$ with some properties
such as the Harder-Narasimhan property.
If $X$ is a K3 surface, a detailed description of 
a connected component of $\Stab(X)$ was given in \cite{Br:3}.  
In particular Bridgeland constructed a particular family of 
stability conditions
so called {\it geometric stability conditions}:
They are stability conditions 
such that ${\cal O}_x$ $(x \in X)$ are stable
with the same phase, and forms an open subset of
$\Stab(X)$.
In \cite{MYY:2011:1}, \cite{MYY:2011:2},
we studied Fourier-Mukai transforms
on K3 and abelian surfaces by using Bridgeland stability conditions.
For this purpose, we constructed stability conditions 
on K3 surfaces by extending 
Bridgeland's construction of geometric stability conditions \cite{Br:3}.
In this note, we give some remarks on our papers.
We first add a remark on the relation between Bridgeland's construction and our extension. 
In \cite{MYY:2011:2}, we constructed isomorphisms
of the moduli stacks by using Fourier-Mukai
transforms under some technical conditions.
In this article, we shall remove one of the technical conditions.
We also give a remark on stability conditions on Enriques surfaces,
and prove that the coarse moduli spaces of semi-stable objects are 
projective schemes, which were obtained by Nuer \cite{N2}.

\section{Preliminaries}

\subsection{Notation.}

Let $X$ be a K3 surface over an algebraically closed field $k$. 
For $E \in{\bf D}(X)$,
We denote the Mukai vector of $E$ by
$$
v(E)=\ch(E)\sqrt{\td_X}=\rk E+c_1(E)+
\left(\ch_2(E)+\rk E \varrho_X \right),
$$
where $\varrho_X$ is the fundamental class of $X$.
We also set
$$
(r,\xi,a)=r+\xi+a \varrho_X, \xi \in \NS(X), a \in {\Bbb Q}.
$$
Let $(H^*(X,{\Bbb Z}),\langle\;\;,\;\; \rangle)$
be the Mukai lattice of $X$,
where
$$
\langle v_1,v_2 \rangle:=(\xi_1 \cdot \xi_2)-r_1 a_2- r_2 a_1,\;
v_i=(r_i,\xi_i,a_i), i=1,2.
$$
We shall give some notation on stability conditions and
also some properties.
For more details, see \cite{Br:stability} and \cite{Br:3}. 
For a stability condition $\sigma=({\cal P}_\sigma,Z_\sigma)$,
$\phi_\sigma$ is the phase function and
${\cal P}_\sigma(\phi)$ denotes the set of $\sigma$-semi-stable
objects $E$ with $\phi_\sigma(E)=\phi$.
${\cal P}_\sigma(I)$ denotes the category generated by
objects $E \in \cup_{\phi \in I}{\cal P}_\sigma(\phi)$.
For $E \in {\bf D}(X)$, $\phi_\sigma^+(E)$ is the maximum of the stable
factors of $E$ and
$\phi_\sigma^-(E)$ the minimum of the stable
factors of $E$. 

Let ${\cal P}(X)$ be the subset of
$v(K(X))_{\Bbb C}$ consisting of $\mho$ such that
$\Rea \mho$ and $\Ima \mho$ span a positive definite
2-plane in $v(K(X))_{\Bbb R}$.
We shall regard ${\cal P}(X)$ as a subset 
of $\Hom(v(K(X)),{\Bbb C})$ by
$\langle \mho,\bullet \rangle$.
Let ${\cal P}^+(X)$ be the connected component of ${\cal P}(X)$
containing $e^{\beta+i \omega}$, where $\omega$ is ample.
Let $P^+(X)_{\Bbb R}$ be the positive cone of $X$ and 
$\Amp(X)_{\Bbb R} (\subset P^+(X)_{\Bbb R})$ the ample cone of $X$.
We have an action of $GL^+_2({\Bbb R})$ on
$\Hom(v(K(X)),{\Bbb C})$ and 
 ${\cal P}^+(X)/GL^+_2({\Bbb R})=\NS(X)_{\Bbb R} \times P^+(X)_{\Bbb R}$.
Let $\Delta(X)$ be the set of Mukai vectors $u$ with
$\langle u^2 \rangle=-2$.
We set 
$$
{\cal P}_0^+(X):={\cal P}^+(X) \setminus 
\cup_{u \in \Delta(X)}u^\perp.
$$

For $(\beta,\omega) \in \NS(X)_{\Bbb R} \times \Amp(X)_{\Bbb R}$
with $\langle e^{\beta+i \omega},u \rangle \not \in 
{\Bbb R}_{ \leq 0}$ $(u \in \Delta(X))$,
Bridgeland constructed a stability condition
$\sigma_{(\beta,\omega)}=({\cal A}_{(\beta,\omega)},Z_{(\beta,\omega)})$
such that $Z_{(\beta,\omega)}(\bullet)=
\langle e^{\beta+i \omega},\bullet \rangle$
and ${\cal O}_x$ $(x \in X)$ are stable objects.
Up to the action of $\widetilde{GL}^+_2({\Bbb R})$
on $\Stab(X)$, it is characterized as a stability condition
$\sigma$ such that ${\cal O}_x$ $(x \in X)$ are $\sigma$-stable
objects with a fixed phase and $Z_\sigma \in {\cal P}^+(X)$.
Let $U(X)$ be the open subset of $\Stab(X)$ consisting of
these stability conditions.
Let $\Stab^{\dagger}(X)$ 
be the connected component containing
$U(X)$.
Then $\Stab^{\dagger}(X) \to {\cal P}^+_0(X)$ is a covering map.


For a Mukai vector $v$,
${\cal M}_\sigma(v)$ denotes the moduli stack of $\sigma$-semi-stable
objects $E$ with $v(E)=v$ and
$M_\sigma(v)$ denotes the coarse moduli scheme
of $S$-equivalence classes of $\sigma$-semi-stable
objects $E$ with $v(E)=v$.
If $\sigma=\sigma_{(\beta,\omega)}$, then
we set ${\cal M}_{(\beta,\omega)}(v):={\cal M}_{\sigma_{(\beta,\omega)}}(v)$
and $M_{(\beta,\omega)}(v):=M_{\sigma_{(\beta,\omega)}}(v)$.
\begin{rem}
For a $\sigma$-semi-stable object $E$ with $v(E)=v$,
$\phi_\sigma(E) \mod 2$ is determined by $Z_\sigma(v)$.
Although we need to fix $\phi_\sigma(E)$ for precise definitions,
we adopt the above definitions. 
\end{rem}

\begin{NB}
Let $\sigma_s$ be a family of stability conditions
such that $Z_{\sigma_s}(\bullet)=
\langle e^{(\beta+sH)+\sqrt{-1}H},\bullet \rangle$. 
Assume that $H$ is ample, ${\cal O}_x$ is $\sigma_0$-semi-stable
and ${\cal O}_x$ is $\sigma_s$-stable for $0>s \gg -1$.
Then 
$\sigma_0=({\cal A}_{(\beta,H)},Z_{(\beta,H)})$.

(Step 1).
We first prove that
every object of ${\cal P}_{\sigma_0}((0,1])$ is a two-term complex.

Let $E$ be a $\sigma_0$-stable object of 
${\cal P}_{\sigma_0}(\phi)$ with $0<\phi<1$.
Then $\Hom(E[i],{\cal O}_x)=0$ for $i>0$ by
$\phi_{\sigma_0}(E[i])=\phi+i>1$.
Hence $H^i(E)=0$ for $i>0$.
$\Hom(E[i],{\cal O}_x)=\Hom({\cal O}_x,E[i+2])^{\vee}=0$
for $i \leq -2$ by $\phi_{\sigma_0}(E[i+2])=\phi+i+2<1$.
Hence $E$ is a 2-term complex with $H^i(E)=0$ for $i \ne -1,0$.

Let $E$ be a $\sigma_0$-stable object of 
${\cal P}_{\sigma_0}(\phi)$ with $\phi=1$.

Assume that $\phi_{\sigma_s}(E)>1$ for $s<0$.
We set $F:=E[-1]$.
Then $H^i(F)=0$ for $i \ne -1,0$.
We set $v_i:=v(H^{-i}(F))=e^\beta(r_i+d_i H+D_i+a_i \varrho_X)$ $(i=0,1)$.
Since $\phi_{\sigma_0}(F)=0$, $d_0-d_1=0$.
By the description of $\sigma_s$ $(s<0)$,
$d_0-r_0 s \geq 0$ and $d_1-r_1 s \leq 0$ for $0>s \gg -1$.
\begin{NB2}
$d_0-r_0 s=0$ if $H^0(F)$ is 0-dimensional.
\end{NB2}
Hence $d_0 \geq 0$ and $d_1 \leq 0$, which implies $d_0=d_1=0$.
Then $d_1-r_1 s \leq 0$ and $s<0$ implies $r_1=0$.
Since $H$ is ample, $H^{-1}(F)$ is a 0-dimensional sheaf or 0.
By the description of ${\cal P}_{\sigma_s}((0,1])$,
$H^{-1}(F)=0$.

Assume that $\phi_{\sigma_s}(E)<1$ for $s<0$.
Then $H^i(E)=0$ for $i \ne -1,0$.
We set $v_i:=v(H^{-i}(E))=e^\beta(r_i+d_i H+D_i+a_i \varrho_X)$ $(i=0,1)$.
Since $\phi_{\sigma_0}(E)=0$, $d_0-d_1=0$.
By the description of $\sigma_s$ $(s<0)$,
$d_0-r_0 s \geq 0$ and $d_1-r_1 s \leq 0$ for $0>s \gg -1$.
Hence $d_0 \geq 0$ and $d_1 \leq 0$, which implies $d_0=d_1=0$.
Then $d_1-r_1 s \leq 0$ and $s<0$ implies $r_1=0$.
Since $H$ is ample, $H^{-1}(E)$ is a 0-dimensional sheaf or 0.
By the description of ${\cal P}_{\sigma_s}((0,1])$,
$H^{-1}(E)=0$.
Therefore the claim holds.

(Step 2).
Let $E$ be a coherent sheaf.
For $A \in {\cal P}_{\sigma_0}(>1)$,
$H^i(A)=0$ for $i \geq 0$ so that $\Hom(A,E)=0$.
For $B \in {\cal P}_{\sigma_0}(\leq -1)$,
$H^i(B)=0$ for $i \leq 0$ so that $\Hom(E,B)=0$.
Hence $E \in {\cal P}_{\sigma_0}((-1,1])$.
For $E \in \Coh(X)$, we have a triangle
$$
D \to E \to F \to D[1]
$$
such that $D \in  {\cal P}_{\sigma_0}((0,1])$ and
$F \in  {\cal P}_{\sigma_0}((-1,0])$.
Then $H^i(D)=0$ for $i \ne -1,0$ and
$H^i(F)=0$ for $i \ne 0,1$.
Taking their cohomology, 
we see that $D,F \in \Coh(X)$.

(Step 3).
Then ${\cal T}:={\cal P}_{\sigma_0}((0,1]) \cap \Coh(X)$ and
${\cal F}:={\cal P}_{\sigma_0}((-1,0]) \cap \Coh(X)$
is a torsion pair and ${\cal P}_{\sigma_0}((0,1])$ is the tilting.

(Step 4).
Let $E$ be a $\beta$-twisted stable sheaf.
Then we have an exact sequence
$$
0 \to E_1 \to E \to E_2 \to 0
$$ 
such that $E_1 \in {\cal T}$ and $E_2 \in {\cal F}$.
Then $Z_{\sigma_0}(E_1) \in {\Bbb H} \cup {\Bbb R}_{<0}$,
which implies $\deg_\beta(E_1)>0$ or $\deg_\beta(E_1)=0$ and
$\chi_\beta(E_1)>\frac{(H^2)}{2}\rk E_1$.
We also have 
 $-Z_{\sigma_0}(E_2) \in {\Bbb H} \cup {\Bbb R}_{<0}$,
which implies $\deg_\beta(E_2)<0$ or $\deg_\beta(E_2)=0$ and
$\chi_\beta(E_2)<\frac{(H^2)}{2}\rk E_2$. 
Then we see that $E_2$ is torsion free.
In particular, if $E$ is a torsion sheaf, then
$E \in {\cal T}$ .
Assume that $E$ is torsion free.
If $E_1$ and $E_2$ are not zero, then
we see that $E$ is not $\beta$-stable.
Therefore $E \in {\cal T}$ or $E \in {\cal F}$.
In particular, $\sigma_0=({\cal A}_{(\beta,H)},Z_{(\beta,H)})$.

A generalization: 

\end{NB}

\subsection{Stability conditions associated to a category of perverse 
coherent sheaves}\label{subsect:perverse}

Let us briefly recall our construction of $\sigma_{(\beta,\omega)}$
for a nef and big divisor $\omega$ in \cite{MYY:2011:1}.
Let $\pi:X \to Y$ be the minimal resolution of a normal K3 surface $Y$.
Let $H$ be the pull-back of an ample divisor
on $Y$ and
$(\beta \cdot C) \not \in {\Bbb Z}$
for all exceptional $(-2)$-curves.
Then there is a category of perverse coherent sheaves
${\frak C}$ with a local projective generator $G$ on $X$ 
such that $c_1(G)/\rk G=\beta$
\cite[Prop. 2.4.5]{PerverseI}:
Thus there is a locally free sheaf $G$ 
with $c_1(G)/\rk G=\beta$ such that
\begin{equation}\label{eq:ST} 
\begin{split}
S:= & \{ E \in \Coh(X) \mid \pi_*(G^{\vee} \otimes E)=0 \},\\
T:=& \{ E \in \Coh(X) \mid R^1 \pi_*(G^{\vee} \otimes E)=0 \}
\end{split}
\end{equation}
is a torsion pair $(T,S)$ of $\Coh(X)$ and the tilting 
is our category of perverse coherent sheaves:
\begin{equation}\label{eq:C:S-T}
{\frak C}=\{E \in {\bf D}(X) \mid 
H^i(E)=0, i \ne -1,0,\; H^{-1}(E) \in S, H^0(E) \in T \}.
\end{equation}
Here we would like to remark that
$R^1 \pi_*({\cal O}_X)=0$ and  
\cite[Assumption 1.1.1]{PerverseI} holds.

\begin{rem}\label{rem:ST}
Let ${\cal D}$ be a connected component
of $\NS(X)_{\Bbb R} \setminus \cup_{n,C}
\{x \mid (x \cdot C)=n \}$.
Then $(S,T)$ depends only on ${\cal D}$ containing $\beta$.
In particular, ${\frak C}$ is well-defined even when
$\beta$ is not defined over ${\Bbb Q}$.  
\end{rem}

\begin{rem}
If $\pi$ is an isomorphism, then
$H$ is ample and ${\frak C}=\Coh(X)$.

\end{rem}


\begin{defn}\label{defn:p^H}
\begin{enumerate}
\item[(1)]
For $E \in {\bf D}(X)$,
${^p H}^i(E) \in {\frak C}$ denotes the $i$-th cohomology object
of $E$ with respect to the $t$-structure defining ${\frak C}$.
\item[(2)]
For a morphism $\psi:E \to F$ in ${\frak C}$,
$\ker_{\frak C} \psi$, $\im_{\frak C} \psi$ and
$\coker_{\frak C} \psi$ denote the kernel, 
the image and the cokernel of $\psi$ in ${\frak C}$
respectively.
\end{enumerate}
\end{defn}

\begin{defn}\label{defn:dim}
For $E \in {\frak C}$,
we define the \emph{dimension} $\dim E$ of $E$ by
\begin{equation*}
\dim E:=\dim \pi(\Supp(H^{-1}(E)) \cup \Supp(H^0(E))).
\end{equation*}
\end{defn}

\begin{defn}
For $\beta \in \NS(X)_{\Bbb R}$ and $\omega 
\in \overline{\Amp(X)_{\Bbb R}}$ with $(\omega^2)>0$,
we set $\deg_\beta(E):=(c_1(E(-\beta)) \cdot \omega)$ and
$\chi_\beta(E):=\chi(E(-\beta))$.
We also set
$a_\beta(E):=-\langle e^\beta,v(E) \rangle$.
\end{defn}

For a local projective generator $G$ of ${\frak C}$ and
a perverse coherent sheaf $E \in {\frak C}$,
we have the $G$-twisted Hilbert polynomial $\chi(G,E(nH))$
with respect to $H$.
The degree of the $G$-twisted Hilbert polynomial
of $E$ is $\dim E$.
By using the $G$-twisted Hilbert polynomial, 
we have a notion of semi-stability as in the 
Gieseker semi-stability for ordinary coherent sheaves.

\begin{defn}[{\cite[Defn. 1.4.1]{PerverseI}}]
\label{defn:stability}
Let ${\frak C}$ be a category of perverse coherent sheaves on $X$, 
$G$ a local projective generator of ${\frak C}$ ,
and $E \in {\frak C}$ a perverse coherent sheaf.  
\begin{enumerate}
\item 
\begin{enumerate}
\item
If $\dim E \leq 1$, then $E$ is called a \emph{torsion object}.

\item
If there is no subobject $F \ne 0$ with $\dim F<d=\dim E$,
then $E$ is called \emph{purely $d$-dimensional}.
In particular, if $E$ is purely $2$-dimensional, 
then $E$ is called \emph{torsion free}. 
\end{enumerate}
\item 
A $2$-dimensional object $E$ is 
\emph{$G$-twisted semi-stable with respect to $H$} if
\begin{equation}\label{eq:G-semi-stability:2dim}
\chi(G, F(nH))
 \leq \frac{\rk F}{\rk E}
\chi(G,E(nH)),\quad n \gg 0
\end{equation} 
for all proper subobjects $F \ne 0$ of $E$. 
We also say a torsion free object $E$ 
is \emph{$\mu$-semi-stable} if
\begin{equation}\label{eq:mu-semi-stability:2dim}
\frac{(c_1(F),H)}{\rk F} \leq \frac{(c_1(E),H)}{\rk E}
\end{equation} 
for all subobjects $F$ of $E$ with
$0<\rk F <\rk E$. 

\item 
If $E$ is 1-dimensional, then 
$E$ is \emph{$G$-twisted semi-stable with respect to $H$} if
\begin{equation}\label{eq:G-semi-stability:1dim}
\chi(G,F) \leq \frac{(H,c_1(F))}{(H,c_1(E))} \chi(G,E)
\end{equation} 
for all proper subobjects $F \ne 0$ of $E$.
\item
Let ${\cal M}_H^\gamma(v)$ denote the moduli stack of $\gamma$-twisted
semi-stable objects $E$ with $v(E)=v$, and
$M_H^\gamma(v)$ the coarse moduli scheme
of $S$-equivalence classes 
of $\gamma$-twisted semi-stable objects \cite[sect. 1.4]{PerverseI}.
\end{enumerate}
\end{defn}

\begin{NB}
In \cite[\S1.5]{PerverseI}, a refined notion of 
semi-stability for 0-dimensional objects is introduced:
for $\gamma \in \NS(X)_{\Bbb Q}$,
a 0-dimensional object $E$ with $v(E)=\varrho_X$ is 
$\gamma$-semi-stable if $-(c_1(F),\gamma) \leq 0$ for 
all subobject $F$ of $E$.
See Definition \ref{defn:mukai-lattice} given below for 
the notation $v(E)$ and $\varrho_X$.
\end{NB}

\begin{defn}
\begin{enumerate}
\item[(1)]
For $(\beta,\omega) \NS(X)_{\Bbb R} \times P^+(X)_{\Bbb R}$ and
$E \in {\bf D}(X)$, we set
$$
Z_{(\beta,\omega)}(E):=\langle e^{\beta+i \omega},v(E) \rangle.
$$
\item[(2)]
Under the assumption 
$e^{\beta+i \omega} \not \in \cup_{u \in \Delta(X)}u^\perp$ with
$(\beta,\omega)\in \NS(X)_{\Bbb Q} \times \pi^*(\Amp(Y)_{\Bbb Q})$,
we define a torsion pair 
$({\cal T}_{(\beta,\omega)},{\cal F}_{(\beta,\omega)})$
of ${\frak C}$ as follows:
\begin{enumerate}
\item
${\cal T}_{(\beta,\omega)}$ 
is the full subcategory of
${\frak C}$ consisting of $E$ such that
$Z_{(\beta,\omega)}(F) \in {\Bbb H} \cup {\Bbb R}_{<0}$
for any quotient $E \to F (\ne 0)$ of $E$.
\item
${\cal F}_{(\beta,\omega)}$ is the full subcategory of
${\frak C}$ 
consisting of $E$ such that
$-Z_{(\beta,\omega)}(F) \in {\Bbb H} \cup {\Bbb R}_{<0}$
for any subobject $(0 \ne) F \to E$ of $E$.
\end{enumerate}
Let ${\cal A}_{(\beta,\omega)}$ be
the tilting of the torsion pair 
$({\cal T}_{(\beta,\omega)},{\cal F}_{(\beta,\omega)})$.
\end{enumerate}
\end{defn}

The definition of $({\cal T}_{(\beta,\omega)},{\cal F}_{(\beta,\omega)})$
is equivalent to the definition in \cite[Defn. 1.5.7]{MYY:2011:1}.

For $(\beta,\omega) \in 
\NS(X)_{\Bbb R} \times \pi^*(\Amp(Y)_{\Bbb R})$,
we also define a pair of subcategories
$({\cal T}_{(\beta,\omega)}^*,{\cal F}_{(\beta,\omega)}^*)$
of $\Coh(X)$ as follows:
\begin{enumerate}
\item
${\cal T}_{(\beta,\omega)}^*$ is the full subcategory of
$\Coh(X)$ consisting of $E$ such that
$Z_{(\beta,\omega)}(F) \in {\Bbb H} \cup {\Bbb R}_{<0}$
for any quotient $E \to F (\ne 0)$ of $E$.
\item
${\cal F}_{(\beta,\omega)}^*$ is the full subcategory of
$\Coh(X)$ consisting of $E$ such that
$-Z_{(\beta,\omega)}(F) \in {\Bbb H} \cup {\Bbb R}_{<0}$
for any subsheaf $(0 \ne) F \to E$ of $E$.
\end{enumerate}
 
The relation of these definitions are given by the following 
proposition.
\begin{prop}\label{prop:A}
Assume that $\beta$ and $\omega$ are defined over ${\Bbb Q}$.
\begin{enumerate}
\item[(1)]
For $E \in {\cal A}_{(\beta,\omega)}$,
$H^i(E)=0,\; (i \ne -1,0)$.
\item[(2)]
For $E \in \Coh(X)$, there is an exact sequence
\begin{equation}\label{eq:T-F}
0 \to E_1 \to E \to E_2 \to 0
\end{equation}
such that $E_1 \in {\cal T}:={\cal A}_{(\beta,\omega)} \cap \Coh(X)$
and $E_2 \in {\cal F}:={\cal A}_{(\beta,\omega)}[-1] \cap \Coh(X)$.
Thus $({\cal T},{\cal F})$ is a torsion pair of $\Coh(X)$.
\item[(3)]
$({\cal T},{\cal F})=
({\cal T}_{(\beta,\omega)}^*,{\cal F}_{(\beta,\omega)}^*)$.
In particular, 
$({\cal T}_{(\beta,\omega)}^*,{\cal F}_{(\beta,\omega)}^*)$ is a torsion pair
whose tilting is ${\cal A}_{(\beta,\omega)}$.
\end{enumerate}
%
\end{prop}

\begin{proof}
Let $(T,S)$ be the torsion pair in \eqref{eq:ST}.
(1)
For $E \in {\cal A}_{(\beta,\omega)}$,
we see that $\pH^{-1}(E)$ is a torsion free object of 
${\frak C}$, which implies
$H^{-1}(\pH^{-1}(E))=0$.
Then 
$H^i(E)=0$ for $i \ne -1,0$, $H^0(\pH^0(E))=H^0(E)$
and we have an exact sequence in $\Coh(X)$
$$
0 \to \pH^{-1}(E) \to H^{-1}(E) \to H^{-1}(\pH^0(E)) \to 0.
$$

(2)
For $E \in \Coh(X)$, we have a decomposition
$$
0 \to E_T \to E \to E_S \to 0
$$
such that $E_T \in T$ and $E_S \in S$.
Since $E_S[1] \in {\frak C}$ is a 0-dimensional object,
$E_S[1] \in {\cal T}_{(\beta,\omega)} \subset {\cal A}_{(\beta,\omega)}$. 
We also have a decomposition
$$
0 \to E_1 \to E_T \to E_1' \to 0
$$
such that $E_1 \in {\cal T}_{(\beta,\omega)}$ and 
$E_1' \in {\cal F}_{(\beta,\omega)}$.
Then $E_2:=\mathrm{Cone}(E_1 \to E) \in {\cal A}_{(\beta,\omega)}[-1]$.
By (1), we see that $E_1,E_2 \in \Coh(X)$.
Hence we have a decomposition \eqref{eq:T-F}.
Since $\Hom(A,B)=0$ for $A \in {\cal T}$ and $B \in {\cal F}$,
$({\cal T},{\cal F})$ is a torsion pair.

(3)
We first prove that
${\cal T}={\cal T}_{(\beta,\omega)}^*$.
We note that  ${\cal T}_{(\beta,\omega)} \cap \Coh(X)={\cal T}$.
Let $E$ be an element of ${\cal T}_{(\beta,\omega)}^*$. 
For $F \in S$ and a non-zero homomorphism
$\psi:E \to F$, $\im \psi \in S$ and
$Z_{(\beta,\omega)}(\im \psi)=-\chi_\beta(\im \psi)> 0$.
Hence $\Hom(E,F)=0$ for $F \in S$ and $E \in T$.
Let $\psi:E \to F (\ne 0)$ be a quotient in ${\frak C}$.
Then we have an exact sequence in $\Coh(X)$
$$
0 \to H^{-1}(F) \to H^0(\ker_{\frak C} \psi) \to H^0(E) \to H^0(F) \to 0.
$$
Since $Z_{(\beta,\omega)}(H^0(F)) \in {\Bbb H} \cup {\Bbb R}_{<0}$
or $H^0(F)=0$,
and $-Z_{(\beta,\omega)}(H^{-1}(F)) \in {\Bbb R}_{< 0}$ or $H^{-1}(F)=0$,
we have $Z_{(\beta,\omega)}(F) \in {\Bbb H} \cup {\Bbb R}_{<0}$.
Hence $E \in {\cal T}_{(\beta,\omega)} \cap \Coh(X)={\cal T}$.
Conversely if $E \in {\cal T}=\Coh(X) \cap {\cal T}_{(\beta,\omega)}$, then
for any quotient $\psi:E \to F (\ne 0)$ in $\Coh(X)$,
$\ker \psi$ in $\Coh(X)$ has a decomposition
$$
0 \to E_1 \to \ker \psi \to E_2 \to 0
$$
such that $E_1 \in T$ and $E_2 \in S$.
Then $E_1=\ker_{\frak C} \psi$,
$E/E_1=\im_{\frak C}\psi$
and $\coker_{\frak C} \psi=E_2[1]$.
Hence $Z_{(\beta,\omega)}(F)=Z_{(\beta,\omega)}(\im_{\frak C}\psi)
-Z_{(\beta,\omega)}(E_2) \in {\Bbb H} \cup {\Bbb R}_{<0}$.
Therefore $E \in {\cal T}_{(\beta,\omega)}^*$, and we get
${\cal T}={\cal T}_{(\beta,\omega)}^*$.

We next prove that ${\cal F}={\cal F}_{(\beta,\omega)}^*$
Let $E$ be an element of ${\cal F}_{(\beta,\omega)}^*$. 
We take a decomposition
$$
0 \to E_T \to E \to E_S \to 0
$$
such that $E_T \in T$ and $E_S \in S$.
Then $E_S[1] \in {\cal T}_{(\beta,\omega)}$.
For a non-zero subobject $\psi:F \to E_T$ of $E_T$ in ${\frak C}$,
$F \in \Coh(X)$ and
we have an exact sequence
$$
0 \to H^{-1}(\coker_{\frak C}\psi) \to H^0(F) \overset{\varphi}{\to} H^0(E_T) 
\to H^0(\coker_{\frak C}\psi) \to 0,
$$
where $H^{-1}(\coker_{\frak C}\psi)  \in S$.
Since $E_T \in {\cal F}_{(\beta,\omega)}^*$,
$-Z_{(\beta,\omega)}(\im \varphi) \in {\Bbb H} \cup {\Bbb R}_{<0}$.
Hence $-Z_{(\beta,\omega)}(H^0(F)) \in {\Bbb H} \cup {\Bbb R}_{<0}$,
which implies $E_T \in {\cal F}_{(\beta,\omega)}$.
Hence $E \in {\cal A}_{(\beta,\omega)}[-1] \cap \Coh(X)={\cal F}$.
For $E \in  {\cal A}_{(\beta,\omega)}[-1] \cap \Coh(X)={\cal F}$,
we have a decomposition
$$
0 \to E_T \to E \to E_S \to 0
$$
such that $E_T \in T$ and $E_S \in S$.
Obviously $S \subset {\cal F}_{(\beta,\omega)}^*$.
Since $\pH^0(E) \in {\cal F}_{(\beta,\omega)}$ and
$\pH^0(E)=E_T$, $E_T \in {\cal F}_{(\beta,\omega)}$.
For a subsheaf $F$ of $E$,
we take a decomposition
$$
0 \to F_T \to F \to F_S \to 0
$$
such that $F_T \in T$ and $F_S \in S$.
Then $F_T$ is a subsheaf of $E_T$.
Moreover $F_T \to E_T$ is injective in ${\frak C}$
by $E_T/F_T \in T$.
Since $E_T \in {\cal F}_{(\beta,\omega)}$,
$-Z_{(\beta,\omega)}(F_T) \in {\Bbb H} \cup {\Bbb R}_{<0}$
unless $F_T=0$.
We note that $-Z_{(\beta,\omega)}(F_S) \in {\Bbb R}_{<0}$
unless $F_S=0$.
Hence $-Z_{(\beta,\omega)}(F) \in {\Bbb H} \cup {\Bbb R}_{<0}$.
Thus $E \in {\cal F}_{(\beta,\omega)}^*$, which implies
${\cal F}={\cal F}_{(\beta,\omega)}^*$.
\end{proof}

\begin{NB}
\begin{enumerate}
\item
${\cal T}_{(\beta,\omega)}$ is the full subcategory of
${\frak C}$ generated by 
torsion objects and $\beta$-twisted stable objects $E$ of ${\frak C}$
such that $Z_{(\beta,\omega)}(E) \in {\Bbb H} \cup {\Bbb R}_{<0}$.
\item
${\cal F}_{(\beta,\omega)}$ is the full subcategory of
${\frak C}$ generated by $\beta$-twisted stable objects $E$ of ${\frak C}$
such that $-Z_{(\beta,\omega)}(E) \in {\Bbb H} \cup {\Bbb R}_{<0}$.
\end{enumerate}
\end{NB}

\begin{prop}[{\cite{MYY:2011:1}}]
$\sigma_{(\beta,\omega)}:=({\cal A}_{(\beta,\omega)},Z_{(\beta,\omega)})$
is a stability condition.
\end{prop}

\begin{defn}\label{defn:A}
If $({\cal T}_{(\beta,\omega)}^*,{\cal F}_{(\beta,\omega)}^*)$ is a
torsion pair of $\Coh(X)$, then 
we also denote the tilting by ${\cal A}_{(\beta,\omega)}$
and set 
$\sigma_{(\beta,\omega)}:=({\cal A}_{(\beta,\omega)},Z_{(\beta,\omega)})$.
\end{defn}
By Proposition \ref{prop:A},
$\sigma_{(\beta,\omega)}$ is the same stability condition
in \cite{MYY:2011:1} if $\beta,\omega$ are defined over ${\Bbb Q}$.

\section{Stability conditions on the boundary of
the geometric chamber}\label{sect:MYY}

\subsection{Relation of stability conditions}\label{subsect:relation}

Let $\pi:X \to Y$ be the minimal resolution of 
a normal K3 surface $Y$.
For $(\beta_0,\omega_0):=(\beta_0,tH)$,
$\sigma_{(\beta_0,\omega_0)}:=
({\cal A}_{(\beta_0,\omega_0)},Z_{(\beta_0,\omega_0)})$
denotes the stability condition constructed in 
subsection \ref{subsect:perverse}.
For these stability conditions,
${\cal O}_x$ are semi-stable objects with phase 1.
We shall remark the relation of these stability conditions
to stability conditions such that ${\cal O}_x$ are stable.

We note that $\sigma_{(\beta_0,\omega_0)}$ satisfies the support 
property \cite{BM},
since the Bogomolov inequality holds.
In particular, we have a wall and chamber structure.
There is a small neighborhood $B$ of $(\beta_0,\omega_0)$
in $\NS(X)_{\Bbb R} \times P^+(X)_{\Bbb R}$
and a continuous map 
\begin{equation}\label{eq:section}
{\frak s}:B \to \Stab(X)
\end{equation}
such that $Z_{{\frak s}(\beta, \omega)}=Z_{(\beta,\omega)}$
and ${\frak s}(\beta_0,\omega_0)=\sigma_{(\beta_0,\omega_0)}$
(see also \cite[Prop. 8.3]{Br:3}).

Let $S$ be the set of stable factor $E (\ne {\cal O}_x)$ of ${\cal O}_x$
with respect to $\sigma \in {\frak s}(B)$.
Since $\{ {\cal O}_x \mid x \in X \}$ has bounded mass, 
\cite[Lem. 9.2]{Br:3} implies
$W:=\{ v(E) \mid E \in S\}$ is a finite set.
Let $S'$ be a subset of $S$ such that
$E_1 \in S'$ if and only if 
\begin{enumerate}
\item
$E_1$ is a stable factor of ${\cal O}_x$ with respect to 
$\sigma \in {\frak s}(B)$
and 
\item
$\phi_\sigma(E_1)=\phi_\sigma^+({\cal O}_x)$.
\end{enumerate}
%
We set
$W':=\{v(E) \in W \mid E \in S',\rk E>0 \}$.
\begin{NB}
If $W = \emptyset$, then stability is independent
in $B$.
\end{NB}

\begin{lem}
Assume that $Z_{\sigma_{(\beta_0,\omega_0)}}(E) \in {\Bbb R}_{<0}$ for all
$E \in S'$.
We set 
$$
B':=\{(\beta,\omega) \in B \mid \omega \in \Amp(X),
\Ima Z_{(\beta,\omega)}(w)>0, w \in W' \}.
$$
\begin{enumerate}
\item[(1)]
Then  $B' \ne \emptyset$
and $(\beta_0,\omega_0) \in \overline{B'}$.
\item[(2)]
 Let 
\begin{equation}
0 \to E_1 \to {\cal O}_x \to E_2 \to 0
\end{equation}
be an exact sequence in ${\cal A}_{(\beta_0,\omega_0)}$ such that
$E_1 \in S'$. Then 
$\Ima Z_{(\beta,\omega)}(E_1)>0$ for $(\beta,\omega)  \in B'$.
\end{enumerate}
\end{lem}

\begin{proof}
Let $w=r+\xi+a \varrho_X$ be an element of $W'$.
For $x, y \in {\Bbb R}_{>0}$ and
$\eta \in \omega_0^\perp$ with
$(\xi-r \beta_0,\eta)+rxy(\omega_0^2)>0$, 
$\Ima Z_{(\beta_0-x\omega_0,y\omega_0+\eta)}(w)>0$.
We can take a small $\eta$ such that $y \omega_0+\eta$ is ample.  
Therefore (1) holds.
\begin{NB}
$(\xi-r \beta_0 \cdot \eta)+rxy(\omega_0^2) \geq 
-\sqrt{-((\xi-r \beta_0)^2)}\sqrt{-(\eta^2)}+rxy(\omega^2)$.
We set $N_1:=\max_w \sqrt{-((\xi-r \beta_0)^2)}$
and assume that $\sqrt{-(\eta^2)} \leq \epsilon/N_1$.
Then $xy(\omega^2)>\epsilon$ is sufficient.
\end{NB}

(2) We note that  $\pH^{-1}(E_1)=0$ and 
there is an exact sequence in ${\frak C}$
\begin{equation}
0 \to \pH^{-1}(E_2) \to \pH^0(E_1) 
\overset{\psi}{\to} {\cal O}_x \to \pH^0(E_2) \to 0.
\end{equation}
Hence $\rk E_1 \geq 0$.
If $\rk E_1>0$, then $v(E_1) \in W'$ implies
$\Ima Z_\sigma(E_1)>0$ for $\sigma \in {\frak s}(B')$.
Assume that $\rk E_1=0$.
Since $E_1,E_2$ are semi-stable objects with
$\phi_{\sigma_{(\beta_0,\omega_0)}}(E_1)=\phi_{\sigma_{(\beta_0,\omega_0)}}(E_2)$, 
$E_1$ is a 0-dimensional object of ${\frak C}$.
Since $\pH^{-1}(E_2)$ is torsion free in ${\frak C}$, 
$\pH^{-1}(E_2)=0$.
Since $H^{-1}({\cal O}_x)=0$, we get $H^{-1}(\pH^0(E_1))=H^{-1}(E_1)=0$.
Thus $c_1(E_1)$ is effective or $E_1={\cal O}_x$. 
Since $E_1 \in S$, the second case does not occur.
Then $\Ima Z_\sigma(E_1)=(c_1(E_1) \cdot \omega)>0$.  
Therefore the claim holds.
\end{proof}

\begin{prop}\label{prop:limit}
Let $B_0'$ be a connected component of $B'$ such that
$(\beta_0,\omega_0) \in \overline{B_0'}$. 
Then ${\cal O}_x$ $(x \in X)$ is $\sigma$-stable for all
$\sigma \in {\frak s}(B_0')$.
In particular, 
${\frak s}(\beta,\omega)=\sigma_{(\beta,\omega)}$
for $(\beta,\omega) \in B_0'$.
\end{prop}

\begin{proof}
We set $\sigma:={\frak s}(\beta, \omega)$ ($(\beta,\omega) \in B_0'$)
and
$\sigma_0:={\frak s}(\beta_0, \omega_0)$.
By shrinking $B$, we may assume that
$$
\sup_{0 \ne E \in {\bf D}(X)}
|\phi_{\sigma'}^\pm(E)-\phi_{\sigma_0}^\pm(E)|<\frac{1}{8}
$$
for $\sigma' \in {\frak s}(B)$.
Then 
$$
\sup_{0 \ne E \in {\bf D}(X)}
|\phi_{\sigma}^\pm(E)-\phi_{\sigma'}^\pm(E)|<\frac{1}{4}
$$
for $\sigma, \sigma' \in {\frak s}(B)$.
Since ${\cal O}_x$ are $\sigma_0$-semi-stable,
$1+\frac{1}{8}>\phi_\sigma^+({\cal O}_x) \geq 
\phi_\sigma^-({\cal O}_x)>1-\frac{1}{8}$.
For any stable factor $E$ of ${\cal O}_x$ with respect to $\sigma$,
$1+\frac{1}{8}+\frac{1}{4}>\phi_{\sigma'}^+(E) 
\geq \phi_{\sigma'}^-(E)>1-\frac{1}{8}-\frac{1}{4}$.
We set ${\cal A}_{\sigma'}:={\cal P}_{\sigma'}((\tfrac{1}{2},\tfrac{3}{2}])
(\subset {\bf D}(X))$.
If ${\cal O}_x$ is not $\sigma$-semi-stable,
then let $E_1$ be the stable factor with 
$\phi_\sigma(E_1)=\phi_\sigma^+({\cal O}_x)
>\phi_\sigma({\cal O}_x)$
and $E_2:=\mathrm{Cone}(E_1 \to {\cal O}_x)$.
Then we have $E_1,E_2 \in {\cal A}_{\sigma'}$ 
for all $\sigma' \in {\frak s}(B)$ 
and an exact sequence 
\begin{equation}\label{eq:O_x}
0 \to E_1 \to {\cal O}_x \to E_2 \to 0
\end{equation}
in ${\cal A}_{\sigma'}$.
Since ${\cal O}_x$ is $\sigma_0$-semi-stable with 
$\phi_{\sigma_0}({\cal O}_x)=1$,
$\phi_{\sigma_0}(E_1) \leq \phi_{\sigma_0}({\cal O}_x)=1$.
If $\phi_{\sigma_0}(E_1)<1$, then
$$
\{\sigma' \in {\frak s}(B) \mid \phi_{\sigma'}(E_1)<1\}
$$
is an open neighborhood of $\sigma_0$ which does not contain
$\sigma$, where
$\phi_{\sigma'}:{\cal A}_{\sigma'} \to (\tfrac{1}{2},\tfrac{3}{2}]$.
So by shrinking $B$, we may assume that
$\phi_{\sigma_0}(E_1)=1$ for all $E_1 \in S'$.
Since 
$\phi_{\sigma_0}(E_2)=\phi_{\sigma_0}(E_1)=1$,
\eqref{eq:O_x} is an exact sequence
in ${\cal A}_{(\beta_0,\omega_0)}$.
If $\sigma \in {\frak s}(B_0')$, then we get
$\mathrm{Im} Z_{\sigma}(E_1)>0> \mathrm{Im} 
Z_{\sigma}(E_2)$, which implies
$\phi_{\sigma}(E_1)<1<\phi_{\sigma}(E_2)$.
Therefore ${\cal O}_x$ is $\sigma$-stable.
In particular ${\frak s}(\beta,\omega)=\sigma_{(\beta,\omega)}$
for $(\beta,\omega) \in B_0'$ by \cite[Prop. 10.3]{Br:3}.
\end{proof}

\begin{NB}
Assume that there is a $\sigma$-stable object $E$ of 
${\cal O}_x$ with $\phi_\sigma(E)>\phi_\sigma({\cal O}_x)$.
Then it is a subobject in ${\cal A}:=
{\cal P}_{\sigma_0}((\frac{1}{2},\frac{3}{2}])$.
By our choice of $W$,
$\phi_{\sigma_0}(E)=\phi_{\sigma_0}({\cal O}_x)$.  
Then $E$ is $\sigma_0$-semi-stable and 
$E$ is a subobject of ${\cal O}_x$
in ${\cal A}_{\sigma_0}$.
\end{NB}

\begin{NB}
\begin{lem}\label{lem:perv-H^0}
Assume that every irreducible object $A$ of ${\frak C}$ is
$\sigma$-stable with $\phi_\sigma(A)=1$.
Then $\pH^i(E)=0$ ($i \ne -1,0$) for
$E \in {\cal P}_\sigma((0,1])$. 
\end{lem}

\begin{proof}

Let $E$ be a $\sigma_0$-stable object of 
${\cal P}_{\sigma_0}(\phi)$ with $0<\phi<1$.
Then $\Hom(E[i],A)=0$ for $i>0$ and all irreducible objects $A \in {\frak C}$
by
$\phi_{\sigma_0}(E[i])=\phi+i>1$.
Hence $H^i(E)=0$ for $i>0$.
$\Hom(E[i],A)=\Hom(A,E[i+2])^{\vee}=0$
for $i \leq -2$ by $\phi_{\sigma_0}(E[i+2])=\phi+i+2<1$.
Hence $E$ is a 2-term complex with $H^i(E)=0$ for $i \ne -1,0$.
Moreover $\Hom(E[1],A)=0$ implies $H^0(E) \in {\frak C}$.
\begin{NB2}
For $A=S[1]$, $S \in \Coh(X) \cap {\frak C}[-1]$,
 $\Hom(E,S)=\Hom(H^0(E),S)=0$, which shows
$H^0(E) \in {\frak C}$. 
\end{NB2}
Therefore $\pH^i(E)=0$ for $i \ne -1,0$.

Let $E$ be a $\sigma_0$-stable object of ${\cal P}_{\sigma_0}(1)$.
If $E \ne A$ for all irreducible objects $A$ of ${\frak C}$, then
$\Hom(E,A)=\Hom(A,E)=0$.
Hence $\pH^i(E) =0$ for $i \ne -1$ and
$\pH^{-1}(E)$ is a local projective object by 
\cite[Prop. 1.1.33]{PerverseI}.
\end{proof}
\end{NB}

In the proof of \cite[Lem. 11.1]{Br:3}, the following claim is proved.
\begin{lem}\label{lem:Delta}
For a bounded set $B$ of $\NS(X)_{\Bbb R} \times P^+(X)_{\Bbb R}$,
\begin{equation}\label{eq:Delta-B} 
\Delta_B:=\{ u \in \Delta(X) \mid
\rk u>0,
Z_{(\beta,\omega)}(u) \in {\Bbb R}_{ \leq 0},(\beta,\omega) \in B \}
\end{equation}
is a finite set. 
\end{lem}

\begin{prop}\label{prop:limit2}
Let $\sigma_s$ $(s \geq 0)$ be a family of stability conditions
such that $Z_{\sigma_s}(\bullet)=
\langle e^{\beta_s+\sqrt{-1}\omega_s},\bullet \rangle$
and ${\cal O}_x$ is $\sigma_s$-stable for $s>0$.
Assume that $\beta_0 \in \NS(X)_{\Bbb Q}$,
$\omega_0 \in {\Bbb R}_{>0}H$, $H \in \NS(X)$ and
\begin{equation}\label{eq:Delta} 
\{ u \in \Delta(X) \mid
\rk u>0,
Z_{(\beta_0,\omega_0)}(u) \in {\Bbb R}_{ \leq 0} \} =\emptyset.
\end{equation}
Then 
$\sigma_0=\sigma_{(\beta_0,\omega_0)}=
({\cal A}_{(\beta_0,\omega_0)},Z_{(\beta_0,\omega_0)})$.
\end{prop}

\begin{proof}
Although the claim follows from Proposition \ref{prop:limit}
and the covering property of $\Stab^{\dagger}(X) \to 
{\cal P}_0^+(X)$,
we shall give a more direct argument.

(Step 1)
We note that
$\omega_s$ is ample for $s>0$ and $\omega_0$ is nef and big.
${\cal O}_x$ is $\sigma_0$-semi-stable.
We set $\phi_s:=\phi_{\sigma_s}$. 
Let $E$ be a $\sigma_0$-stable object with
$0<\phi_0<1$.
Then $H^i(E)=0$ for $i \ne -1,0$, since
$E$ is a $\sigma_s$-stable object of 
$0<\phi_s(E)<1$ for a small $s>0$ (\cite[Prop. 10.3]{Br:3}).
Let $E$ be a $\sigma_0$-stable object of 
$\phi_0(E)=1$.
Assume that $\phi_s(E)>1$ for $s>0$ and set $F:=E[-1]$.
Since $F \in {\cal P}_{\sigma_s}((0,1])$ for all small $s>0$,
$H^i(F)=0$ for $i \ne -1,0$ and $H^{-1}(F)$ is torsion free
(\cite[Prop. 10.3]{Br:3}).
Assume that $H^{-1}(F) \ne 0$, that is, $\rk H^{-1}(F)>0$.
Since $H^{-1}(F)[1] \in {\cal P}_{\sigma_s}((0,1])$,
$H^{-1}(F)[1] \in {\cal P}_{\sigma_0}([0,1])$.
We also have $H^0(F) \in   {\cal P}_{\sigma_0}([0,1])$.
Since $\phi_0(F)=0$,
the exact triangle 
$$
H^0(F)[-1] \to  H^{-1}(F)[1] \to F \to H^0(F)
$$
implies $H^{-1}(F)[1] \in {\cal P}_{\sigma_0}(0)$.
By the stability of $F$,
$H^0(F)[-1]  \in  {\cal P}_{\sigma_0}(0)$.
Thus $H^0(F) \in  {\cal P}_{\sigma_0}(1)$.
Then $H^{-1}(F) \in  {\cal P}_{\sigma_0}(-1)$.
In particular, 
$Z_{(\beta_0,\omega_0)}(H^{-1}(F))=
\rk H^{-1}(F) \frac{(\omega^2_0)}{2}-a_{\beta_0}(H^{-1}(F))<0$. 
Let $F_1$ be a subsheaf of $H^{-1}(F)$. Then
$(c_1(F_1(-\beta_0)) \cdot \omega_0) >0$
implies $(c_1(F_1(-\beta_s)) \cdot \omega_s) >0$
for $1 \gg s>0$, which contradicts with
the description of ${\cal P}_{\sigma_s}((0,1])$.
Therefore $(c_1(F_1(-\beta_0)) \cdot \omega_0)\leq 0$, which implies 
$H^{-1}(F)$ is $\mu$-semi-stable with respect to
$\omega_0$.
By Lemma \ref{lem:Delta2} below,
we conclude that $H^{-1}(F)=0$.
In particular, $H^i(E)=0$ for $i \ne -1$.
If $\phi_s(E) \leq 1$, then obviously $H^i(E)=0$ for $i \ne -1,0$.
Therefore $H^i(E)=0$ for $i \ne -1,0$.

(Step 2)
We take $E \in \Coh(X)$.
  For $A \in {\cal P}_{\sigma_0}(>1)$,
$H^i(A)=0$ for $i \geq 0$ so that $\Hom(A,E)=0$.
For $B \in {\cal P}_{\sigma_0}(\leq -1)$,
$H^i(B)=0$ for $i \leq 0$ so that $\Hom(E,B)=0$.
Hence $E \in {\cal P}_{\sigma_0}((-1,1])$.
For $E \in \Coh(X)$, we have a triangle
$$
D \to E \to F \to D[1]
$$
such that $D \in  {\cal P}_{\sigma_0}((0,1])$ and
$F \in  {\cal P}_{\sigma_0}((-1,0])$.
Then $H^i(D)=0$ for $i \ne -1,0$ and
$H^i(F)=0$ for $i \ne 0,1$.
Taking their cohomology, 
we see that $D,F \in \Coh(X)$.
We set 
\begin{equation}
{\cal T}:= {\cal P}_{\sigma_0}((0,1]) \cap \Coh(X),\;
{\cal F}:={\cal P}_{\sigma_0}((-1,0]) \cap \Coh(X).
\end{equation}
Then $({\cal T},{\cal F})$ is a torsion pair.
We show that the tilting is
${\cal P}_{\sigma_0}((0,1])$.
For $E \in {\cal P}_{\sigma_0}((0,1])$,
$H^{-1}(E) \in {\cal F}$ and $H^0(E) \in {\cal T}$.
Indeed for $F \in {\cal F}$,
$\Hom(E,F)=\Hom(H^{-1}(E)[2],F)=0$ implies $\Hom(H^0(E),F)=0$,
which shows $H^0(E) \in {\cal T}$.
For $T \in {\cal T}$, $\Hom(T[1],E)=\Hom(T[1],H^0(E)[-1])=0$
implies $\Hom(T,H^{-1}(E))=0$, which shows $H^{-1}(E) \in {\cal F}$.

\begin{NB}
We show that $H^0(E) \in {\frak C}$.
Let $H^0(E)\to F$ be a quotient sheaf such that
$F[1] \in {\frak C}$.
Assume that $F \ne 0$.
We note that $F$ is a torsion sheaf such that $(c_1(F),\omega_0)=0$
and $a_\beta(F)=\chi_\beta(F)<0$, and hence
$Z_{\sigma_0}(F) \in {\Bbb R}_{>0}$.
Since $({\cal T},{\cal F})$ is a torsion pair,
we get $S \in {\cal T}$, which implies that
$Z_{\sigma_0}(F) \in {\Bbb R}_{<0}$.
Therefore $H^0(E) \in {\frak C}$ and
$\pH^i(E)=0$ for $i \ne -1,0$.
\end{NB}

(Step 3)
Finally we shall prove that $({\cal T},{\cal F})=
({\cal T}_{(\beta_0,\omega_0)}^*,{\cal F}_{(\beta_0,\omega_0)}^*)$.
We note that 
$Z_{(\beta,\omega)}(F) \in {\Bbb H} \cup {\Bbb R}_{<0}$ for any 
$0 \ne F \in {\cal T}$.
Let $E$ be an element of ${\cal T}$.
Since $F \in {\cal T}$ for any quotient sheaf $F$ of $E$, 
we have $E \in  {\cal T}_{(\beta_0,\omega_0)}^*$.
Thus ${\cal T} \subset {\cal T}_{(\beta_0,\omega_0)}^*$.
We also have ${\cal F} \subset {\cal F}_{(\beta_0,\omega_0)}^*$.
Since $({\cal T},{\cal F})$ is a torsion pair,
the definition of 
${\cal T}_{(\beta_0,\omega_0)}^*,{\cal F}_{(\beta_0,\omega_0)}^*$
implies 
$({\cal T},{\cal F})=
({\cal T}_{(\beta_0,\omega_0)}^*,{\cal F}_{(\beta_0,\omega_0)}^*)$.
Hence 
$({\cal T}_{(\beta_0,\omega_0)}^*,{\cal F}_{(\beta_0,\omega_0)}^*)$ is a torsion pair of 
$\Coh(X)$ and
${\cal P}_{\sigma_0}((0,1])={\cal A}_{(\beta_0,\omega_0)}$
(cf. Definition \ref{defn:A}).
\end{proof}

\begin{cor}\label{cor:limit}
Assume that \eqref{eq:Delta} holds at $(\beta_0,\omega_0)=(\beta_0,tH)$.
Then there is a neighborhood $B$ of $(\beta_0,\omega_0)$ such that
${\frak s}(\beta,\omega)=\sigma_{(\beta,\omega)}$ for $(\beta,\omega) \in B$
such that $\omega \in \pi^*(\Amp(Y)_{\Bbb R})$ and $\beta \in \NS(X)_{\Bbb R}$.
\end{cor}

\begin{proof}
We first assume that
 $\omega \in \pi^*(\Amp(Y)_{\Bbb Q})$ and $\beta \in \NS(X)_{\Bbb Q}$.
By Lemma \ref{lem:Delta} and \eqref{eq:Delta},
we may assume that $\Delta_B =\emptyset$.
For the family of stability conditions \eqref{eq:section},
${\cal O}_x$ are ${\frak s}(\beta',\omega')$-stable if $(\beta',\omega') \in B$
and $\omega' \in \Amp(X)_{\Bbb R}$. 

Then applying Proposition \ref{prop:limit2},
${\frak s}(\beta,\omega)=\sigma_{(\beta,\omega)}$.  

We next treat the general case.
We set $\sigma:={\frak s}(\beta,\omega)$.
By the proof of Proposition \ref{prop:limit2},
it is sufficient to
show that
$H^i(E)=0$ ($i \ne -1,0$) for
all $\sigma$-stable object $E$ with 
$\phi_\sigma(E)=1$.
Let $U$ be a neighborhood of $\sigma$ such that
$E$ is $\sigma'$-stable for all $\sigma' \in U$.
If $\rk E \ne 0$, then
there is $(\beta',\omega') \in 
\NS(X)_{\Bbb Q} \times \pi^*(\Amp(Y)_{\Bbb Q})$ 
such that $\sigma':={\frak s}(\beta',\omega') \in U$ and
$\phi_{{\frak s}(\beta',\omega')}(E)<1$.
Hence $H^i(E)=0$ for $i \ne -1,0$.
Assume that $\rk E=0$. If $c_1(E) \not \in \pi^*(\Amp(Y))^{\perp}$,
then we can take $(\beta',\omega')$ such that 
$\phi_{\sigma'}(E) < 1$, which implies
$H^i(E)=0$ for $i \ne -1,0$.
If $c_1(E)  \in \pi^*(\Amp(Y))^{\perp}$, 
then $\phi_{\sigma'}(E)=1$, which also implies
$H^i(E)=0$ for $i \ne -1,0$.
\end{proof}

\begin{lem}\label{lem:Delta2}
Assume that $\beta_0$ and $\omega_0$ are rational, 
and satisfy \eqref{eq:Delta}.
Then there is no $\mu$-semi-stable sheaf $E$ of $\rk E>0$ with respect to
$\omega_0$ such that $Z_{(\beta_0,\omega_0)}(E) \in {\Bbb R}_{<0}$.
 \end{lem}

\begin{NB}
\begin{defn}
A coherent sheaf $E$ of $\rk E>0$ is $\mu$-semi-stable with respect to $H$,
if $(c_1(E_1),H) \leq \rk E_1 \frac{(c_1(E),H)}{\rk E}$
for all subsheaf $E_1$ of $E$.
\end{defn}
\end{NB}
\begin{proof}
Let $E$ be a $\mu$-semi-stable sheaf with $\deg_{\beta_0}(E)=0$.
\begin{NB}
We note that $(c_1(E),\omega_0)=\rk E(\beta,\omega_0)$
implies $\beta=c_1(E)/\rk E+D$, $D \in \omega_0^\perp$.
Hence perturbing $D$, we may assume that $D$ is rational, that is,
$\beta \in \NS(X)_{\Bbb Q}$.
\end{NB} 
Let ${\frak C}$ be a category of perverse coherent sheaves 
associated to $\beta_0$.
Then we have a decomposition 
$$
0 \to E_1 \to E \to E_2 \to 0  
$$
such that $E_1 \in {\frak C} \cap \Coh(X)$ and
$E_2 \in {\frak C}[-1] \cap \Coh(X)$.
Then $a_{\beta_0}(E_2)=\chi_{\beta_0}(E_2) \leq 0$, and hence
$a_{\beta_0}(E) \leq a_{\beta_0}(E_1)$.
In particular,
$Z_{(\beta_0,\omega_0)}(E) \geq 
Z_{(\beta_0,\omega_0)}(E_1) (\in {\Bbb R}_{ \leq 0})$.
Since $E$ is $\mu$-semi-stable with respect to
$\omega_0$,
$E_1$ is a $\mu$-semi-stable perverse coherent sheaf.
In ${\frak C}$, $E_1$ s generated by
$\beta$-twisted stable torsion free objects $F$
with $\deg_{\beta_0}(F)=0$.
If $\langle v(F)^2 \rangle \geq 0$, then
$a_{\beta_0}(F) \leq 0$, and hence 
$Z_{(\beta_,\omega_0)}(F) \in {\Bbb R}_{>0}$.
If $v(F) \in \Delta(X)$, then our assumption implies
$Z_{(\beta_,\omega_0)}(F) \in {\Bbb R}_{>0}$.
Therefore $Z_{(\beta,\omega)}(E_1) \in {\Bbb R}_{ > 0}$.
\end{proof}

\begin{prop}\label{prop:limit3}
Let $\sigma_s$ ($0 \geq s \gg -1$) be a family of stability conditions
such that $Z_{\sigma_s}(\bullet)=
\langle e^{(\beta+sH)+iH},\bullet \rangle$. 
Assume that irreducible objects $A \in {\frak C}$ are $\sigma_s$-semi-stable
with $\phi_s(A)=1$ 
and $\sigma_s$-stable for $s<0$.
Then 
$\sigma_0=({\cal A}_{(\beta,H)},Z_{(\beta,H)})$.
\end{prop}

\begin{proof}
We first prove that
$\pH^i(E)=0$ $(i \ne -1,0)$ for all 
$E \in {\cal P}_{\sigma_0}((0,1])$.
Let $E$ be a $\sigma_0$-stable object of 
${\cal P}_{\sigma_0}(\phi)$ with $0<\phi<1$.
Let $A$ be an irreducible object of ${\frak C}$.
Then $\Hom(E[i],A)=0$ for $i>0$ by
$\phi_{\sigma_0}(E[i])=\phi+i>1$.
Hence $\pH^i(E)=0$ for $i>0$.
$\Hom(E[i],A)=\Hom(A,E[i+2])^{\vee}=0$
for $i \leq -2$ by $\phi_{\sigma_0}(E[i+2])=\phi+i+2<1$.
Hence $\pH^i(E)=0$ for $i \ne -1,0$.

Let $E$ be a $\sigma_0$-stable object of 
${\cal P}_{\sigma_0}(\phi)$ with $\phi=1$.
Assume that $\phi_{\sigma_s}(E)>1$ for $s<0$.
We set $F:=E[-1]$.
Then $\pH^i(F)=0$ for $i \ne -1,0$.
We set 
\begin{equation}\label{eq:v_i}
v_i:=v(\pH^{-i}(E))=e^\beta(r_i+d_i H+D_i+a_i \varrho_X),\;D_i \in H^\perp, \;(i=0,1).
\end{equation}
%
Since $\phi_{\sigma_0}(F)=0$, $d_0-d_1=0$.
By the description of $\sigma_s$ $(s<0)$,
$d_0-r_0 s \geq 0$ and $d_1-r_1 s \leq 0$ for $0>s \gg -1$.
\begin{NB}
$d_0-r_0 s=0$ if $\pH^0(F)$ is 0-dimensional.
\end{NB}
Hence $d_0 \geq 0$ and $d_1 \leq 0$, which implies $d_0=d_1=0$.
Then $d_1-r_1 s \leq 0$ and $s<0$ implies $r_1=0$.
Hence $\pH^{-1}(F)$ is a torsion object of ${\frak C}$.
By the description of ${\cal P}_{\sigma_s}((0,1])$,
$\pH^{-1}(F)=0$.
If $\phi_{\sigma_s}(E) \leq 1$ for $s<0$, then 
we also have $\pH^i(E)=0$ for $i \ne -1,0$.

Then it is easy to see that
${\cal T}:={\cal P}_{\sigma_0}((0,1]) \cap {\frak C}$ and
${\cal F}:={\cal P}_{\sigma_0}((-1,0]) \cap {\frak C}$
is a torsion pair of ${\frak C}$,
${\cal P}_{\sigma_0}((0,1])$ is the tilting and
$({\cal T},{\cal F})=({\cal T}_{(\beta,H)},{\cal F}_{(\beta,H)})$.

\begin{NB}
We define $v_i$ as in \eqref{eq:v_i}.
Since $\phi_{\sigma_0}(E)=0$, $d_0-d_1=0$.
By the description of $\sigma_s$ $(s<0)$,
$d_0-r_0 s \geq 0$ and $d_1-r_1 s \leq 0$ for $0>s \gg -1$.
Hence $d_0 \geq 0$ and $d_1 \leq 0$, which implies $d_0=d_1=0$.
Then $d_1-r_1 s \leq 0$ and $s<0$ implies $r_1=0$.
By $r_1=d_1=0$, $\pH^{-1}(E)$ is a 0-dimensional object of ${\frak C}$
or 0.
By the description of ${\cal P}_{\sigma_s}((0,1])$,
$\pH^{-1}(E)$ is a torsion free object of
${\frak C}$, which shows
$\pH^{-1}(E)=0$.
Therefore the claim holds.

(Step 2).
Let $E$ be an object of ${\frak C}$.
For $A \in {\cal P}_{\sigma_0}(>1)$,
$\pH^i(A)=0$ for $i \geq 0$ so that $\Hom(A,E)=0$.
For $B \in {\cal P}_{\sigma_0}(\leq -1)$,
$\pH^i(B)=0$ for $i \leq 0$ so that $\Hom(E,B)=0$.
Hence $E \in {\cal P}_{\sigma_0}((-1,1])$.
For $E \in {\frak C}$, we have a triangle
$$
D \to E \to F \to D[1]
$$
such that $D \in  {\cal P}_{\sigma_0}((0,1])$ and
$F \in  {\cal P}_{\sigma_0}((-1,0])$.
Then $\pH^i(D)=0$ for $i \ne -1,0$ and
$\pH^i(F)=0$ for $i \ne 0,1$.
Taking their cohomogy, 
we see that $D,F \in {\frak C}$.

(Step 3).
Then ${\cal T}:={\cal P}_{\sigma_0}((0,1]) \cap {\frak C}$ and
${\cal F}:={\cal P}_{\sigma_0}((-1,0]) \cap {\frak C}$
is a torsion pair and ${\cal P}_{\sigma_0}((0,1])$ is the tilting.

(Step 4).
Let $E$ be a $\beta$-twisted stable object of ${\frak C}$.
Then we have an exact sequence
$$
0 \to E_1 \to E \to E_2 \to 0
$$ 
such that $E_1 \in {\cal T}$ and $E_2 \in {\cal F}$.
If $E_1 \ne 0$,
$Z_{\sigma_0}(E_1) \in {\Bbb H} \cup {\Bbb R}_{<0}$,
which implies $\deg_\beta(E_1)>0$ or $\deg_\beta(E_1)=0$ and
$a_\beta(E_1)>\frac{(H^2)}{2}\rk E_1$.
If $E_2 \ne 0$, then
we also have 
 $-Z_{\sigma_0}(E_2) \in {\Bbb H} \cup {\Bbb R}_{<0}$,
which implies $\deg_\beta(E_2)<0$ or $\deg_\beta(E_2)=0$ and
$a_\beta(E_2)<\frac{(H^2)}{2}\rk E_2$. 
If $E$ is a torsion object of ${\frak C}$, then
$E_2$ is also a torsion object of ${\frak C}$.
Therefore $E_2=0$, which implies
$E \in {\cal T}$.
Since a torsionn objects is generated by $\beta$-twisted stable 
torsion objects,
${\cal T}$ contains all torsion objects of ${\frak C}$,
which also implies all objects of ${\cal F}$ are torsion free.
 
Assume that $E$ is torsion free.
If $E_1$ and $E_2$ are not zero, then
we see that $E$ is not $\beta$-stable.
Therefore $E \in {\cal T}$ or $E \in {\cal F}$.
In particular, $\sigma_0=({\cal A}_{(\beta,H)},Z_{(\beta,H)})$.
\end{NB}
\end{proof}

\subsection{A family of stability conditions parametrized by a half plane}
\label{subsect:half-plane}

We consider stability conditions
\begin{equation}\label{eq:Plane}
P_{\gamma,H}:=\{\sigma_{(\gamma+sH,tH)} 
\mid s \in {\Bbb R}, t \in {\Bbb R}_{>0},\; 
Z_{(\gamma+sH,tH)}(u) \ne 0\; (u \in \Delta(X)) \}.
\end{equation}
To be more precise,
$\sigma_{(\gamma+sH,tH)}$ is well-defined on a simply connected open subset
such that 
$Z_{(\gamma+sH,tH)}(u) \not \in {\Bbb R}_{\leq 0}$
for all $u \in \Delta(X)$ with $\rk u>0$.
We shall regard $\sigma_{(\gamma+sH,tH)}$ as a limit 
as in Proposition \ref{prop:limit3} if  
$Z_{(\gamma+sH,tH)}(u) \in {\Bbb R}_{<0}$
for a Mukai vector $u \in \Delta(X)$ with $\rk u>0$.

For a Mukai vector $v$,
let $U_v$ be a chamber in $P_{\gamma,H}$.
For a wall $W$ in $\Stab(X)$,  we have
\begin{enumerate}
\item[(1)] $P_{\gamma,H} \subset W$ or
\item[(2)] $W$ and $P_{\gamma,H}$ intersect properly and 
$W\cap   P_{\gamma,H}$ is a wall for $v$.
\end{enumerate}
Hence if ${\cal C}_v$ is a chamber in $\Stab(X)$
with $\overline{{\cal C}_v} \cap U_v \ne \emptyset$,
then $U_v \subset  \overline{{\cal C}_v}$.
Let 
$$
v_1=e^\gamma(r_1+d_1 H+D_1+a_1 \varrho_X),\; D_1 \in H^\perp
$$
 be a Mukai vector
which defines a wall $W$ for $v=e^\gamma(r+dH+D+a \varrho_X)$
($D \in H^\perp$).
Then 
$$
P_{\gamma,H} \subset W \Longleftrightarrow
(r_1,d_1,a_1) \in {\Bbb Q}(r,d,a).
$$
For $v=\varrho_X$,
there is a chamber
$$
{\cal C}_\varrho:=\{(\beta,\omega) \mid ( C \cdot \omega)>0 
\text{ for any exceptional curve $C$} \}
=\{(\beta,\omega) \mid \omega \in \Amp(X)_{\Bbb R} \}
$$
in $\NS(X)_{\Bbb R} \times P^+(X)_{\Bbb R}$. 
For $v_0:=r_0 e^\gamma$,
let $U_{v_0}$ be a chamber in $P_{\gamma,H}$.
We take $(\beta_0,\omega_0):=(\gamma+s_0 H,t_0 H) \in U_{v_0}$.
We take $(\beta,\omega)$ in an adjacent chamber ${\cal C}$
of $U_{v_0}$, and let ${\cal E} \in {\bf D}(X \times X')$ 
be a universal family (as a twisted object in general)
of $\sigma_{(\beta,\omega)}$-stable objects with
the Mukai vector $v_0$, where $X'$ is the coarse moduli scheme of
$\sigma_{(\beta,\omega)}$-stable objects.
Let 
$$
\Phi:=\Phi_{X \to X'}^{{\cal E}^{\vee}[2]}:{\bf D}(X) \to {\bf D}(X')
$$
be a Fourier-Mukai transform defined by
\begin{equation}\label{eq:FM}
\Phi(E):={\bf R}p_{X'*}(p_X^*(E) \otimes {\cal E}^{\vee}[2]),
\end{equation}
where $p_X,p_{X'}$ are projections from $X \times X'$ to $X$ and $X'$ respectively.
$\Phi$ induces an isomorphism
\begin{equation}
\begin{matrix}
\Phi& :\NS(X)_{\Bbb R} \times P^+(X)_{\Bbb R} 
& \to & \NS(X')_{\Bbb R} \times P^+(X')_{\Bbb R}\\
& (\beta,\omega) & \mapsto & (\beta', \omega'),
\end{matrix}
\end{equation}
where 
$(\beta',\omega')$ is defined by
\begin{equation}\label{eq:Phi}
e^{\beta'+i \omega'}
:=\frac{\Phi(e^{\beta+i \omega})}
{-\langle e^{\beta+i \omega},v_0 \rangle}.
\end{equation}
For $(\beta,\omega) \in {\cal C}$,
$\omega'$ is ample and 
$\Phi(\sigma_{(\beta,\omega)})\equiv \sigma_{(\beta',\omega')} \mod
\widetilde{GL}_2^+({\Bbb R})$.
For $(\beta,\omega)=(\gamma+sH,tH)$,
we set $(\beta',\omega'):=(\gamma'+s' H',t'H')$,
where $\Phi(\varrho_X)=r_0 e^{\gamma'}$ and 
$\Phi(e^\gamma H)=-e^{\gamma'} H'$.
Then $\Phi(U_{v_0})$ is a chamber in $P_{\gamma',H'}$.
Since $\Phi(\overline{\cal C}) \supset \Phi(U_{v_0}) \ni (\beta_0',\omega_0')$,
there is a category of perverse coherent sheaves ${\frak C}'$
associated to a contraction $\pi':X' \to Y'$ by $H'$, and
$\sigma_{(\beta',\omega')}$ $((\beta,\omega) \in U_{v_0})$ 
is the stability condition 
in subsection \ref{subsect:perverse}
(Proposition \ref{prop:limit2}).

We set $L:= (e^\gamma)^\perp \cap 
(H e^{\gamma})^\perp \cap v(K(X))$.
Then $L$ is negative semi-definite.
Let 
\begin{equation}\label{eq:decomp}
r_0 e^\gamma=\sum_i n_i u_i,\;
(n_i \in {\Bbb Z}_{>0},\; u_i \in L)
\end{equation}
 be a decomposition
of $r_0 e^\gamma$ in $L$ such that 
\begin{enumerate}
\item
$u_i$ are indecomposable Mukai vectors with
$\langle u_i^2 \rangle=-2$,
\item
$\rk u_i/r_0>0$ and $\sum_i n_i \rk u_i/r_0=1$,
\end{enumerate}
where 
$u_i$ is indecomposable, if there is no
decomposition
$u_i=\sum_j n_{ij}u_{ij}$ $(n_{ij} \in {\Bbb Z}_{>0}, u_{ij} \in L)$
such that $\langle u_{ij}^2 \rangle=-2$ and
$\rk u_{ij}/r_0>0$.
By \cite[Lem. A.1.1]{PerverseI},
the sublattice $\oplus_i {\Bbb Z}u_i$ is of type 
$\widetilde{A},\widetilde{D},\widetilde{E}$.
\begin{NB}
$d_{\gamma+sH}(r_0 e^\gamma)=-r_0 s(H^2)$ and
$d_{\gamma+sH}(u_i)=-\rk u_i s(H^2)$ have the same sign.
\end{NB}

Then there are $\sigma_{(\beta_0,\omega_0)}$-semi-stable objects $E_i$
with $v(E_i)=u_i$.
Since $u_i$ are indecomposable,
$E_i$ are $\sigma_{(\beta_0,\omega_0)}$-stable.

\begin{NB}
Then there is a properly $\sigma_{(\gamma+sH,tH)}$-semi-stable
object ${\cal E}_{|X \times \{ x' \}}$, and 
for any properly $\sigma_{(\gamma+sH,tH)}$-semi-stable
object ${\cal E}_{|X \times \{ x' \}}$,
the $S$-equivalence class is $\oplus_i E_i^{\oplus a_i}$.
\end{NB}

\begin{lem}
Let $E$ be a $\sigma_{(\beta_0,\omega_0)}$-stable
object with $\phi_{(\beta_0,\omega_0)}(E)=
\phi_{(\beta_0,\omega_0)}(r_0 e^\gamma)$ and
$v(E) \ne r_0 e^\gamma$.
Assume that $E$ satisfies $\langle v(E),e^{\gamma} \rangle=0$.
Then $\langle v(E)^2 \rangle=-2$ and
$E$ is a stable factor of  ${\cal E}_{|\{ x' \} \times X}$.
\end{lem}

\begin{proof}
$\Phi(E)$ is a $\sigma_{(\beta_0',\omega_0')}$-stable 
 object with $\phi_{(\beta_0',\omega_0')}(\Phi(E))=1$.
In particular, 
$\Phi(E) \in {\cal A}_{(\beta_0',\omega_0')}$.
Since $\rk \Phi(E)=0$ and irreducible,
$\pH^i(\Phi(E))=0$ for $i \ne -1$ or $i \ne 0$.
Since $\pH^{-1}(\Phi(E))$ is torsion free,
$\Phi(E) \in {\frak C}'$.
Thus $\Phi(E)$ is a 0-dimensional object of ${\frak C}'$.
We note that
$F \in {\cal A}_{(\beta_0',\omega_0')}$
is an irreducible object with $\rk F=0$
if and only if $F$ is an irreducible object of ${\frak C}'$.
\begin{NB}
Note that $(\beta'_0,\omega'_0)$ is general in 
$P_{\gamma',H'}$.
\end{NB}
Hence $\langle v(E)^2 \rangle=\langle v(\Phi(E))^2 \rangle=-2$.
Since every irreducible object of ${\frak C}'$
is a stable factor of ${\cal O}_{x'}$ by \cite[Lem. 1.1.21]{PerverseI}, 
$E$ is a stable factor of ${\cal E}_{|\{ x' \} \times X}$.
\end{proof}

An object  
$E \in {\frak C}'$ with $v(E)=\varrho_{X'}$
is $\nu$-stable in the sense of \cite[sect. 2.2]{PerverseI}, if
$-(\nu,c_1(F_1)) \leq 0$ for all subobject $F_1$ of
$E$.
Hence it is the same as $\sigma_{(\beta',\omega')}$-semi-stability,
where $(\beta',\omega')=(\gamma'+s' H'+\mu,t' H'+\nu)$
is sufficiently close to $(\beta_0',\omega_0')$.
If $\nu$ is relatively ample with respect to $\pi'$,
then ${\cal O}_{x'}$ is $\nu$-stable. 

Since ${\cal M}_{(\beta_0,\omega_0)}(r_0 e^\gamma)
\cong {\cal M}_{(\beta_0',\omega_0')}(\varrho_{X'})$, 
\cite[Prop. 2.2.8]{PerverseI}
implies
 ${\cal M}_{(\beta_0,\omega_0)}(r_0 e^\gamma)$ is irreducible.
Moreover the $S$-equivalence classes of 
properly $\sigma_{(\beta_0,\omega_0)}$-semi-stable objects
are $\oplus_i E_i^{\oplus n_i}$, where
$v_0=\sum_i n_i v(E_i)$ is a decomposition of \eqref{eq:decomp}.

\begin{thm}\label{thm:2-dim}
Let $U_{v_0}$ be a chamber for $v_0=r_0 e^\gamma$ in $P_{\gamma,H}$
and $(\beta_0,\omega_0) \in U_{v_0}$.
Assume that $\chr(k)=0$ or $M_{(\beta,\omega)}(r_0 e^\gamma)$
is a fine moduli space for $(\beta,\omega)$ in an adjacent chamber ${\cal C}$
of $U_{v_0}$.
Then $M_{(\beta_0,\omega_0)}(r_0 e^\gamma)$
is isomorphic to a normal K3 surface $Y'$ which is obtained as 
a contraction $\pi':X' \to Y'$ by the 
nef and big divisor $\omega_0'$. 
\end{thm}

\begin{proof}
If $\chr(k)=0$, then \cite[Prop. 2.2.11]{PerverseI}
implies $M_{(\beta_0',\omega_0')}(\varrho_{X'})=(X')^0$
is normal. Hence the claim holds.
If $M_{(\beta_0,\omega_0)}(r_0 e^\gamma)$
is a fine moduli space, then 
$M_{(\beta_0',\omega_0')}(\varrho_{X'})=(X')^0$
is the moduli of untwisted 0-dimensional
objects of Mukai vector $\varrho_{X'}$, where
$(X')^0$ is the moduli of $0$-semi-stable objects 
in \cite[Def. 2.2.1]{PerverseI}.
In this case, there is an autoequivalence
$\Phi$ of ${\bf D}(X')$ such that 
$\Phi({\frak C})=^{-1} \Per(X'/Y')$ by
\cite[Prop. 2.3.27]{PerverseI}.
Applying \cite[Rem. 2.2.13]{PerverseI},
we see that $(X')^0 \cong Y'$.
\end{proof}

\begin{NB}
The walls for $v_0$ are defined by $(-2)$-vectors $u=e^\gamma(r+D) \in L$.
We set $(\beta,\omega):=(\gamma+sH+\mu,tH+\nu)$,
$(\mu,\nu \in H^\perp)$.
Then the wall for $e^\gamma$ is
\begin{equation}
\frac{(s^2-t^2)(H^2)+((\mu^2)-(\nu^2))}{2}(\nu,D)
-(st(H^2)+(\mu,\nu))(\mu,D)=0.
\end{equation}
Thus in a neighborhood of $(\gamma+s_0 H,t_0 H)$,
the tangent space of the wall is
$$
\left((s_0^2-t_0^2)\frac{(H^2)}{2}\nu-s_0 t_0(H^2)\mu,D \right)=0.
$$

\end{NB}

\begin{NB}
Let $U$ be the chamber in $(s,t)$-plane such that
$t>f(s)$ and $s<0$ is contained.
Let ${\cal C}$ be a chamber in $\NS(X)_{\Bbb R} \times \NS(X)_{\Bbb R}$
such that $U \subset \overline{\cal C}$.
Then ${\cal O}_x$ is stable in ${\cal C}$.    
Indeed the walls are defined by the Mukai vectors of 
0-dimensional perverse coherent sheaves $E$.
Thus if $v(E)=e^\gamma(D+a \varrho_X)$, then
$(D,\omega)=0$ is the equation of the wall. 

We set $(\beta,\omega):=(\gamma+sH+\mu,tH+\nu)$,
$(\mu,\nu \in H^\perp)$.
Then the wall for $e^\gamma$ is
\begin{equation}
\frac{(s^2-t^2)(H^2)+((\mu^2)-(\nu^2))}{2}(\nu,D)
-(st(H^2)+(\mu,\nu))(\mu,D)=0.
\end{equation}
Thus in a neighborhood of $(\gamma+s_0 H,t_0 H)$,
the tangent space of the wall is
$$
\left((s_0^2-t_0^2)\frac{(H^2)}{2}\nu-s_0 t_0(H^2)\mu,D \right)=0.
$$
Hence ${\cal C}$ is approximated by
a connected component of the complement
of hyperplanes defined by $v(E)=e^\gamma(r+D)$.

For $(\beta,\omega) \in {\cal C}$,
let ${\cal E}$ be a universal family
of $\sigma_{(\beta,\omega)}$-stable complexes
with the Mukai vector $r_0 e^\gamma$.
Then $\Phi({\cal O}_x)={\cal E}_{|\{ x \} \times X'}^{\vee}[2]$ 
is $\sigma_{(\beta_0',\omega_0')}$-semi-stable
and $\sigma_{(\beta',\omega')}$-stable.
If $U \subset \overline{\cal C}$, then 
$\Phi(U) \subset \overline{\Phi(\cal C)}$.
Hence ${\cal E}_{|\{ x \} \times X'}$ is a twisted stable sheaf.

Let $E$ be a $\sigma_{(\gamma+sH,tH)}$-semi-stable object with
$v(E)=r_0 e^\gamma$.
For a subobject $E_1$ of $E$
with $v(E_1)=e^\gamma(r_1+D_1)$,
$\phi_{(\beta,\omega)}(E_1)<\phi_{(\beta,\omega)}(E)$
if and only if 
\begin{equation}
-r_0 \frac{(\beta^2)-(\omega^2)}{2}(\nu,D_1)+r_0(\beta,\omega)(\mu,D_1)<0
\end{equation}
where $(\beta,\omega)=(sH+\mu,tH+\nu)$.
In particular, if $t \gg 0$, then
the condition is $r_0(\nu,D_1)>0$.
Assume that $r_0>0$.
Then for $t \gg 0$,
${\cal M}_{(\gamma+sH,tH)}(r_0 e^\gamma)={\cal M}_H^\gamma(r_0 e^\gamma)$.
$E \in{\cal M}_H^\gamma(r_0 e^\gamma)$ is $(\gamma+\lambda)$-stable
$(|(\lambda^2)| \ll 1)$
if 
\begin{equation}
-\frac{(c_1(E_1),\lambda)}{\rk E_1}<
-\frac{(c_1(E),\lambda)}{\rk E}
\end{equation}
for all $\gamma$-twisted stable subobjects $E_1$
of $E$ with $v(E_1)=e^\gamma(r_1+D_1)$.

\end{NB}

\begin{NB}
 Let $\sigma_0$ be a stability condition such that
$Z_{\sigma_0}(\bullet)=\langle e^{\beta_0+\sqrt{-1}\omega_0},\bullet \rangle$.
Assume that ${\cal O}_x$ is $\sigma_0$-semi-stable for all
$x \in X$.
In a neighborhood $U$ of $(\beta_0,\omega_0)$,
we have a continuous map $f:U \to \Stab(X)^*$
 such that $f(\beta_0,\omega_0)=\sigma_0$,
$Z_{f(\beta,\omega)}(\bullet)=
\langle e^{\beta+\sqrt{-1} \omega},\bullet \rangle$
and $U \to f(U)$ is homeomorphic.
So we can regard $U$ as a neighborhood of $\sigma_0$.
Let $u_1,u_2,...,u_n$ be Mukai vectors of stable factors of
${\cal O}_x$.
We set $P_\pm:=\{i \mid \pm \rk u_i>0\}$ and
$P_0:=\{ i \mid \rk u_i=0 \}$.
Let ${\cal C}$ be the chamber 
such that $(\beta,\omega) \in {\cal C}$ if and only if
$\omega$ is ample, 
$\Ima Z_{(\beta,\omega)}(u_i)>0$ for $i\in P_+$, and
$\Ima Z_{(\beta,\omega)}(u_i)<0$ for $i\in P_-$.
Then $(\beta_0+s\omega_0,\omega_0) \in \overline{\cal C}$
for $0>s \gg -1$.

\begin{NB2}
Assume that $\omega_0+\eta$ is ample.
In the $(s,t)$ plane $(\beta_0+s \omega_0,\omega_0+t \eta)$,
${\cal C}$ is $t>0$ and
$$
\pm(t(c_1(u_i)-\rk u_i, \eta)-s \rk u_i (\omega_0^2)
-st \rk u_i (\eta,\omega_0))>0
$$
for $i \in P_\pm$.
\end{NB2}

Assume that ${\cal O}_x$ is stable for all $\sigma \in {\cal C}$.

We take $\beta \in \NS(X)_{\Bbb Q}$ in the same chamber of $\beta_0$ and
${\frak C}$ be the associated category of perverse coherent sheaves.
\end{NB}

\section{Stability conditions on an Enriques surface}
\label{sect:Enriques}

\subsection{2-dimensional moduli spaces}\label{subsect:2-dim}

The space of stability conditions on an Enriques surface was 
studied in \cite{MMS}
by comparing the stability conditions on the covering K3 surface.
In this section, we shall explain some of the results.
For this purpose, we prepare some notations. 
Let $X$ be a classical Enriques surface over $k$, that is $K_X \ne 0$.
As in the case of K3 surfaces, we introduce the following definition.
\begin{defn}
\begin{enumerate}
\item
[(1)]
For $E \in {\bf D}(X)$, 
$$
v(E)=\ch(E)\sqrt{\td_X}=
\rk E+c_1(E)+\left(\ch_2(E)+\frac{\rk E}{2}\varrho_X \right) 
\in H^*(X,{\Bbb Q})
$$
is the Mukai vector of $E$.
Let $(v(K(X)),\langle \;\;,\;\; \rangle)$ be the Mukai lattice of $X$.
\item[(2)]
Let $\Delta(X)$ be the subset of $v(K(X))$
consisting of $u=(r,\xi,\frac{b}{2})$
such that (i) $\langle u^2 \rangle=-1$ or
(ii) $\langle u^2 \rangle=-2$ and $\xi \equiv D \mod 2$, where
$D$ is a nodal cycle.
\end{enumerate}
\end{defn}

As in \cite{Br:3}, let ${\cal P}(X)$ be the subset of
$v(K(X))_{\Bbb C}$ consisting of $\mho$ such that
$\Rea \mho$ and $\Ima \mho$ span a positive definite
2-plane in $v(K(X))_{\Bbb R}$.
We shall regard ${\cal P}(X)$ as a subset 
of $\Hom(v(K(X)),{\Bbb C})$ by
$\langle \mho,\bullet \rangle$.
Let ${\cal P}^+(X)$ be the connected component of ${\cal P}(X)$
containing $e^{\beta+i \omega}$, where $\omega$ is ample.
We set 
$$
{\cal P}_0^+(X):={\cal P}^+(X) \setminus 
\cup_{u \in \Delta(X)}u^\perp.
$$
Then for the connected component $\Stab^{\dagger}(X)$ 
containing geometric stability conditions,  
\begin{equation}
\begin{matrix}
\Stab^{\dagger}(X) & \to & {\cal P}_0^+(X)\\
\sigma & \mapsto & Z_\sigma
\end{matrix}
\end{equation}
is a covering map with the group of deck transformations
$\Aut^0({\bf D}(X))/\langle \otimes K_X \rangle$
\cite[Cor. 3.8]{MMS} at least if $k={\Bbb C}$.
By the same argument of Bridgeland \cite[Prop. 13.2]{Br:3}, 
we have
\begin{equation}\label{eq:Stab/U(X)}
\Stab^{\dagger}(X)=\cup_{\Phi \in {\bf T}} \Phi(\overline{U(X)})
\end{equation}
where ${\bf T} \subset \Aut({\bf D}(X))$ ie the subgroup of autoequivalences 
generated by twist functors $T_A^2$ and 
$T_{{\cal O}_C(k)}$ (see Definition \ref{defn:twist}),
where $A$ is a spherical object or an exceptional object, and $C$ is a 
$(-2)$ curve on $X$.
Thus $\Phi(U(X))$ is the chamber of $\varrho_X$ and
we have a fine moduli space for every chamber (see \cite{Ha:1} for
the corresponding result on a K3 surface).

\begin{thm}\label{thm:FM-kernel}
\begin{enumerate}
\item[(1)]
Let $v_0=(r,\xi,\frac{s}{2})$ be a primitive and isotropic
Mukai vector such that 
$\gcd(r,\xi,s)=2$. 
Let ${\cal C}$ be a chamber with respect to $v_0$.
Then
$M_{(\beta,\omega)}(v_0)$ $((\beta,\omega) \in {\cal C})$
is a fine moduli space and
$M_{(\beta,\omega)}(v_0) \cong X$.
\item[(2)]
For $\eta/p$ such that $\eta \in \NS(X)$ and
$p$ is an odd integer,
a primitive element
$v_0 \in {\Bbb Q}e^{\eta/p} \cap v(K(X))$ 
satisfies the assumption.
\end{enumerate}
\end{thm}

\begin{proof}
(1)
By \eqref{eq:Stab/U(X)}, $M_\sigma(\varrho_X) \cong X$
if $\sigma$ belongs to a chamber.
We treat the general case, by using a special kind of Fourier-Mukai transforms.
Let ${\bf G}$ be the subgroup of $\Aut({\bf D}(X))$ generated
by the following autoequivalences:
\begin{enumerate}
\item
A twist functor
$T_{{\cal O}_X}$.
$T_{{\cal O}_X}$ is a Fourier-Mukai transform $\Phi_{X \to X}^{{\cal E}[1]}$,
where ${\cal E}_{|\{x \} \times X}$ is a stable sheaf with 
Mukai vector $w_0:=v({\cal O}_X \oplus {\cal O}_X(K_X))-\varrho_X$.
\item
For $D \in \NS(X)$,
\begin{equation}
\begin{matrix}
{\cal L}_D: & {\bf D}(X) & \to &{\bf D}(X)\\
& E & \mapsto & E(D).
\end{matrix}
\end{equation} 
\item
The shift functor: 
\begin{equation}
\begin{matrix}
[1]: & {\bf D}(X) & \to &{\bf D}(X)\\
& E & \mapsto & E[1].
\end{matrix}
\end{equation} 
\end{enumerate}

The equivalence
$\Phi_{X \to X}^{{\cal E}}:{\bf D}(X) \to {\bf D}(X)$ in (i) 
induces an  isomorphism
$\Stab^{\dagger}(X) \to \Stab^{\dagger}(X)$, and hence isomorphisms
${\cal M}_\sigma(v) \to 
{\cal M}_{\Phi_{X \to X}^{{\cal E}}(\sigma)}(v')$, 
where $v':=\Phi_{X \to X}^{{\cal E}}(v)$.
Indeed since $M_H(w_0)=M_{(\beta,\omega)}(w_0)$
for some $(\beta,\omega) \in \NS(X)_{\Bbb Q} \times \Amp(X)_{\Bbb Q}$
with $\sigma_{(\beta,\omega)} \in U(X)$,
there is $\sigma' \in U(X) \subset \Stab^{\dagger}(X)$ such that 
$\Phi_{X \to X}^{{\cal E}}(\sigma')=\sigma_{(\beta,\omega)} 
\in U(X) \subset \Stab^{\dagger}(X)$, 
which implies the connected component $\Stab^{\dagger}(X)$ is preserved
under $\Phi_{X \to X}^{{\cal E}}$.
Obviously ${\cal L}_D$ and $[1]$ also preserve the stability.
Since there is an autoequivalence $\Phi \in {\bf G}$
such that $\Phi(v_0)=\varrho_X$ (cf. \cite{Y:Enriques}),
we get $M_\sigma(v_0) \cong M_{\Phi(\sigma)}(\varrho_X) \cong X$. 

\begin{NB}
Old argument:
We may assume that $r>0$ or $r=0$ and $\xi$ is effective.
We take a general ample divisor $H$
with respect to $v_0$.
Then the moduli scheme of semi-stable sheaves 
$M_H(v_0)$ is a fine moduli space isomorphic
to $X$.
Let ${\cal E}$ be a 
universal family on $X \times X$.
Then the equivalence 
$\Phi_{X \to X}^{{\cal E}}:{\bf D}(X) \to {\bf D}(X)$
induces an  isomorphism
$\Stab^{\dagger}(X) \to \Stab^{\dagger}(X)$ and isomorphisms
${\cal M}_\sigma(\varrho_X) \to 
{\cal M}_{\Phi_{X \to X}^{{\cal E}}(\sigma)}(v_0)$.
Indeed since $M_H(v_0)=M_{(\beta,\omega)}(v_0)$
for some $(\beta,\omega) \in \NS(X)_{\Bbb Q} \times \Amp(X)_{\Bbb Q}$
(cf. \cite[Cor. 3.2.10]{MYY:2011:1}),
there is $\sigma \in U(X) \subset \Stab^{\dagger}(X)$ such that 
$\Phi_{X \to X}^{{\cal E}}(\sigma)=\sigma_{(\beta,\omega)} 
\in U(X) \subset \Stab^{\dagger}(X)$, 
which implies the connected component $\Stab^{\dagger}(X)$ is preserved
under $\Phi_{X \to X}^{{\cal E}}$.
Hence there  is a fine moduli space 
$M_{\sigma'}(v_0)$ of $\sigma'$-stable complexes
$E$ with $v(E)=v_0$ for all general $\sigma'$ by
\eqref{eq:Stab/U(X)}.
\end{NB}

(2)
We set $v_0:=lp e^{\eta/p}=(lp,l\eta,\frac{s}{2})$.
Then we see that 
$lp$ is even.
Since $p$ is odd, $l$ is even, which implies
$\gcd(lp,l \eta,s)=2$ by $lp \equiv s \mod 2$. 
\begin{NB}
For $\frac{\xi}{r}$ with $\gcd(r,\xi)=1$,
assume that $r$ is odd.
Let $(lr,l\xi,a)$ be a primitive and isotropic Mukai vector.
Then  $l$ is even: Indeed
if $l$ is odd, then $2a$ is odd, which implies
$(\xi^2)=2lra$ is odd.
Hence $l$ is even.

\begin{NB2}
We set $s:=\gcd(r,(\xi^2)/2)$,
$p:=r/s$ and $t:=(\xi^2)/(2s)$.
Then $\gcd(p,t)=1$.

If $r=ps$ is odd, then $(2p^2 s,2p\xi,2t)$ is primitive.
Indeed $\gcd(2p^2 s,2p \xi,4t)=2$ and 
$p^2 s+2t$ is odd.  

If $p$ is even and $t$ is odd, then
$(p^2 s,p\xi,t)$ is primitive.
Indeed $\gcd(p^2 s,p \xi,2t)=2$ and 
$p^2 s/2+t$ is odd.  
\end{NB2}

\begin{lem}
For $\frac{\xi}{r}$ with $\gcd(r,\xi)=1$,
assume that $r$ is odd.
Then for a primitive and isotropic
Mukai vector $(lr,l\xi,a)$,
$M_H(lr,l\xi+\frac{lr}{2}K_X,a)$ is isomorphic to $X$.
\end{lem}
\end{NB}
\end{proof}

\begin{rem}
The essential part of (1) is the existence of a Fourier-Mukai transform
associated to $v_0$, which was first proved by Nuer \cite{N}. 
\end{rem}

\begin{rem}\label{rem:k}
Assume that $k$ is not algebraically closed.
If all divisors on $X$ are defined over $k$, then
$M_H^\beta(v)$ has a universal family if $v$ is primitive,
by the unimodularity of Mukai lattice.
If $\langle v^2 \rangle=-1,-2$, then 
$M_H(v)$ is a reduced one point. Hence 
there is a $\beta$-stable object $E$ with $v(E)=v$.
Then $\Phi \in {\bf T}$ are defined over $k$. 
\end{rem}

\subsection{Gieseker chambers on an Enriques surface}

The results in section \ref{sect:MYY} hold for the case of Enriques surfaces. 
In \cite{FM-duality},
we studied Gieseker chambers in a 
2-dimensional subspace of $\Stab^{\dagger}(X)$
for a primitive and isotropic Mukai vector $v_0$ on
a K3 surface. In this section, we present a similar result
for stability conditions associated to a category of perverse coherent
sheaves on
a K3 surface and also an Enriques surface $X$.   
Let $H$ be a nef and big divisor which defines a contraction
$\pi:X \to Y$ of $(-2)$-curves $C \in H^\perp$.
Let ${\frak C}$ be a category of perverse coherent sheaves
with a local projective generator $G$.
We set $\beta:=c_1(G)/\rk G$.  
Let $v_0:=r_0 e^\gamma$ be a primitive and isotropic Mukai vector
such that $X':=M_H^\beta(v_0)$ is a smooth surface, that is,
$v_0=(r_0,\xi,\frac{b}{2})$ with $\gcd(r_0,\xi,b)=2$.
We set $\beta=\gamma+sH+\mu$ ($\mu \in H^\perp$)
and assume that $\mu$ is sufficiently close to $0$. 
\begin{defn}
For a Mukai vector $v$, we set
$\epsilon=1,2$ according as $\rk v$ is odd or even. 
For a stable object $E$, $\langle v(E)^2 \rangle \geq -\epsilon$. 
\end{defn}
 
Let $\sigma_{(\gamma+sH,tH)}$ be a family of stability condition
associated to ${\frak C}$.
\begin{NB}
We shall study the condition
$(\star 1)$ in \cite[sect. 3.1]{MYY:2011:1}.
\begin{defn}
For $\beta \in \NS(X)_{\Bbb Q}$ and a nef and big divisor $H$ on $X$,
we define $d_\beta(v), a_\beta(v) \in {\Bbb Q}, 
D_\beta(v) \in \NS(X)_{\Bbb Q} \cap H^\perp$ as
$$
v=e^\beta(\rk(v)+d_\beta(v)H+D_\beta(v)+a_\beta(v) \varrho_X).
$$ 
\end{defn}
Then the condition 
$(\star 1)$ for $v$ is the following.
\begin{enumerate}
\item[(1)]
$\rk v \geq 0$, $d_\beta(v)>0$ and
\item[(2)]
$d_\beta(v) \rk(v_1)-d_\beta(v_1)\rk(v)>0$ implies
\begin{equation}\label{eq:star1}
(d_\beta(v) \rk(v_1)-d_\beta(v_1)\rk(v)) 
\frac{(\omega^2)}{2}-(d_\beta(v) a_\beta(v_1)-
d_\beta(v_1)a_\beta(v))>0
\end{equation}
for any $v_1 \in H^*(X,{\Bbb Z})_{\alg}$ with
$0<d_\beta(v_1)<d_\beta(v)$ and $\langle v_1^2 \rangle \geq -\epsilon$.
\end{enumerate}

\begin{NB2}
\begin{rem}
In \cite{MYY:2011:1}, we assume that $\langle v_1^2 \rangle -(D_\beta(v_1)^2) 
\geq -\epsilon$. 
However it is sufficient to $\langle v_1^2 \rangle \geq -\epsilon$
for the proof of \cite[Prop. 3.2.1]{MYY:2011:1}. 
\end{rem}
\end{NB2}
Applying this condition to $v_0$, 
we shall study the Gieseker chamber for $v_0$.
Let
$$
v_1:=e^\gamma(r_1+d_1 H+D_1+a_1 \varrho_X),\; D \in H^\perp
$$
be a Mukai vector such that $\langle v_1^2 \rangle \geq -\epsilon$,
where
$$
d_1=d_\gamma(v_1)=\frac{(c_1(v_1)-r_1 \gamma,H)}{(H^2)} 
\in \frac{\delta}{r_0}{\Bbb Z}
$$
and 
\end{NB}
We set
$$
\delta:=\frac{1}{(H^2)} \min \{(D,H)>0 \mid D \in \NS(X) \}.
$$

\begin{NB}
\begin{NB2}
Assume that $Z_{(\gamma+sH,tH)}(v_1)=0$ and $s \ne 0$.
Then 
$s=d_1/r_1$ and 
$a_1-\frac{d_1^2}{2 r_1}(H^2)+r_1 \frac{t^2}{2}(H^2)$.
Hence
$a_1=\frac{r_1}{2}(s^2+t^2)(H^2)$.
Since $d_1^2(H^2)-2r_1 a_1 \geq \langle v_1^2 \rangle=-\epsilon$,
$r_1^2 t^2 (H^2) \leq 2$.
Since $s^2=d_1^2/r_1^2 \geq \frac{\delta^2}{r_0^2 r_1^2}$,
we get 
$s^2 \geq  \frac{\delta^2}{r_0^2 }t^2 \frac{(H^2)}{2}$. 
\end{NB2}

\begin{NB2}
We note that
\begin{equation}
\begin{split}
d_{\gamma+sH}(e^\gamma)=& -s,\; a_{\gamma+sH}(e^\gamma)=\frac{s^2}{2}(H^2),\\
d_{\gamma+sH}(v_1)=& d_1-r_1 s,\; a_{\gamma+sH}(v_1)=a_1-sd_1 (H^2)
+r_1\frac{s^2}{2}(H^2).
\end{split}
\end{equation}
Assume that
\begin{equation}
0 \leq (d_{\gamma+sH}(e^\gamma)r_1-d_{\gamma+sH}(v_1))\frac{t^2(H^2)}{2}
-(d_{\gamma+sH}(e^\gamma)a_{\gamma+sH}(a_1)-
d_{\gamma+sH}(v_1)a_{\gamma+sH}(e^\gamma).
\end{equation}
\end{NB2}

We write 
\begin{equation}
\begin{split}
e^\gamma=& e^{\gamma+sH}(r'+d'H+a' \varrho_X)\\
v_1=& e^{\gamma+sH}(r_1'+d_1' H+D_1'+a_1' \varrho_X),
\end{split}
\end{equation}
where
\begin{equation}
\begin{split}
r'=1,\;d'=& -s,\; a'=\frac{s^2}{2}(H^2),\\
r_1'=r_1,\;d_1'=& d_1-r_1 s,\; a_1'=a_1-sd_1 (H^2)
+r_1\frac{s^2}{2}(H^2).
\end{split}
\end{equation}
Assume that 
\eqref{eq:star1} does not hold for $v_1$. Thus
$d' r_1-d_1'>0$, $r_0 d'>d_1'>0$, $\langle v_1^2 \rangle \geq -2$ and 
\begin{equation}\label{eq:star}
0 \geq (d' r_1-d_1')\frac{t^2(H^2)}{2}-(d' a_1'-d_1' a').
\end{equation}
Then $-d_1=d' r_1-d_1'>0$,
$r_1>d_1'/d'>0$ and 
\begin{equation}
\begin{split}
(d' r_1-d_1')\frac{t^2(H^2)}{2}-(d' a_1'-d_1' a')
=& -d_1 \frac{t^2(H^2)}{2}-(-s a_1+\frac{d_1}{2}s^2 (H^2))\\
=& -d_1\frac{(H^2)}{2}(s^2+t^2)+a_1 s.
\end{split}
\end{equation}
Since $d'=-s>0$ and $-d_1>0$, 
\eqref{eq:star} implies $a_1>0$.
Hence $-\langle e^\gamma,v_1 \rangle>0$.
By $r_0 d'>d_1'>0$,
$s$ satisfies
\begin{equation}\label{eq:r_1}
s \leq \frac{d_1}{r_1},\;
(r_1-r_0)s>d_1.
\end{equation}
\begin{NB2}
By $\langle v_1^2 \rangle \geq -2$ and
$-r_1 s>-d_1$,
$0<r_1 a_1 \leq d_1^2 (H^2)<r_1^2 s^2 (H^2)$.
\end{NB2}
Let $C_{v_1}$ be the circle defined by
\begin{equation}
s^2-\frac{2}{(H^2)}\frac{a_1}{d_1}s+t^2=
\left(s-\frac{1}{(H^2)}\frac{a_1}{d_1}\right)^2+t^2-
\left(\frac{1}{(H^2)}\frac{a_1}{d_1}\right)^2=0.
\end{equation}
Then \eqref{eq:star} implies $(s,t)$ is in the circle $C_{v_1}$.
\end{NB}

\begin{defn}
\begin{enumerate}
\item[(1)]
Let ${\frak E}_\epsilon$ $(\epsilon=1,2)$ be the set of Mukai vectors 
$$
v_1=e^\gamma(r_1+d_1 H+D_1+a_1 \varrho_X),\; D_1 \in H^\perp
$$
such that
$v_1 \in \Delta(X)$,
$\langle v_1^2 \rangle=-\epsilon$ and
$r_1>0,d_1<0,a_1>0$.
\item[(2)]
For $v_1 \in {\frak E}={\frak E}_1 \cup {\frak E}_2$,
we set
$$
f_{v_1}(s):=
\begin{cases}
\sqrt{\frac{2}{(H^2)}\frac{a_1}{d_1}s-s^2}, & s \in 
[\frac{2}{(H^2)}\frac{a_1}{d_1},\frac{d_1}{r_1}],\\
0 & \text{otherwise}.  
\end{cases}
$$
\item[(3)]
We set
$$
f(s):=\max_{v_1 \in {\frak E}} f_{v_1}(s).
$$
\end{enumerate}
\end{defn}
\begin{NB}
By Lemma \ref{lem:d_1} (3) and \eqref{eq:r_1},
$v_1$ satisfies $d_1'<d' r_0$ if $f_{v_1}(s)>0$.

\begin{lem}\label{lem:d_1}
Assume that $r_1>0$, $d_1<0$ and $a_1>0$.
\begin{enumerate}
\item[(1)]
$$
\frac{2}{(H^2)}\frac{a_1}{d_1} \geq \frac{d_1}{r_1} 
\Longleftrightarrow 
\langle v_1^2 \rangle-(D_1^2) \geq 0.
$$
\item[(2)]
If $d_1^2 (H^2) \geq \epsilon$, then
$$
\frac{1}{(H^2)}\frac{a_1}{d_1} \geq \frac{d_1}{r_1}.
$$
\begin{NB2}
If $r_1 \geq 2r_0$, then
$r_1 a_1 \geq 2r_0 a_1 \geq 2$.
Hence the claim also holds.
\end{NB2}
\item[(3)]
If $r_1>r_0$, then
$$
\frac{2}{(H^2)}\frac{a_1}{d_1} \geq \frac{d_1}{r_1 -r_0}.
$$ 
\end{enumerate}
\end{lem}
For a proof of the claims, see \cite[Lem. 1.7]{FM-duality}. 

\begin{proof}
We note that
$\langle v_1^2 \rangle-(D_1^2) \geq -\epsilon$ and
$\langle v_1^2 \rangle-(D_1^2) <0$ implies 
$\langle v_1^2 \rangle=-\epsilon$.
\begin{equation}
\frac{2}{(H^2)}\frac{a_1}{d_1}- \frac{d_1}{r_1}
=\frac{2a_1 r_1-d_2^2(H^2)}{r_1 d_1(H^2)}
=\frac{\langle v_1^2 \rangle-(D_1^2)}{r_1 (-d_1)(H^2)}. 
\end{equation}
\begin{equation}
\frac{1}{(H^2)}\frac{a_1}{d_1}- \frac{d_1}{r_1}
=\frac{a_1 r_1-d_1^2(H^2)}{r_1 d_1(H^2)}
=\frac{(\langle v_1^2 \rangle-(D_1^2))+d_1^2(H^2)}{2r_1 (-d_1)(H^2)}. 
\end{equation}
Hence (1), (2) hold.
If $r_1>r_0$, then
\begin{equation}
\frac{2}{(H^2)}\frac{a_1}{d_1}- \frac{d_1}{r_1 -r_0}
=\frac{\langle v_1^2 \rangle-(D_1^2)+2r_0 a_1}{(r_1-r_0) (-d_1)(H^2)}. 
\end{equation}
Since $r_0 a_1 \in {\Bbb Z}$,
$r_0 a_1 \geq 1$.
Hence $\langle v_1^2 \rangle-(D_1^2)+2r_0 a_1 \geq 0$.
Thus (3) holds.
\end{proof}
\end{NB}

The following result characterize the Gieseker chamber
for $v_0$.
\begin{prop}[{\cite[Prop. 1.11]{FM-duality}}]
Assume that $s$ is rational.
\begin{enumerate}
\item[(1)]
If $t>f(s)$, then 
$M_{(\gamma+sH,tH)}(r_0 e^\gamma)={M}_H^\gamma(r_0 e^\gamma)$.
\item[(2)]
If $t<f(s)$, then all $E \in {M}_H^\gamma(r_0 e^\gamma)$
are not $\sigma_{(\gamma+sH,tH)}$-semi-stable.
\end{enumerate}
\end{prop}

\begin{NB}
We first assume that $d_1^2(H^2) \geq \epsilon$.
By Lemma \ref{lem:d_1} and $d_1^2(H^2)-2r_1 a_1 \geq -\epsilon$,
we get
\begin{equation}\label{eq:f(d/r)}
f_{v_1}(s) \leq f_{v_1}(d_1/r_1)=
\sqrt{\frac{2a_1 r_1-d_1^2(H^2)}{(H^2)r_1^2}}
\leq \frac{1}{r_1}\sqrt{\frac{\epsilon}{(H^2)}}.
\end{equation}
We next assume that $d_1^2(H^2)<\epsilon$.
Then $r_1 a_1 \leq \frac{d_1^2(H^2)}{2}+\epsilon/2<\epsilon$.
Hence $a_1 \leq \frac{\epsilon}{r_1}$.
Then 
$$
\frac{1}{(H^2)}\frac{a_1}{-d_1} \leq 
\frac{1}{(H^2)}\frac{r_0}{\delta}a_1 \leq \frac{\epsilon r_0}{(H^2)\delta}.
$$
\end{NB}
In the same way as in \cite{FM-duality}, we get the following results.
\begin{lem}[{\cite[Lem. 1.13]{FM-duality}}]\label{lem:estimate1}
$$
f_{v_1}(s) \leq \max \left\{\sqrt{\frac{\epsilon}{(H^2)}},
\sqrt{-\frac{4r_0}{(H^2)\delta}s-s^2} \right\}.
$$
\begin{NB}
$$
\sqrt{\frac{2a_1}{(H^2)d_1}s-s^2} \leq
\sqrt{-\frac{4r_0}{(H^2)\delta}s-s^2}.
$$
\end{NB}
\end{lem}

\begin{NB}
In a neighborhood of $s=0$, we have a more precise estimate. 
We first note that
$d_1^2(H^2)-2r_1 a_1 \geq -\epsilon$ and $r_0 a_1 \in {\Bbb Z}$
imply that
$d_1^2(H^2)+\epsilon \geq 2\frac{r_1}{r_0}$.
Hence
\begin{equation}\label{eq:circle-a}
\left(\frac{d_1}{r_1}\right)^2+
\left(\sqrt{\frac{\epsilon}{(H^2)}}\frac{1}{r_1}-
\frac{1}{r_0} \sqrt{\frac{1}{\epsilon(H^2)}}\right)^2
\geq \left(\frac{1}{r_0} \sqrt{\frac{1}{\epsilon(H^2)}}\right)^2.
\end{equation}
We also have $f_{v_1}(\frac{d_1}{r_1}) \leq 
\frac{1}{r_1}\sqrt{\frac{\epsilon}{(H^2)}}$ by 
$d_1^2(H^2)-2r_1 a_1 \geq -\epsilon$. 
\end{NB}

\begin{prop}[{\cite[Prop. 1.14]{FM-duality}}]\label{prop:estimate2}
Assume that
\begin{equation}\label{eq:|s|}
|s| \leq \min \left\{\frac{1}{r_0}\sqrt{\frac{1}{\epsilon(H^2)}},
\frac{\delta}{\epsilon r_0^2} \right\}.
\end{equation}
Then
\begin{equation}\label{eq:f}
f_{v_1}(s) \leq 
\frac{1}{r_0}\sqrt{\frac{1}{\epsilon(H^2)}}
-\sqrt{\frac{1}{r_0^2}\frac{1}{\epsilon(H^2)}-s^2}
\end{equation}
for all $v_1 \in {\frak E}_\epsilon$.
In particular if
\begin{equation}\label{eq:|s|-2}
|s| \leq \min \left\{\frac{1}{r_0}\sqrt{\frac{1}{2(H^2)}},
\frac{\delta}{2 r_0^2} \right\},
\end{equation}
then
\begin{equation}\label{eq:f2}
f(s) \leq 
\frac{1}{r_0}\sqrt{\frac{1}{2(H^2)}}
-\sqrt{\frac{1}{r_0^2}\frac{1}{2(H^2)}-s^2}.
\end{equation}
\end{prop}

\begin{NB}
\begin{proof}
Since the RHS of \eqref{eq:f} is non-negative,
we may assume that
$|s| \geq |\frac{d_1}{r_1}|$.
If $s$ satisfies \eqref{eq:|s|}, 
then
$\frac{\delta}{\epsilon r_0^2} \geq |s| \geq 
|\frac{d_1}{r_1}| \geq \frac{\delta}{r_0 r_1}$
implies $r_1 \geq \epsilon r_0$.
\begin{NB2}
Thus $\sqrt{\frac{2}{(H^2)}}\frac{1}{r_1}-
\frac{1}{r_0} \sqrt{\frac{1}{2(H^2)}} \leq 0$.
\end{NB2}
Then 
$$
d_1^2(H^2) \geq 2r_1 a_1-\epsilon \geq 
2 \epsilon r_0 a_1-\epsilon \geq \epsilon.
$$
By Lemma \ref{lem:d_1} (2),
 $s  \leq \frac{d_1}{r_1} \leq \frac{a_1}{d_1(H^2)}$.
Hence \eqref{eq:f(d/r)} holds and 
$(\frac{d_1}{r_1},f_{v_1}(\frac{d_1}{r_1}))$ satisfies
\eqref{eq:circle-a} with $f_{v_1}(\frac{d_1}{r_1}) 
\leq \frac{1}{r_0}\sqrt{\frac{1}{\epsilon(H^2)}}$.
Therefore
\begin{equation}
f_{v_1}(s) \leq 
\frac{1}{r_0}\sqrt{\frac{1}{\epsilon (H^2)}}
-\sqrt{\frac{1}{r_0^2}\frac{1}{\epsilon(H^2)}-s^2}.
\end{equation}
Since
\begin{equation}
\frac{1}{r_0}\sqrt{\frac{1}{2 (H^2)}}
-\sqrt{\frac{1}{r_0^2}\frac{1}{2(H^2)}-s^2}
 \geq
 \frac{1}{r_0}\sqrt{\frac{1}{(H^2)}}
-\sqrt{\frac{1}{r_0^2}\frac{1}{(H^2)}-s^2}
\end{equation}
for $|s|<\sqrt{\frac{1}{r_0^2}\frac{1}{2(H^2)}}$,
we get then claim.
\end{proof}
\end{NB}

\subsection{A bound on the Gieseker chamber}\label{subsect:Gieseker}
We set
$$
s_0:=\min \left\{\frac{1}{r_0}\sqrt{\frac{1}{2(H^2)}},
\frac{\delta}{2r_0^2} \right\}.
$$
By the description of ${\frak E}$ and Corollary \ref{cor:limit},
we get the following.
\begin{prop}\label{prop:geometric-chamber}
Assume that
\begin{equation}
\begin{split}
0 & <|s|<  s_0,\;
t  > \frac{1}{r_0}\sqrt{\frac{1}{2(H^2)}}-
\sqrt{\frac{1}{2r_0^2 (H^2)}-s^2}.
\end{split}
\end{equation}
\begin{NB}
Assume that
\begin{equation}\label{eq:|s|-3}
\begin{split}
0<|s| & \leq \min 
\left\{\frac{1}{r_0}\sqrt{\frac{1}{2(H^2)}},
\frac{\delta}{2 r_0^2} \right\},\\
t & \geq 
\frac{1}{r_0}\sqrt{\frac{1}{2(H^2)}}
-\sqrt{\frac{1}{r_0^2}\frac{1}{2(H^2)}-s^2}.
\end{split}
\end{equation}
\end{NB}
Then
${\cal O}_x$ $(x \in X)$ is $\sigma_{(\gamma+sH,tH)}$-semi-stable
such that all stable factors are irreducible objects of ${\frak C}$.
\end{prop}

Let $v=e^\beta(r+D+a \varrho_X)$, $(D \in H^\perp)$
be a Mukai vector.
We set
$$
p_0:=\frac{\frac{2}{(H^2)}+s_0^2+\frac{\langle v^2 \rangle -(D^2)}
{r^2(H^2)}}{2s_0}. 
$$
Let $V_v(X)$ be the open subset 
defined by
\begin{equation}
\begin{split}
& s \ne 0,\; t^2+(|s|-p_0)^2>p_0^2
-\frac{\langle v^2 \rangle -(D^2)}{r^2(H^2)},\\
& t  \geq \frac{1}{r_0}\sqrt{\frac{1}{2(H^2)}}-
\sqrt{\frac{1}{2r_0^2 (H^2)}-s^2}.
\end{split}
\end{equation}
If $t>\sqrt{\frac{2}{(H^2)}}$,
then $Z_{(\gamma+sH,tH)}(u) \not \in {\Bbb R}_{\leq 0}$
for $u \in \Delta(X)$ with $\rk u>0$.

Assume that $(\gamma+s_1 H,t_1 H)$ belongs to a Gieseker
chamber for a Mukai vector $v$, that is,
${\cal M}_{(\gamma+s_1 H,t_1 H)}(v)={\cal M}_H^{\gamma}(v)$.
We set
$$
p:=\min \left\{\frac{t_1^2+s_1^2+\frac{\langle v^2 \rangle -(D^2)}{r^2(H^2)}}
{2s_1},-p_0
\right\}.
$$
\begin{NB}
We set
$$
p':=\min \left\{\frac{t_0^2+s_0^2+\frac{\langle v^2 \rangle -(D^2)}{r^2(H^2)}}
{2s_0},
\frac{\frac{1}{r_0^2(H^2)}+\frac{\langle v^2 \rangle -(D^2)}{r^2(H^2)}}
{-2\sqrt{\frac{1}{2r_0^2(H^2)}}} \right\}.
$$
The circle 
$$
t^2+(s-p')^2 = {p'}^2-\frac{\langle v^2 \rangle -(D^2)}{r^2(H^2)}
$$
and 
$$
\left(t - \frac{1}{r_0}\sqrt{\frac{1}{2(H^2)}} \right)^2-
\left(\frac{1}{2r_0^2 (H^2)}-s^2 \right)=0
$$
passes $(-\sqrt{\frac{1}{2r_0^2(H^2)}},\sqrt{\frac{1}{2r_0^2(H^2)}})$.
If $p'$ decreases, then the $t$-coordinate of the intersection 
increase.
\end{NB}

Assume that 
${\cal M}_{(\gamma+s_1' H,t_1' H)}(v)=
\{E \mid E^{\vee} \in {\cal M}_H^{-\gamma}(v^{\vee}) \}$.
We set
$$
p':=\max \left\{
\frac{{t_1'}^2+{s_1'}^2+\frac{\langle v^2 \rangle -(D^2)}{r^2(H^2)}}{2s_1'},
p_0
\right\}.
$$
By \cite[Cor. 3.2.10]{MYY:2011:1} and \cite[Cor. 3.6]{Y:wall}, 
we get the following.
\begin{prop}\label{prop:Gieseker}
Let $X$ be a K3 surface or an Enriques surface.
Let $v=e^\beta(r+D+a \varrho_X)$, $(D \in H^\perp)$
be a Mukai vector.
\begin{enumerate}
\item[(1)]
Assume that $(s,t)$ satisfies
$$
s<0,\;
t^2+(s-p)^2 > p^2-\frac{\langle v^2 \rangle -(D^2)}{r^2(H^2)},\;
t  \geq \frac{1}{r_0}\sqrt{\frac{1}{2(H^2)}}-
\sqrt{\frac{1}{2r_0^2 (H^2)}-s^2}.
$$
Then 
${\cal M}_{(\gamma+s H,t H)}(v)^{ss}={\cal M}_H^{\gamma}(v)^{ss}$.
\item[(2)]
Assume that $(s,t)$ satisfies
$$
s>0,\;t^2+(s-p')^2 > {p'}^2-\frac{\langle v^2 \rangle -(D^2)}{r^2(H^2)},
t  \geq \frac{1}{r_0}\sqrt{\frac{1}{2(H^2)}}.
$$
Then 
${\cal M}_{(\gamma+s H,t H)}(v)=
\{E \mid E^{\vee} \in {\cal M}_H^{-\gamma}(v^{\vee}) \}$.
\end{enumerate}
\end{prop}

\begin{NB}
Then for $(s,t)$ satisfying $s>0$ and
$$
t^2+(s-p')^2 > {p'}^2-\frac{\langle v^2 \rangle -(D^2)}{r^2(H^2)},
$$
we have 
${\cal M}_{(\gamma+s H,t H)}(v)=
\{E \mid E^{\vee} \in {\cal M}_H^{-\gamma}(v^{\vee}) \}$.
\end{NB}

\begin{NB}
Assume that $(\beta',\omega')$ belongs to a chamber for
$v$.
Let $v_0$ be an isotropic Mukai vector and assume that
$\phi_{(\beta',\omega')}(v_0)=\phi_{(\beta',\omega')}(v)$.
Let ${\cal C}$ be a chamber such that
$\sigma_{(\beta',\omega')} \in \overline{\cal C}$.
Let $X'$ be the moduli of ${\cal C}$-stable objects 
$E$ with $v(E)=v_0$.
Then $X'$ is a smooth K3 surface.
Let ${\cal E}$ be a universal family of 
${\cal C}$-stable objects.
Then 
the Fourier-Mukai transform 
$\Phi_{X \to X'}^{{\cal E}^{\vee}}:{\bf D}(X) \to {\bf D}(X')$
induces an isomorphism
$\Phi:\Stab(X) \to \Stab(X')$.
We set ${\cal C}':=\Phi({\cal C})$.
Then $(\beta,\omega)=\Phi((\beta',\omega'))$
satisfies $\sigma_(\beta,\omega) \in {\cal C}'$
and $\phi_{(\beta,\omega)}(w)=1$
for $w=\Phi_{X \to X'}^{{\cal E}^{\vee}}(v)$.
Thus $(c_1(w)-(\rk w) \beta,\omega)=0$.
If $(\beta+sH,tH) \in {\cal C}'$ for a $0>s \gg -1$, then
$(\beta+sH,tH)$ belongs to the Gieseker chamber
of $w$.
Hence $\Phi_{X \to X'}^{{\cal E}}$ induces
an isomorphism
${\cal M}_{(\beta',\omega')}(v) \cong {\cal M}_{(\beta,\omega)}(w) \cong
{\cal M}_H(w)$.
\end{NB}

\subsection{An isomorphism by a Fourier-Mukai transform}
\label{subsect:application}

We consider a family of stability conditions
$P_{\gamma,H}$ in \eqref{eq:Plane}.
Let $v_0:=r_0 e^\gamma$ be a primitive isotropic Mukai
vector such that
$\dim M_{(\beta,\omega)}(v_0)=2$ for a general
$(\beta,\omega)$.
Let $C_v \ne \emptyset$ be a semi-circle defined by
${\Bbb R}_{>0}Z_{(\gamma+sH,tH)}(r_0 e^\gamma)
={\Bbb R}_{>0}Z_{(\gamma+sH,tH)}(v)$,
where $v$ is a Mukai vector with $\langle v^2 \rangle \geq 0$.
Since $\langle v^2 \rangle,\langle v_0^2 \rangle \geq 0$,
\cite[(5.10)]{MYY:2011:2} implies
$\langle v,v_0 \rangle>0$.

Let $(\gamma+s_0 H,t_0 H)$ be a point of $C_v$ and
$U$ a neighborhood of $(\gamma+s_0 H,t_0 H)$.
Let $U_\pm$ be the connected components of $U \setminus C_v$
such that
$$
U_\pm \subset 
\{(\gamma+sH,tH) \mid 
\pm(\phi_{(\gamma+sH,tH)}(v)-\phi_{(\gamma+sH,tH)}(v_0))>0 \}.
$$
For chambers ${\cal C}_\pm$ 
with $U_\pm \subset \overline{{\cal C}_\pm}$,
we consider moduli schemes $X':=M_{(\beta,\omega)}(r_0 e^\gamma)$
($(\beta,\omega)\in {\cal C}_\pm$).  Then 
$X'$ is a K3 surface or an Enriques surface.
Let  ${\cal E}_\pm \in {\bf D}(X \times X')$ be universal families
(as twisted objects).  
Let 
$$
\Phi_\pm:=\Phi_{X \to X'}^{{\cal E}_\pm^{\vee}[2]}
:{\bf D}(X) \to {\bf D}(X')
$$ 
be a Fourier-Mukai transform in \eqref{eq:FM}.
We use the notation in subsection \ref{subsect:half-plane}.
Then $\Phi_\pm(C_v)$ is the line defined by $s'=s_0'$,
where $\Phi(\gamma+s_0 H,t_0 H)=(\gamma'+s_0' H',t_0' H')$. 
We set $U_\pm':=\Phi_\pm(U_\pm)$.
By shrinking $U$, we may assume that 
there is no $(\gamma'+s' H',t' H') \in U_\pm'$
such that $Z_{(\gamma'+s'H',t'H')}(u) \not \in {\Bbb R}_{\leq 0}$
for a Mukai vector $u \in \Delta(X)$ with $\rk u>0$. 
\begin{NB}
We take $(\beta,\omega)$ in a chamber ${\cal C}$
such that $(\gamma+s_0 H,t_0 H) \in \overline{\cal C}$.
Let $W$ be a wall for $e^\gamma$ defined by $u$.
Then $W$ contains the plane $P_{\gamma,H}$ or 
$W \cap P_{\gamma,H}$
is a semi-circle passing $(\gamma+s_0 H,t_0 H)$.
\end{NB}
Then $\sigma_{(\gamma'+s'H',t'H')}$ is the stability
condition associated to a category of perverse coherent
sheaves if $(\gamma'+s'H',t' H')  \in  U_\pm'$. 

For $E \in {\cal M}_{(\gamma+s_0 H,t_0 H)}(v)$,
we set $F:=\Phi_{X \to X}^{{\cal E}_\pm^{\vee}[2]}(E)$.
Then $\Phi_\pm(C_v)$ is 
$$
((c_1(F)-(\rk F) (\gamma_0'+s' H'))\cdot H')=0.
$$
We note that $\rk F=-\langle v,v_0 \rangle<0$.
Then we have the following.
\begin{enumerate}
\item
$\sigma_{(\gamma'_0+s' H',t' H')} \in U_+'$
if and only if
$s'<0$.
\item
$\sigma_{(\gamma'+s' H',t' H')} \in U_-'$
if and only if
$s'>0$.
\end{enumerate}
Applying Proposition \ref{prop:Gieseker},
we get a generalization of \cite[Thm. 1.2]{MYY:2011:2}.

\begin{NB}
Since $U_- \subset \overline{\cal C}'$ for another chamber ${\cal C}'$,
we can take ${\cal C}$ such that
$\Phi(U_+) \subset \{s' \mid s'<c\}$.
Then $\sigma_{(\gamma'+s_0' H',t_0' H')}$ is the category
defined in \cite{MYY:2011:1}.
\end{NB}

\begin{thm}\label{thm:isom}
Let $X$ be a K3 surface or an Enriques surface over $k$.
Let $v$ be a Mukai vector
with $\langle v^2 \rangle \geq 0$.
\begin{enumerate}
\item[(1)]
If $(\gamma+sH,tH) \in U_+$, then we have an isomorphism
\begin{equation}
\begin{matrix}
{\cal M}_{(\gamma+sH,tH)}(v)& \to & {\cal M}_{H'}^{\gamma'}(v')\\
E & \mapsto & \Phi_{X \to X}^{{\cal E}_+^{\vee}[1]}(E).
\end{matrix}
\end{equation}
\item[(2)]
If $(\gamma+sH,tH) \in U_-$, then
we have an isomorphism
\begin{equation}
\begin{matrix}
{\cal M}_{(\gamma+sH,tH)}(v) & \to & 
{\cal M}_{H'}^{-\gamma'}({v'}^{\vee})\\
E & \mapsto & (\Phi_{X \to X}^{{\cal E}_-^{\vee}[1]}(E))^{\vee}.
\end{matrix}
\end{equation}
\end{enumerate}
\end{thm}

The following result is a slight generalization of
\cite[Thm. 7.6]{N2}.
\begin{thm}\label{thm:Enriques}
Let $X$ be a classical Enriques surface and
$v$ be a Mukai vector with $\langle v^2 \rangle \geq 0$. 
For a general $\sigma$ with respect to
$v$, there is a nef and big divisor $H$ and $\gamma \in \NS(X)_{\Bbb Q}$
such that  
${\cal M}_\sigma(v)$ is isomorphic to ${\cal M}_H^\beta(w)$.
In particular, there is a projective moduli space  $M_\sigma(v)$
of $S$-equivalence classes of $\sigma$-semi-stable objects
$E$ with $v(E)=v$.
\end{thm}

\begin{proof}
We may assume that $\sigma=\sigma_{(\beta_0,\omega_0)}$
($(\beta_0,\omega_0) \in \NS(X)_{\Bbb R} \times \Amp(X)_{\Bbb R}$)
and $\Ima Z_{(\beta_0,\omega_0)}(v)>0$.
We take $E \in {\cal M}_{(\beta_0,\omega_0)}(v)$.
By perturbing $(\beta_0,\omega_0)$,
we can find a primitive and isotropic Mukai vector 
$v_0:=lr e^{\xi/r}$ such that $r$ is odd and 
\begin{equation}\label{eq:phi=1}
(\beta_0,\omega_0) \in W:=\{(\beta,\omega) \mid \phi_{(\beta,\omega)}(E)=
\phi_{(\beta,\omega)}(v_0) \}.
\end{equation}
\begin{NB}
$l$ may be negative.
\end{NB}
We set $r_0:=lr$ and $\gamma:=\xi/r$.
Applying Theorem \ref{thm:isom},
we get our claim.
\begin{NB}
By perturbing $(\beta_0,\omega_0)$ with
the condition \eqref{eq:phi=1},
we may assume that $(\beta_0,\omega_0)$ is general with respect to
$v_0$ or   
$W$ is the unique wall for $v_0$ in a neighborhood of $(\beta_0,\omega_0)$.
Then $\omega_0'$ is ample. 
We set 
$$
{\cal C}_\pm:=\{(\beta,\omega) \mid \pm(\phi_{(\beta,\omega)}(v)-
\phi_{(\beta,\omega)}(v_0))>0 \}.
$$
Then 
${\cal C}_\pm$ are chamber for $v_0$ such that
$\sigma_{(\beta_0,\omega_0)} \in \overline{{\cal C}_\pm}$.
By Proposition \ref{prop:FM},
 there is a complex ${\cal E}_\pm$ on $X \times X$
such that  
$({\cal E}_{\pm})_{ |X \times \{ x \}}$ is a 
${\cal C}_\pm$-stable complexes
with the Mukai vector $v_0$ and ${\cal E}_\pm$ is the universal family.
Let
\begin{equation}
\Phi_\pm:=
\Phi_{X \to X'}^{{\cal E}_\pm^{\vee}[2]}:{\bf D}(X) \to {\bf D}(X)
\end{equation}
be the Fourier-Mukai transform whose kernel is
${\cal E}_\pm^{\vee}[2]$.
It induces an isomorphism
$\Phi_\pm:\Stab(X) \to \Stab(X)$.
We set ${\cal C}_\pm':=\Phi_\pm({\cal C}_\pm)$.
For $\sigma_{(\beta,\beta)} \in U(X)$,
we set
\begin{equation}\label{eq:Phi}
e^{\beta'+i \omega'}
:=\frac{\Phi_\pm(e^{\beta+i \omega})}
{-\langle e^{\beta+i \omega},v_0 \rangle}.
\end{equation}
Then we have two family of stability conditions
$\tau_{\pm,(\beta',\omega')}=\Phi_\pm(\sigma_{(\beta,\omega)})$ 
whose central charge is 
$\langle e^{\beta'+i \omega'},\bullet \rangle$.
${\cal O}_x$ is $\tau_{\pm,(\beta',\omega')}$-stable
for $(\beta',\omega') \in {\cal C}_\pm'$.
We set $\tau_{\pm,(\beta_0',\omega_0')}:
=\Phi_\pm(\sigma_{(\beta_0,\omega_0)})$.
Since $(\beta_0,\omega_0) \in \overline{{\cal C}_\pm}$,
$\omega_0' \in \overline{\Amp(X)_{\Bbb R}}$.
We have $\Phi_\pm(\sigma_{(\beta_0,\omega_0)}) 
\in \overline{{\cal C}_\pm'}$
and $\phi_{(\beta_0',\omega_0')}(F)=1$
for $F=\Phi_{X \to X}^{{\cal E}_\pm^{\vee}[2]}(E)$.
Thus $(c_1(F)-(\rk F) \beta_0',\omega_0')=0$.
Replacing $(\beta_0,\omega_0)$, we may assume that
$\beta_0' \in \NS(X)_{\Bbb Q}$ and
$\omega_0' \in \Amp(X)_{\Bbb Q}$.
Let $U$ be a small neighborhood of $(0,1)$.
For the subfamily of stability conditions parameterized by
$$
\{(\beta_0'+s\omega_0',t \omega_0') \mid (s,t) \in U \}, 
$$
the condition 
$\phi_{(\beta',\omega')}(F)=1$
is $s=0$. 
We note that $\rk F=-\langle v,v_0 \rangle<0$.
Then we have the following.
\begin{enumerate}
\item
$\tau_{+,(\beta_0+s \omega_0',t \omega_0')} \in {\cal C}_+'$
if and only if
$s<0$.
\item
$\tau_{-,(\beta_0+s \omega_0',t \omega_0')} \in {\cal C}_-'$
if and only if
$s>0$.
\end{enumerate}
For the case (i),
$\Phi_+(E)[-1]$ is a $\beta_0'$-semi-stable object with respect to
$\omega_0'$ and
we have an isomorphism
${\cal M}_{(\beta_0,\omega_0)}(v)^{ss} \cong 
{\cal M}_{\omega_0'}^{\beta_0'}(w)^{ss}$,
where $w=v(\Phi_+(E)[-1])$.

For case (ii),
$(\Phi_-(E)[-1])^{\vee}$ is a $(-\beta_0')$-semi-stable 
object with respect to
$\omega_0'$ and
we have an isomorphism
${\cal M}_{(\beta_0,\omega_0)}(v)^{ss} \cong 
{\cal M}_{\omega_0'}^{-\beta_0'}(w^{\vee})^{ss}$,
where $w=v(\Phi_-(E)[-1])$.
\begin{NB2}
For $\Phi_+$,
$\phi_{(\beta_0'+s\omega_0',t \omega_0')}(F)>1$
means $s>0$ for $\rk w>0$ and
$s<0$ for $\rk w<0$.
For $\Phi_-$,
$\phi_{(\beta_0'+s\omega_0',t \omega_0')}(F)<1$
means $s<0$ for $\rk w>0$ and
$s>0$ for $\rk w<0$.
\end{NB2}
\begin{NB2}
If $\rk w>0$, then $\phi_{(\beta_0'+s\omega_0',t \omega_0')}(\Phi_+(E)[-2])
\in (-1,0)$ and $(\Phi_+(E)[-2])^{\vee}$ is 
$(-\beta_0')$-twisted semi-stable.

If $\rk w<0$, then $\phi_{(\beta_0'+s\omega_0',t \omega_0')}(\Phi_+(E)[-1])
\in (0,1)$ and $\Phi_+(E)[-1]$ is 
$\beta_0'$-twisted semi-stable.

If $\rk w>0$, then $\phi_{(\beta_0'+s\omega_0',t \omega_0')}(\Phi_-(E))
\in (0,1)$ and $\Phi_+(E)$ is 
$\beta_0'$-twisted semi-stable.

If $\rk w<0$, then $\phi_{(\beta_0'+s\omega_0',t \omega_0')}(\Phi_-(E)[-1])
\in (-1,0)$ and $(\Phi_-(E)[-1])^{\vee}$ is 
$(-\beta_0')$-twisted semi-stable.
\end{NB2}

\begin{NB2}
If $(\beta_0'+s\omega_0',tH) \in {\cal C}'$ for a $0>s \gg -1$, then
$(\beta+sH,tH)$ belongs to the Gieseker chamber
of $w$.
Hence $\Phi_{X \to X}^{{\cal E}}$ induces
an isomorphism
${\cal M}_{(\beta',\omega')}(v) \cong {\cal M}_{(\beta,\omega)}(w) \cong
{\cal M}_H^\beta(w)$.
If $(\beta+sH,tH) \in {\cal C}'$ for a $0<s \ll 1$, then
$(\beta+sH,tH)$ belongs to the Gieseker chamber
of $w$.
Hence $\Phi_{X \to X}^{{\cal E}}$ induces
an isomorphism
${\cal M}_{(\beta',\omega')}(v) \cong {\cal M}_{(\beta,\omega)}(w) \cong
{\cal M}_H^{-\beta}(w^{\vee})$.
\end{NB2}

\end{NB}
\end{proof}  

\begin{rem}
\begin{enumerate}
\item[(1)]
In the proof of Theorem \ref{thm:Enriques},
by perturbing $(\beta_0,\omega_0)$ with
the condition \eqref{eq:phi=1},
we may assume that $(\beta_0,\omega_0)$ is general with respect to
$v_0$ or   
$W$ is the unique wall for $v_0$ in a neighborhood of $(\beta_0,\omega_0)$.
Then $\omega_0'$ is ample. Hence
we can take $H$ to be ample.
\item[(2)]
If all divisor classes on $X$ are defined, then by Remark \ref{rem:k},
the same claim holds even if $k$ is not algebraically closed.  

\end{enumerate}
\end{rem}

\begin{rem}
Assume that $X$ is a K3 surface.
Then the same proof also works.
In this case, the result was obtained in \cite{MYY:2011:2}
under the assumption $\rho(X)=1$ and in 
\cite{BM2} for the general case combining
 a classification of walls for $\varrho_X$ with 
the argument of  \cite{MYY:2011:2}.
We also remark that the same proof of \cite{BM2} also work for
the case of an Enriques surface.
\end{rem}


\begin{NB}
\begin{rem}
For case (i),
if $s=0$ is a wall and $(\beta_0',\omega_0')$ is a general point,
then there is an exact sequence
$$
0 \to (B \oplus B(K_X))^{\oplus r} \to {\cal O}_x \to T_B({\cal O}_x) \to 0
$$
or $(A_+)$ in Proposition \ref{prop:general-wall}.
Then ${\cal F}_x:=\Phi_{X \to X}^{{\cal E}_+}(T_B({\cal O}_x))$
forms a universal family of $\sigma_{(\beta_0,\omega_0)}$-stable
objects with
$\phi_{(\beta_0,\omega_0)}(E)=\phi_{(\beta_0,\omega_0)}({\cal F}_x)$.  
Hence $\Phi_{X \to X}^{{\cal F}^{\vee}[1]}(E)$ is 
Gieseker semi-stable (indeed it is $\mu$-stable, if $E$ is stable).
The relation is given by 
$$
0 \to \Phi_{X \to X}^{{\cal F}^{\vee}[1]}(E) \to \Phi_+(E)[-1] \to 
B^{\oplus r_1} \oplus B(K_X)^{\oplus r_2} \to 0. 
$$
\end{rem}
\end{NB}

\begin{NB}
We take $v=(x,\eta,z) \in {\Bbb Q}_{>0}e^{\xi/r}$ with 
$x,z \in {\Bbb Z}$ and $\eta \in \NS(X)$.
Then $(\eta^2)=2xz$.
Replacing $v$ by $2v$ if necessary,
we may assume that $x$ is even......
\end{NB}

\begin{NB}
\begin{rem}
Remark on walls on K3:

Let $X' \to X$ be the covering K3 surface.
Let $E$ be an exceptional vector bundle which defines a wall 
for $\Stab(X)$. 

If $\iota^*(E)=E$, then
$E=\pi^*(F)$.
Hence $\langle v(F)^2 \rangle=-1$.

If $\iota^*(E) \ne E$, then
$\Hom(E,\iota^*(E))=0$, which implies that
$\langle v(E), v(\iota^*(E)) \rangle \geq 0$.
We shall prove that
$\langle v(E),\iota(v(E)) \rangle=0$.

$v(E),v(\iota(E))$ spans a negative definite lattice
if and only if $\langle v(E),\iota(v(E)) \rangle=0$.
    
\end{rem}
\end{NB}

\subsection{Examples of isomorphisms}\label{subsect:example}

Assume that $X$ is a K3 surface with $\Pic(X)={\Bbb Z}H$.
Let $I_Z$ be an ideal sheaf with
$v(I_Z)=(1,0,-n)$.
Then
$\phi_{(sH,tH)}(e^{\lambda H})=\phi_{(sH,tH)}(I_Z)$
if and only if 
\begin{equation}\label{eq:I_Z}
t^2+(s-\lambda)\left(s-\tfrac{2n}{(H^2)\lambda} \right)=0.
\end{equation}
We note that
 $(-1,\sqrt{\frac{2}{(H^2)}})$
satisfies
$$
t^2+(s-\lambda)\left(s-\tfrac{2n}{(H^2)\lambda} \right) \leq 0
$$ 
if and only if 
\begin{equation}\label{eq:lambda}
(\lambda+1)\left(\lambda \tfrac{(H^2)}{2}+n \right)+\lambda \geq 0.
\end{equation}
Under this condition,
we shall consider a Fourier-Mukai transform
$\Phi_{X \to Y}^{{\cal E}^{\vee}[1]}$,
where $Y=M_{(sH,t' H)}(r_0 e^{\lambda H})$,
$t'$ is sufficiently close to $t$
and ${\cal E}$ is a universal family.
By our assumption,
$t>\sqrt{\frac{2}{(H^2)}}$ on $s=-1$, and hence
${\cal M}_{(sH,tH)}(v)={\cal M}_H(v)$.

For example, 
if $-\lambda$ is a positive integer with
$-\lambda \geq \frac{2(n+2)}{(H^2)}$ and $-\lambda \geq 2$,
then
since $f_{v_1}(s) \leq \sqrt{\frac{2}{(H^2)}}$ 
(\cite[Rem. 1.15 (1)]{FM-duality}),
$$
{\cal M}_{(sH,tH)}(-e^{\lambda H})=\{ I_x^{\vee}(\lambda H)[1] \mid x \in X\}
$$
for the point $(-1,t)$ on the semi-circle \eqref{eq:I_Z}.
Thus $Y=X$ and
$I_\Delta^{\vee}(\lambda H)[1]={\cal E}$.
Therefore $\Phi_{X \to Y}^{{\cal E}^{\vee}[1]}=
\Phi_{X \to X}^{I_\Delta(-\lambda H)}$ and
we get the following result.

\begin{prop}
Assume that
$-\lambda$ is a positive integer with
$-\lambda \geq \frac{2(n+2)}{(H^2)}$ and $-\lambda \geq 2$.
Then
$(\Phi_{X \to X}^{I_\Delta(-\lambda H)}(I_Z))^{\vee}$is a stable sheaf.
\end{prop}

\begin{proof}
We note that $\phi_{(sH,tH)}(I_x^{\vee}(\lambda H)[1])>
\phi_{(sH,tH)}(I_Z)$ on the outside of the semi-circle. 
By Theorem \ref{thm:isom} (2), we get the claim.
\end{proof}

\begin{NB}
If $-\lambda \geq \frac{2(n+2)}{(H^2)}$ and $-\lambda \geq 2$,
then $(-\lambda-1)(1+\frac{2n}{(H^2)\lambda}) \geq
\frac{4}{(H^2)}-\frac{4}{(-\lambda)(H^2)} \geq
\frac{2}{(H^2)}$ by $-\lambda \geq 2$.
\end{NB}

Assume that $(H^2)=n+2$ and $\lambda=-2$, i.e.,
$-\lambda = \frac{2(n+2)}{(H^2)}$.
\begin{NB}
If $\frac{(H^2)}{2}=-\frac{(\lambda+1)n+\lambda}{\lambda(\lambda+1)}$,
then $(-1,\sqrt{\frac{2}{(H^2)}})$ lies on the semi-circle.
If $\lambda \in {\Bbb Z}$, then we see that $\lambda=-2$.
If $\lambda=-\frac{m+1}{m}$, then if $(m+1) \mid n$, then
$(H^2)/2$ is an integer.  
\end{NB}
Then the semi-circle \eqref{eq:I_Z} passes at $(-1,\sqrt{\frac{2}{(H^2)}})$
and defines a wall for $v$.
Indeed
there is an ideal sheaf $I_Z$ fitting in
the exact sequence
\begin{equation}\label{eq:s<-1}
0 \to {\cal O}_X(-H)^{\oplus 2} \to I_Z \to  I_x^{\vee}(-2H)[1] \to 0
\end{equation}
(see \cite{Reflection}, \cite{Y:BN}).
\begin{NB}
Equation of the wall:
$
t^2+(s+2)(s+\frac{n}{n+2})=0.
$
\end{NB}
If $s<-1$, then \eqref{eq:I_Z} is a wall.
If $n \geq 2$, then
there is an ideal sheaf $I_Z$ fitting in
the exact sequence
\begin{equation}\label{eq:s>-1}
0 \to F \to I_Z \to {\cal O}_X(-H)[1] \to 0
\end{equation}
where $F \in {\cal M}_H(2,-H,2-\frac{n}{2})$,
which gives a wall for $s>-1$. 
Moreover if $n \geq 4$, then
all $I_Z$ fits in an exact sequence
$$
0 \to F' \to I_Z \to {\cal O}_X(-H)^{\oplus (\frac{n}{2}-2)}[1] \to 0
$$
where $F' \in {\cal M}_H(\frac{n}{2}-1,-(\frac{n}{2}-2)H,\frac{n^2}{4}-n-4)$.

By the Fourier-Mukai transform $I_{X \to X}^{I_\Delta (2H)}$,
we have an exact sequence from \eqref{eq:s<-1}:
$$
0 \to E_0^{\oplus 2} \to E \to {\cal O}_x^{\vee}[1] \to 0,
$$
where $E_0=I_{X \to X}^{I_\Delta (2H)}({\cal O}_X(-H)) 
\in {\cal M}_H(\frac{n}{2}+2,-H,1)$.
\begin{NB}
The action of $\Phi_{X \to X}^{I_\Delta}$ on $z=s+it \in {\Bbb H}$
is $z'=-\frac{2}{(H^2)}\frac{1}{z}$.
\end{NB}
$E^{\vee}$ is a non-locally free sheaf with 
$v(E^{\vee})=(n+4,2H,1)$. 
We would like to remark that
$I_x^{\vee}(-2H)[1]$ is properly $\sigma_{(sH,tH)}$-semi-stable
on \eqref{eq:I_Z} with $s>-1$.
Indeed 
we set $E_x:=\Phi_{X \to X}^{I_\Delta^{\vee}[2]}(I_x^{\vee}(-H))$.
Then $E_x$
is a stable locally free sheaf with
$v(E_x)=(\frac{n}{2}+1,H,1)$ and we 
have an exact triangle
\begin{equation}\label{eq:n+2}
E_x(-H) \to I_x^{\vee}(-2H)[1] \to {\cal O}_X(-H)[1]^{\oplus (\frac{n}{2}+2)} 
\overset{\varphi}{\to}
E_x(-H)[1].
\end{equation}
We shall prove that 
$E_x(-H)$ and ${\cal O}_X(-H)[1]$ are stable objects
on \eqref{eq:I_Z} with $s>-1$.
We note that $\varphi$ is the evaluation map
${\cal O}_X(-H)[1] \otimes H^0(E_x) \to E_x(-H)[1]$.
We set $\Psi:=\Phi_{X \to X}^{I_\Delta (H)} \circ \Phi_{X \to X}^{I_\Delta}$.
Then $\Psi(E_x)={\cal O}_x[-2]$.
Since $\Psi({\cal O}_X)=
E_0[-2]$,
$$
\Psi(I_x^{\vee}(-H)[1])[1]=
\ker(E_0 \otimes \Hom(E_0,{\cal O}_x) \to {\cal O}_x).
$$
Since $E_0$ and ${\cal O}_x$ are $\sigma_{(sH,tH)}$-stable on
$s=-\frac{2}{n+4}$ and $t>\frac{2}{n+4}\sqrt{\frac{2}{n+2}}$,
$E_x(-H)$ and ${\cal O}_X(-H)[1]$ are $\sigma_{(sH,tH)}$-stable objects
on \eqref{eq:I_Z} with $s>-1$.

By the Fourier-Mukai transform $I_{X \to X}^{I_\Delta (2H)}$,
\eqref{eq:n+2} is transformed to the exact triangle
\begin{equation}\label{eq:boundary-E}
F_x \to {\cal O}_x[-1] \to E_0^{\oplus(\frac{n}{2}+2)}[1]
\to F_x[1],
\end{equation}
where
$F_x=T_{E_0}^{-1}({\cal O}_x)[-1]$. We also have an expression 
$$
F_x^{\vee}=\ker(E_0 \otimes \Hom(E_0,{\cal O}_x) \to {\cal O}_x).
$$
Thus \eqref{eq:boundary-E} shows that
the corresponding stability 
condition is the boundary of $U(X)$ of type $((E_0)_-)$.
\begin{rem}
Since $\phi_{(sH,tH)}(E_x)=\phi_{(sH,tH)}(I_Z)$ on
\eqref{eq:I_Z} with $-1<s<-\frac{n}{n+2}$,
we can also apply Theorem \ref{thm:isom} (1).
Indeed $\phi_{(sH,tH)}(E_x)<\phi_{(sH,tH)}(I_Z)$ for $-1<s<-\frac{n}{n+2}$
on the outside of the semi-circle.
Let ${\bf E}$ be the family $\{E_x(-H) \mid x \in X\}$. 
Then $\Phi_{X \to X}^{{\bf E}^{\vee}[1]}(I_Z)$ is a stable sheaf.
Since
$\Phi_{X \to X}^{{\bf E}^{\vee}[1]}=T_{E_0} \circ 
\Phi_{X \to X}^{I_\Delta(2H)}$,
the relation with the stable sheaf 
$\Phi_{X \to X}^{I_\Delta(2H)}(I_Z)^{\vee}$ is given by 
$\Phi_{X \to X}^{{\bf E}^{\vee}[1]}(I_Z)=
T_{E_0}(\Phi_{X \to X}^{I_\Delta(2H)}(I_Z))$
(cf. \cite[Thm. 2.3]{Y:twist1}).
\begin{NB}
$\phi_{(sH,tH)}(E_x)<\phi_{(sH,tH)}(I_Z)$ for $-1<s<-\frac{n}{n+2}$.
For a general $I_Z$, we have an exact sequence
$$
0 \to \Phi_{X \to X}^{I_\Delta(2H)}(I_Z) \to 
\Phi_{X \to X}^{{\bf E}^{\vee}[1]}(I_Z)
\to E_0^{\oplus (\frac{n}{2}-2)} \to 0.
$$
\end{NB}
\end{rem}
\begin{NB}
We set $v=(r,H,a)$.
Then $(-1,\sqrt{\frac{2}{(H^2)}})$ 
is on the semi-circle if and only if 
\begin{equation}
\lambda^2-\lambda 
\left(\frac{1}{r}-\frac{2r^2+\langle v^2 \rangle}{r(H^2)} \right)
-\frac{2}{(H^2)}=0.
\end{equation}
If $-\lambda \geq \frac{2r^2+\langle v^2 \rangle}{r(H^2)}$
and $-\lambda \geq 2$,
$(\Phi_{X \to Y}^{I_\Delta(-\lambda H)}(E))^{\vee}$is a stable sheaf
for $E \in M_H(v)$.
\end{NB}

\begin{NB}
Assume that $4n-2 \geq (H^2)$.
Then there is an exact sequence
$$
0 \to E \to I_Z \to {\cal O}_X(-H)[1] \to 0
$$
where $E \in {\cal M}_H(2,H,\frac{(H^2)}{2}+1-n)$.
In this case, in order to preserve the stability,
we need $-\lambda \geq 2$ and
\eqref{eq:lambda}.
If $4n-2 < (H^2)$, then
the stability is preserved for all $\lambda \leq -1$.
\end{NB}

\section{Appexdix}

\subsection{Modifications of some results in \cite{Br:3}}

In this section, we shall explain similar technical results 
to those in \cite{Br:3}
which are necessary to describe
$\Stab^{\dagger}(X)$ for Enriques surfaces.
\begin{lem}
Then $\otimes K_X$ acts trivially on $\Stab^{\dagger}(X)$.
In particular 
$E$ is $\sigma$-stable if and only if $E(K_X)$ is $\sigma$-stable.
\end{lem}

\begin{proof}
Let $\sigma(K_X)$ be the stability condition
induced by the action $\otimes K_X:{\bf D}(X) \to {\bf D}(X)$.
Then 
$$
\Stab(X,K_X):=\{\sigma \in \Stab(X) \mid \sigma(K_X)=\sigma \}
$$
is a closed subset of $\Stab(X)$.
By \cite[Lem. 6.4]{Br:stability},
it is also an open subset.
Indeed for $\sigma \in  \Stab(X,K_X)$,
if $f(\sigma,\tau)<1/2$, then
$f(\sigma(K_X),\tau(K_X))=f(\sigma,\tau)<1/2$
(see also \cite[Lem. 2.3]{Br:3} for the definition of $f(\sigma,\tau)$).
Hence $f(\tau(K_X),\tau)<1$. Since the central charge of
$\tau$ and $\tau(K_X)$ are the same,  we get the claim.

Moreover $U(X)$ is contained.
Hence $\Stab^{\dagger}(X)$ is also contained.
\end{proof}

\begin{lem}\label{lem:even}
If $E \in {\bf D}(X)$ satisfies $E(K_X) \cong E$, then
$\rk E$ is even.
\end{lem}

\begin{proof}
Since $\det(E) \cong \det(E)(\rk E K_X)$,
$(\rk E)K_X=0$. Hence $\rk E$ is even.
\end{proof}

\begin{lem}\label{lem:rigid}
Assume that $E \in {\bf D}(X)$ is 
$\sigma$-stable with $\langle v(E)^2 \rangle<0$.
Then we have the following.
\begin{enumerate}
\item 
\begin{equation}
\langle v(E)^2 \rangle=
\begin{cases}
-1, & \rk E \equiv 1 \mod 2,\\
-2, & \rk E \equiv 0 \mod 2.
\end{cases}
\end{equation}
\item
If $\langle v(E)^2 \rangle=-1$, then
$\Ext^1(E,E)=\Ext^2(E,E)=0$.
\item
If $\langle v(E)^2 \rangle=-2$, then 
$\Ext^1(E,E)=0$ and $\Ext^2(E,E) \cong k$.
\end{enumerate}
\end{lem}

\begin{proof}
We set $v(E)=(r,\xi,\frac{s}{2})$. Then
$r,s \in {\Bbb Z}$, $\xi \in \NS(X)$ and 
$r \equiv s \mod 2$.
Since $\langle v(E)^2 \rangle=(\xi^2)-rs$,
$\langle v(E)^2 \rangle$ is even if and only if
$r$ is even.

Assume that $r$ is odd. Then
$\Hom(E,E(K_X)) =0$. Indeed
if there is a non-zero map $E \to E(K_X)$, then
it is isomorphic by the stability of $E$ and $E(K_X)$, which
contradicts Lemma \ref{lem:even}.
Then 
$$
0>\langle v(E)^2 \rangle=\dim \Ext^1(E,E)-1 \geq -1
$$
implies $\Ext^1(E,E)=0$ and $\langle v(E)^2 \rangle=-1$.

Assume that $r$ is even.
Then $0> \langle v(E)^2 \rangle$ means $\langle v(E)^2 \rangle \leq -2$.
Then we see that
\begin{equation}
\begin{split}
-2 \geq & \langle v(E)^2 \rangle=\dim \Ext^1(E,E)-1-\dim \Hom(E,E(K_X)) \\
\geq & -1-\dim \Hom(E,E(K_X)) \geq -2.
\end{split}
\end{equation}
Hence $\langle v(E)^2 \rangle=-2$,
$\Ext^1(E,E)=0$ and $\Hom(E,E(K_X)) \cong k$.
\end{proof}

\begin{defn}\label{defn:twist}
\begin{enumerate}
\item[(1)]
For a spherical object $A$, $T_A$ denotes the twist functor.
\item[(2)]
An object $B$ is exceptional, if $\Hom(B,B)=k$ and $\Ext^i(B,B)=0$
for $i \ne 0$.
For an exceptional object $B$,
we have an autoequivalence $T_B:=\Phi_{X \to X}^{{\cal E}}$ of ${\bf D}(X)$, 
where
$$
{\cal E}:=\mathrm{Cone}(B \boxtimes B^{\vee} \oplus 
B(K_X) \boxtimes (B(K_X))^{\vee} \to {\cal O}_\Delta).
$$
\end{enumerate}
\end{defn} 

\begin{prop}[{\cite[Thm. 12.1]{Br:3}}]\label{prop:general-wall}
Let $\sigma=(Z,{\cal P}) \in \partial U(X)$
be a general point of the boundary.
Then exactly one of the conditions 
$(A^+), (A^-), (C_k)$ in \cite[Thm. 12.1]{Br:3} 
or the following conditions holds.
\begin{enumerate}
\item[($B^+$)]
There is a rank $r$ simple and rigid vector bundle $B$ 
with $\langle v(B)^2 \rangle=-1$ such that
$B$, $B(K_X)$ and $T_B({\cal O}_x)$ are the stable factors
of ${\cal O}_x$ and
the Jordan-H\"{o}lder filtration
is 
$$
0 \to B^{\oplus r} \oplus B(K_X)^{\oplus r} \to {\cal O}_x \to
T_B({\cal O}_x) \to 0.
$$ 
\item[($B^-$)]
There is a rank $r$ simple and rigid vector bundle $B$ 
with $\langle v(B)^2 \rangle=-1$ such that
$B$, $B(K_X)$ and $T_B^{-1}({\cal O}_x)$ are the stable factors
of ${\cal O}_x$ and
the Jordan-H\"{o}lder filtration
is 
$$
0 \to T_B^{-1}({\cal O}_x) \to
{\cal O}_x \to 
(B^{\oplus r} \oplus B(K_X)^{\oplus r})[2] \to 0.
$$ 
\end{enumerate}  
\end{prop}
\begin{NB}
By \cite[Lem. 1.7.9]{MYY:2011:1},
stable factors are $\beta$-twisted stable.
\end{NB}
For the proof, we need a modification of 
\cite[Lem. 12.2]{Br:3}.
 \begin{lem}
\label{lem:12-2}
Let $\sigma=(Z,{\cal P})$ be a stability condition on $X$ and
$E \in{\cal P}(1)$ a semi-stable object of phase 1
such that $E(K_X) \cong E$.
\begin{enumerate}
\item[(1)]
If $\Ext^1(E,E)=0$, then any stable factor $F$ of $E$ satisfies
$\Ext^1(F,F)=0$.
\item[(2)]
If $\Ext^1(E,E) \cong k^{\oplus 2}$, 
then there is a stable factor $A \in {\cal P}(1)$
satisfying 
\begin{enumerate}
\item
$\Ext^1(E,E)=0$ and 
\item
$\Hom(A,E) \ne 0$ or $\Hom(E,A) \ne 0$.
\end{enumerate}
\end{enumerate}
\end{lem}

\begin{proof}
Let $F$ be a stable factor of $E$. Then
there is an exact sequence
$$
0 \to F' \to E \to G \to 0
$$
in ${\cal P}(1)$ such that $F'$ is a successive extension of $F$ and
$F(K_X)$, and
$\Hom(F,G)=\Hom(F(K_X),G)=0$.
Since $\Hom(F',G)=\Hom(F',G(K_X))=0$ and $E(K_X) \cong E$,
we see that $F'(K_X) \cong F'$ and $G(K_X) \cong G$.
Then we get
\begin{equation}
\dim \Ext^1(F',F')+\dim \Ext^1(G,G) \leq 
\dim \Ext^1(E,E) \leq 2.
\end{equation}

(1)
If $\Ext^1(E,E)=0$, then
$\Ext^1(F',F')=0$, which implies 
$\langle v(F'),v(F') \rangle<0$.
Since $v(F') \in {\Bbb Z}v(F)$,
we also have $\langle v(F)^2 \rangle<0$.
Then $\Ext^1(F,F)=0$ by Lemma \ref{lem:rigid}.
If $\rk F$ is odd, then $\Ext^2(F,F)=0$ and
if $\rk F$ is even, then $\Ext^2(E,E) \cong k$.
A similar claim also holds for $F(K_X)$.

Since $G$ also satisfies the assumption of (1),
inductively we get the claim for stable factors of $G$.

(2)
We note that $\rk F'$ and $\rk G$ are even by Lemma \ref{lem:even}.
Hence $\langle v(F')^2 \rangle$ and
$\langle v(G)^2 \rangle$ are even.
Since 
\begin{equation}
\dim \Ext^2(F',F') =\dim \Hom(F',F'),\;
\dim \Ext^2(G,G) = \dim \Hom(G,G),
\end{equation}
$\dim \Ext^1(F',F')$ and
$\dim \Ext^1(G,G)$ are even.
Therefore $\Ext^1(F',F')=0$ or
$\Ext^1(G,G)=0$.
Applying (1), we get the claim.
\end{proof}

\begin{NB}

(Step 6)
Assume that there is a map ${\cal O}_x \to A$.
Then $A \in {\cal Q}(\phi)$ with $1<\phi<1+\epsilon$.
Since $A[-1] \in {\cal Q}((0,1])$,
\cite[Lem. 10.1]{Br:3} implies 
$H^i(A[-1])=0$ for $i \ne -1,0$.
We set $C:=H^0(A[-2])[2]$ and $D:=H^0(A[-1])[1]$.
Then $C[-1] \in {\cal Q}((0,2])$ and
$D[-1] \in {\cal Q}((-1,1])$, and hence
$C[-1] \in {\cal P}((-\epsilon,2+\epsilon])$ and
$D[-1] \in {\cal Q}((-1-\epsilon,1+\epsilon])$.
Since $\epsilon$ is arbitrary small positive number,
 $C[-1] \in {\cal P}([0,2])$ and
$D[-1] \in {\cal P}([-1,1])$.
By the triangle
$$
D[-1] \to C \to A \to D,
$$
$C \in {\cal P}(\leq 1)$.
Since $C \in {\cal P}([1,3])$,
$C \in {\cal P}(1)$.  Since $A$ is a simple object,
$D[-1] \in {\cal P}(1)$.
Assume that $C,D \ne 0$.
We note that the choice of $A$ is finite.
Hence the choice of the cohomology sheaves $C,D$ are also finite.
Since $\sigma$ is general, there are integers 
$\lambda_C,\mu_C,\lambda_D, \mu_D$
and
$$
v(C)=\lambda_C v(A)+\mu_C v({\cal O}_x),\;
v(D)=\lambda_D v(A)+\mu_D v({\cal O}_x).
$$
Since $C[-1] \in {\cal Q}((0,\epsilon))$ and
$D[-1] \in {\cal Q}((1-\epsilon,1])$,
$\Ima W(C[-1]), \Ima W(A[-1])>0$ and $\Ima W(D[-1]) \geq 0$,
$\lambda_C>0$ and $\lambda_D \geq 0$.
Since $\lambda_C+\lambda_D=1$,
$\lambda_C=1$ and $\lambda_D=0$.
Then $\langle v(D)^2 \rangle=0$, which implies $D$ is not rigid.
Therefore $C=0$ or $D=0$.

If $D=0$, then $C=A$. In particular, 
$A[-2]$ is a torsion free sheaf by \cite[Lem. 10.1]{Br:3}.
Hence $A[-2]$ is locally free.

Assume that $C=0$.
Then $A=D$ with $A[-1] \in \Coh(X)$.
Let $T$ be the torsion part of $A[-1]$ and $Q:=A[-1]/T$.
Then $T \in {\cal Q}((0,1])$ implies $T \in {\cal P}(\geq 0)$. 
We see that $T \in {\cal P}(0)$ and $Q \in {\cal P}(1)$.
Since $\sigma$ is general,
$v(T)=\lambda_T v(A)+\mu_T v({\cal O}_x)$.
If $\rk A \ne 0$, then $\lambda_T =0$.
Then $T$ is 0-dimensional, which means $T \in {\cal P}(1)$.
Therefore $T=0$.
Then $A[-1]$ is a locally free rigid sheaf, which implies
$\Hom({\cal O}_x,A)=0$.
Therefore $\rk A=0$.
In this case, if there is a decomposition
$$
0 \to C \to A[-1] \to D \to 0
$$ 
in $\Coh(X)$, then
$C,D \in {\cal Q}((0,1])$ implies that
$C \in {\cal P}(0)$ and $D \in {\cal P}(1)$.
Hence $A[-1]$ is purely 1-dimensional.
Since $\sigma$ is general,
$v(C)=\lambda_C v(A)+\mu_C v({\cal O}_x)$.
Then $c_1(C)$ is a multiple of $c_1(A)$.
Hence $c_1(A)$ is a smooth rational curve $C$.
Therefore $A={\cal O}_C(k)[1]$.

Assume that $\rk A \ne 0$.
We set $r:=\rk A$.
Then there is an exact sequence
$$
0 \to B \to {\cal O}_x \to A^{\oplus r} \oplus A(K_X)^{\oplus r} \to 0
$$
or 
$$
0 \to B \to {\cal O}_x \to A^{\oplus r} \to 0
$$
according as $v(A)=-1,-2$.
We set $A':=A^{\oplus r} \oplus A(K_X)^{\oplus r}$.
Since $\Hom({\cal O}_x,A[i])=0$ for $i \ne 0$,
$\Hom(B,A[i])=\Hom(B,A(K_X)[i])=0$ for $i=0,2$.
Then we have $B(K_X) \cong B$.
Assume that $B$ is not $\sigma$-stable.
Then Lemma \ref{lem:12-2} implies that
there is a stable and rigid object $C \in {\cal P}(1)$ such that
$\Hom(C,B) \ne 0$ or $\Hom(B,C) \ne 0$.
Since $\sigma$ is general,
$v(C)=\lambda_C v(A) +\mu_C v({\cal O}_x)$, where
$\lambda_C,\mu_C \in {\Bbb Z}$.
Then $\rk C=r\lambda_C$ and
$\langle v(C)^2 \rangle=
\lambda_C(\lambda_C \langle v(A)^2 \rangle
+2\mu_C \langle v(A),v({\cal O}_x) \rangle)$.
If $ \langle v(C)^2 \rangle=-2$, then $\rk C$ is even.
Hence $2 \mid \lambda_C$ or $2 \mid \rk A$.
For the first case, we get $4 \mid \langle v(C)^2 \rangle$,
which is a contradiction. 
For the second case, we have $\langle v(A)^2 \rangle=-2$, which
implies that $\lambda_C=\pm 1$ and $\mu_C=0$.
Since $A,C \in {\cal P}(1)$, $v(A)=v(C)$.
If $ \langle v(C)^2 \rangle=-1$, then $\rk C$ is odd.
Hence $\lambda_C=\pm 1, \langle v(A)^2 \rangle=-1$
and $\mu_C=0$.
Then we also see that $v(A)=v(C)$.
By the Riemann-Roch theorem, we see that 
$C \cong A,A(K_X)$, which is a contradiction.

We normalize $Z(\bullet)$ as 
$\langle e^{\beta+i \omega},\bullet \rangle$.
Let $E_1$ be a subsheaf of $A[-2]$.
Since $A[-2] \in {\cal Q}((-1,0]) \cap \Coh(X)$,
we have $E_1 \in {\cal Q}((-1,0]) \cap \Coh(X)$.
Thus $\Ima W(E_1) \leq 0$, which implies 
$\Ima Z(E_1) \leq 0$.
Since $\Ima Z(A[-2])=0$,
$A$ is $\mu$-semi-stable with respect to
$\omega$.
If $\Ima Z(E_1)=0$, then by the stability of $A$,
$Z(E_1) \in {\Bbb R}_{>0}$.
Hence $\chi_\beta(A[-2])/r>\frac{(\omega^2)}{2}>
\chi_\beta(E_1)/\rk E_1$.
Thus $A$ is $\beta$-twisted stable. 
 
\end{NB}

\begin{NB}
Let $X$ be an Enriques surface.
Let ${\cal E}$ be a universal family of
stable sheaves $E$ with $v(E)=v_0$.

$\Phi:\Stab(X) \to \Stab(X')$
For a family of stable objects ${\cal F}$ on $X'$
with the Mukai vector $\varrho_X$,
$\Phi_{X' \to X}^{{\cal E}[2]}({\cal F})$
is a family of stable complexes with Mukai vector $v_0$.
\end{NB}

\begin{NB}
\subsection{}

\begin{lem}\label{lem:Hodge}
Let $H$ be a divisor with $(H^2)>0$.
Let $\omega$ be a divisor with $(\omega^2)>0$.
If $C=x H+D$ $D \in H^\perp$ satisfies
$(C,\omega)=0$, then
$$
(C^2) \leq \frac{(\omega^2)(H^2)}{(H,\omega)^2} (D^2).
$$
\end{lem}

\begin{proof}
We set $\omega=pH+\lambda$ ($\lambda \in H^\perp$).
$0=(C,\omega)=p x (H^2)+(D,\lambda)$.
Hence
$x^2(H^2)^2 =\frac{(D,\lambda)^2}{p^2} \leq \frac{(D^2)(\lambda^2)}{p^2}$
by using Hodge index theorem and Schwarz inequality.
\begin{equation}
(C^2)=x^2(H^2)+(D^2) \leq \frac{(D^2)(\lambda^2)}{p^2(H^2)}+(D^2)
=(D^2)\frac{(\lambda^2)+p^2(H^2)}{p^2(H^2)}=
(D^2)\frac{(\omega^2)(H^2)}{(\omega,H)^2}.
\end{equation}
\end{proof}
 
\begin{lem}\label{lem:Bdd}
Let $B$ be a compact subset of $\NS(X)_{\Bbb R} \times P^+(X)_{\Bbb R}$.
Then there is finitely many $u=r+\xi+a \varrho_X \in \Delta(X)$
such that
$r>0$, $(\xi-r \beta,\omega)=0$ and 
$-\langle e^\beta,w \rangle>r \frac{(\omega^2)}{2}$
for $(\beta,\omega) \in B$.
\end{lem}
\begin{NB2}
cf. \cite[Lem. 11.1]{Br:3}.
\end{NB2}

\begin{proof}
Since $\frac{(\xi^2)}{r^2}+\frac{-\langle u^2 \rangle}{r^2}=2\frac{a}{r}$,
\begin{equation}
-r\langle e^\beta,w \rangle=ra-r(\xi,\beta)+r^2 \frac{(\beta^2)}{2}
=\frac{1}{2}((\xi-r\beta)^2)+\frac{-\langle u^2 \rangle}{2}.
\end{equation}
Hence $r^2(\omega^2)<((\xi-r\beta)^2)-\langle u^2 \rangle 
\leq -\langle u^2 \rangle$.
In particular $r$ is bounded, which implies the choice of $r$ is finite.
We also note that $-((\xi-r\beta)^2) \leq -\langle u^2 \rangle$.
We fix a divisor $H$ with $(H^2)>0$ and write $\xi-r\beta=pH+D$, 
$D \in H^\perp$.
Since $((\xi-r\beta)^2) \geq \langle u^2 \rangle$,
$(D^2)$ is also bounded by Lemma \ref{lem:Hodge}. 
Hence $\{p \}$ and $\{D \}$ are bounded sets.
Therefore the choice of $\xi$ is finite, which also implies
the choice of $a$ is finite.
\end{proof}
\end{NB}

\begin{NB}
If $A$ is a subobject of ${\cal O}_x$ in an abelian category 
${\cal P}_\tau((a,a+1])$ with 
$\phi_\sigma(A)=\phi_\sigma({\cal O}_x)$, then
$(c_1(A)-(\rk A) \beta,\omega)=0$ and
$\rk A \frac{(\omega^2)}{2}-\chi_\beta(A)<0$.
Hence $v(A)=e^\beta(\rk A+D+a_\beta \varrho_X)$
with $\langle v(A)^2 \rangle=(D^2)-2(\rk A) a_\beta<0$
unless $v(A)=v({\cal O}_x)$.
If $v(A)=v({\cal O}_x)$, then
${\cal O}_x /A \in {\cal P}_\tau((a,a+1])$
satisfies $Z_\tau(   {\cal O}_x /A )=0$.
Hence ${\cal O}_x /A =0$.
Therefore if $A$ is $\sigma'$-stable and $A \ne {\cal O}_x$, then 
$\langle v(A)^2 \rangle=-2$.
\end{NB}


\begin{NB}
We have an exact sequence
\begin{equation}
0 \to H^0(\pH^{-1}(B)) \to H^0(B) \to H^{-1}(H^0(B)) \to 0.
\end{equation}

\begin{equation}
\begin{CD}
@.  @. 0 @. 0 @.\\
@. @. @AAA @AAA @.\\
@. @. {\cal O}_x @= {\cal O}_x @.\\
@. @. @AAA @AAA @.\\
0 @>>> H^0(\pH^{-1}(B)) @>>> H^0(\pH^0(A)) @>>> H^0(C) @>>> 0\\
@. @| @AAA @AAA @.\\
0 @>>> H^0(\pH^{-1}(B)) @>>> H^0(B) @>>> H^{-1}(H^0(B)) @>>> 0\\
@. @. @AAA @AAA @.\\
@. @. 0 @. 0 @. .
\end{CD}
\end{equation}

\end{NB}

\begin{NB}
Replacing $B$ by a compact neighborhood of $(\beta_0,\omega_0)$,
we assume that all 
$w=r+\xi+a \varrho_X \in W$ satisfy 
$Z_{(\beta_0,\omega_0)}(w) \in {\Bbb R}_{\leq 0}$.
\begin{lem}
We set $B':=\{(\beta,\omega) \in B_0 \mid \omega \in \Amp(X),
\Ima Z_{(\beta,\omega)}(w)>0, w \in W' \}$.
\begin{enumerate}
\item[(1)]
Then  $B' \ne \emptyset$
and $(\beta_0,\omega_0) \in \overline{B'}$.
\item[(2)]
 Let 
\begin{equation}
0 \to A \to {\cal O}_x \to B \to 0
\end{equation}
be an exact sequence in ${\cal A}_{(\beta_0,\omega_0)}$. Then 
$\Ima Z_{(\beta,\omega)}(A)>0$ for $(\beta,\omega)  \in B'$.
\end{enumerate}
\end{lem}

\begin{proof}
For $x, y \in {\Bbb R}_{>0}$ and
$\eta \in \omega_0^\perp$ with
$(\xi-r \beta_0,\eta)+rxy(\omega_0^2)>0$, 
$\Ima Z_{(\beta_0-x\omega_0,y\omega_0+\eta)}(w)>0$.
Hence (1) holds.

(2) We note that  $\pH^{-1}(A)=0$ and 
there is an exact sequence in ${\frak C}$
\begin{equation}
0 \to \pH^{-1}(B) \to \pH^0(A) 
\overset{\psi}{\to} {\cal O}_x \to \pH^0(B) \to 0.
\end{equation}
Since $A,B$ are semi-stable objects of phase 1, 
$\pH^0(B)$ is a 0-dimensional object,
$\pH^{-1}(B)$ is a $\mu$-semi-stable object with
$\deg_\beta \pH^{-1}(B)=0$
and $\pH^0(A)$ is an extension of a $\mu$-semi-stable object $F$ with
$\deg_\beta F=0$ by a 0-dimensional object $T$.
Moreover $F$ is generated by 
$\beta_0$-twisted stable objects $F_i$
with $\chi_{\beta_0}(F_i)/\rk F_i \geq \frac{(\omega_0^2)}{2}$. 
Thus $v(F_i) \in W$. Then $\Ima Z_{(\beta,\omega)}(F_i)>0$. 
Since $\pH^{-1}(B)$ is a torsion free object of ${\frak C}$,
$T$ is a subobject of ${\cal O}_x$. In particular,
$c_1(T)$ is effective, $T={\cal O}_x$ or $T=0$.
Since $\omega$ is ample, 
$\Ima Z_{(\beta,\omega)}(H^0(\pH^0(A)))>0$
unless $H^0(\pH^0(A))=0,{\cal O}_x$.

Let $C$ be the image of $\psi$ in ${\frak C}$.
Then $H^{-1}(C)=0$ by $H^{-1}({\cal O}_x)=0$.
Hence $H^{-1}(\pH^0(A))\cong H^{-1}(\pH^{-1}(B))=0$
by the torsion freeness of $\pH^{-1}(B)$.
Since $A \ne {\cal O}_x, 0$, we get $\phi_{(\beta,\omega)}(A)<1$.
\end{proof}

\begin{lem}[{cf. \cite[Lem. 13.3]{Br:3}}]
The stability condition
$\sigma_{(\beta_0,\omega_0)}=
({\cal A}_{(\beta_0,\omega_0)},Z_{(\beta_0,\omega_0)})$
defined in \cite{MYY:2011:1} 
is in the boundary of $U(X)$. 
\end{lem}

\begin{proof}
We first note that $\sigma_{(\beta_0,\omega_0)}$ satisfies the support 
property,
since the Bogomolov inequality holds.
In particular, we have a wall and chamber structure.
There is a neighborhood of $(\beta_0,\omega_0)$
in $\NS(X)_{\Bbb R} \times \Amp(X)_{\Bbb R}$
and a continuous map ${\frak s}:V \to \Stab(X)$
such that $Z_{{\frak s}(\beta, \omega)}=Z_{(\beta,\omega)}$
and ${\frak s}(\beta_0,\omega_0)=\sigma_{(\beta_0,\omega_0)}$.
We show that ${\frak s}(\beta, \omega)$ is geometric
for $(\beta,\omega) \in B_0'$.
We set $\sigma:={\frak s}(\beta, \omega)$ and
$\sigma_0:={\frak s}(\beta_0, \omega_0)$.
We may assume that
$$
\sup_{0 \ne E \in {\bf D}(X)}
|\phi_{\sigma}^\pm(E)-\phi_{\sigma_0}^\pm(E)|<\frac{1}{8}
$$
for $0 \leq s \leq 1$.
We set ${\cal A}:={\cal P}_{\sigma_0}((\tfrac{1}{2},\tfrac{3}{2}])
(\subset {\bf D}(X))$.
If ${\cal O}_x$ is not $\sigma$-semi-stable,
then we have $A,B \in {\cal A}$ 
and an exact sequence 
$$
0 \to A \to {\cal O}_x \to B \to 0
$$
in ${\cal A}$
such that $\phi_{\sigma}(A) \geq \phi_{\sigma}({\cal O}_x)$.
We may assume that $A$ is $\sigma$-stable.

$\phi_{\sigma_0}(A)=\phi_{\sigma_0}(B)$.
Thus we have an exact sequence in
${\cal A}_{(\beta_0,\omega_0)}$.
$$
0 \to A \to {\cal O}_x \to B \to 0. 
$$
By our assumption, we get
$\mathrm{Im} Z_{\sigma}(E_1)>0> \mathrm{Im} 
Z_{\sigma}(B)$, which implies
$\phi_{\sigma}(A)<1<\phi_{\sigma}(B)$.
Therefore ${\cal O}_x$ is $\sigma$-stable.
In particular ${\frak s}(\beta,\omega)=\sigma_{(\beta,\omega)}$
for $(\beta,\omega) \in B_0'$.
\end{proof}

\end{NB}

\subsection{A complement on the wall crossing in \cite{MYY:2011:1}}

Let $P_{\gamma,H}$ be a family of stability condition in \eqref{eq:Plane}.
We shall study the wall crossing in \cite{MYY:2011:1} by using the description
of stability conditions in section \ref{sect:MYY}.
For simplicity,
we assume that $X$ is a K3 surface.
Similar claims also hold for the case of an Enriques surface. 
Let $U$ be the open subset of $P_{\gamma,H}$ such that
$(\beta,\omega) \in U$ if and only if 
$Z_{(\beta,\omega)}(u) \not \in {\Bbb R}_{\leq 0}$
for any $u \in \Delta(X)$ with $\rk u>0$.
\begin{NB}
Let $U$ be the open subset of
$\NS(X)_{\Bbb R} \times \Amp(X)_{\Bbb R}$
such that $(\beta,\omega) \in U$ if and only if 
$Z_{(\beta,\omega)}(u) \not \in {\Bbb R}_{\leq 0}$
for any $u \in \Delta(X)$ with $\rk u>0$.
\begin{NB2}
For $u=e^\beta(r+D+a \varrho_X)$,
$Z_{(\beta,\omega)}(u)=
r(\omega^2)/2-a$.
\end{NB2}
\end{NB}
For a Mukai vector
$v=r+\xi+a \varrho_X$ ($r \in {\Bbb Z}_{>0}, \xi \in \NS(X), a \in {\Bbb Q}$),
we set $\delta:=\frac{\xi}{r}$. Then
$v=r e^\delta-\frac{\langle v^2 \rangle}{2r}\varrho_X$.
 
As in \cite{movable}, 
we set
\begin{equation}\label{eq:def-xi}
\begin{split}
\xi(\beta,\omega)
:=& \xi(\beta,\omega,1)/r\\
=& 
e^\gamma \left(
\frac{(\omega^2)-((\beta-\delta)^2)}{2}+
\frac{\langle v^2 \rangle}{2r^2} \right)
\omega\\
&+e^\gamma ((\beta-\delta) \cdot \omega)(\beta-\delta)
+((\beta-\delta) \cdot \omega)
\left(e^\delta+\frac{\langle v^2 \rangle}{2r^2}\varrho_X \right)
\in C^+(v)
\end{split}
\end{equation}
for $(\beta,\omega) \in \NS(X)_{\Bbb R} \times P^+(X)_{\Bbb R}$,
where $P^+(v)$ is the positive cone
of $v^\perp$, and
$C^+(v):=P^+(v)/{\Bbb R}_{>0}$. 
We have
$$
\xi(\beta,\omega) =
\mathrm{Im}\frac{e^{\beta+\sqrt{-1}\omega}}
{Z_{(\beta,\omega)}(v)} \in C^+(v).
$$
For $v_1 \in H^*(X,{\Bbb Q})_{\alg}$,
$Z_{(\beta,\omega)}(v_1) \in {\Bbb R}Z_{(\beta,\omega)}(v)$
if and only if $\xi(\beta,\omega) \in v_1^\perp$.
For the open set $U$,
semi-stability is constant on the fiber of $\xi$ \cite[Cor. 3.6]{Y:wall}.
We shall slightly generalize the result to a point of the boundary
of $U$.
Let $(\beta,t_0 H)$ be a point of $\partial U$ such that
$\beta \in \NS(X)_{\Bbb Q}$ and
$Z_{(\beta,t_0 H)}(u)\ne 0$ for all $u \in \Delta(X)$ with 
$\rk u>0$.

\begin{NB}
\begin{lem}[{cf. \cite[Lem. 13.3]{Br:3}}]
The stability condition
$\sigma_{(\beta,t' H)}=({\cal A}_{(\beta,t' H)},Z_{(\beta,t' H)})$
defined in \cite{MYY:2011:1} 
is in the boundary if $t'$ is general.
\end{lem}

\begin{proof}
We first note that $\sigma_{(\beta,t' H)}$ satisfies the support 
property,
since the Bogomolov inequality holds.
In particular, we have a wall and chamber structure.
There is a neighborhood of $(\beta,t' H)$
in $\NS(X)_{\Bbb R} \times \Amp(X)_{\Bbb R}$
and a continuous map ${\frak s}:V \to \Stab(X)$
such that $Z_{{\frak s}(\gamma, \omega)}=Z_{(\gamma,\omega)}$
and ${\frak s}(\beta,t' H)=\sigma_{(\beta,t' H)}$.
Let $V'$ be an open subset of $V$ such that
$\sigma \in {\frak s}(V')$ if and only if 
$\mathrm{Im} Z_\sigma (G)>0$, i.e.,
$d_\gamma(G)>0$ for all
$(-2)$-vectors $G$ with $v(G)=e^\beta(r+D+a \varrho_X)$,
$r,a >0$, $D \in H^\perp$.
Let $\sigma_s$ $(0 \leq s \leq 1$) be a path connecting
$\sigma_1:=\sigma \in {\frak s}(V')$ and $\sigma_0:=\sigma_{(\beta,t' H)}$
such that $\sigma_s \in U'$ for $s>0$.
We may assume that
$$
\sup_{0 \ne E \in {\bf D}(X)}
|\phi_{\sigma_s}^\pm(E)-\phi_{\sigma_0}^\pm(E)|<\frac{1}{8}
$$
for $0 \leq s \leq 1$.
We set ${\cal A}:={\cal P}_{\sigma_0}((\tfrac{1}{2},\tfrac{3}{2}])
(\subset {\bf D}(X))$.
If $k_x$ is not $\sigma_s$-semi-stable for $s>0$,
then we have $E_1,E_2 \in {\cal A}$ 
and an exact sequence 
$$
0 \to E_1 \to k_x \to E_2 \to 0
$$
in ${\cal A}$
such that $\phi_{\sigma_s}(E_1) \geq \phi_{\sigma_s}(k_x)$ and
$\phi_{\sigma_0}(E_1)=\phi_{\sigma_0}(E_2)$.
Thus we have an exact sequence in
${\cal A}_{(\beta,t' H)}$.
$$
0 \to E_1 \to k_x \to E_2 \to 0. 
$$
Then $H^{-1}(E_1)=0$ and
we have an exact sequence
$$
0 \to H^{-1}(E_2) \to H^0(E_1) \to k_x \to H^0(E_2) \to 0.
$$
If $H^0(E_2) \cong k_x$, then
$H^{-1}(E_2) \to H^0(E_1)$ is an isomorphism, which is a contradiction.
Hence $H^0(E_2)=0$.
Then $H^{-1}(E_2)^{\vee \vee}$ is generated by
stable locally free sheaves $G$ with 
$v(G)=e^\beta(r+D+a \varrho_X)$,
$r,a >0$.
By our assumption, we get
$\mathrm{Im} Z_{\sigma_s}(E_1)>0> \mathrm{Im} 
Z_{\sigma_s}(E_2)$ $(s>0)$, which implies
$\phi_{\sigma_s}(E_1)<1<\phi_{\sigma_s}(E_2)$ $(s>0)$.
Therefore $k_x$ is $\sigma_s$-stable for
$s>0$ and $x \in X$.
In particular ${\frak s}(\gamma,\omega)=\sigma_{(\gamma,\omega)}$
for $(\gamma,\omega) \in V'$.
\end{proof}
\end{NB}

By Proposition \ref{prop:limit},
$\sigma_{(\beta,t_0 H)}$-semi-stability
is equivalent to $\sigma_{(\beta+sH,tH)}$-semi-stability
for $(\beta+sH,tH) \in \xi^{-1}(\xi(\beta,t_0 H)) \cap U$
and $s<0$
(\cite[Cor. 3.6]{Y:wall}).

We next consider a point $(\beta,t_0 H) \in \partial U$ such that
$Z_{(\beta,t_0 H)}(E)=0$ for a $\beta$-twisted stable
object $E$.
Let ${\frak S}$ be the set of $\beta$-twisted 
semi-stable objects $E$ with $Z_{(\beta,t_0 H)}(E)=0$, and let
${\frak E}=\{G_1,...,G_n \}$ be the set of $\beta$-twisted stable objects
$G_i \in {\frak S}$.

\begin{lem}\label{lem:G_i}
There is a positive number $\epsilon$ such that 
$G_i$ are $\sigma_{(\beta+sH,tH)}$-stable
for all $0>s \geq -\epsilon$ and $t_- \leq t_0 \leq t_+$.
Moreover $M_{(\beta+sH,tH)}(v(G_i))^{ss}=\{G_i \}$.
\end{lem}

\begin{proof}
$G_i \in {\cal A}_{(\beta,t_- H)}$ are $\sigma_{(\beta,t_- H)}$-stable
and 
$G_i \in {\cal A}_{(\beta,t_+ H)}[-1]$ are $\sigma_{(\beta,t_+ H)}$-stable.
Hence 
there is a positive number $\epsilon$ such that 
$G_i$ are $\sigma_{(\beta+sH,t_\pm H)}$-stable
for all $0>s \geq -\epsilon$.
If $t_- \leq t_0 \leq t_+$, then
$G_i$ are $\sigma_{(\beta+sH,tH)}$-stable for 
all $0<s \leq \epsilon$.
Assume that $G$ is a $\sigma_{(\beta+sH,tH)}$-semi-stable 
object with $v(G)=v(G_i)$.
Since $\chi(G,G_i)>0$, there is a morphism 
$\psi_1:G \to G_i$ or a morphism $\psi_2:G_i(K_X) \to G$.
Since the phase are the same,
$\psi_1$ is injective and $\psi_2$ is surjective.
Since $v(G)=v(G_i)$, $\psi_1,\psi_2$ are isomorphisms.
Therefore our claim holds.
\end{proof}

\begin{NB}
We set $v(G_i)=e^\beta(r_i+d_i H+D_i+a_i \varrho_X)$.
The equation of walls are
$$
(x+a)^2+y^2=(p+a)^2-q.
$$
If $q:=\frac{\langle v(G_i)^2 \rangle-(D_i^2)}{(H^2)r_i^2} <0$,
then
all walls pass $(x,y)=(d_i/r_i,\sqrt{-q})$.
Thus the structure of walls is quite different from
the case of $\langle v^2 \rangle \geq 0$.
\end{NB}

\begin{defn}[{\cite[Defn. 4.2.1]{MYY:2011:1}}]
We take $t_+>t_0>t_-$ such that $t_+-t_-$ is sufficiently small.
\begin{enumerate}
\item[(1)]
$E \in {\cal A}_{(\beta,t_- H)}$ is 
$\sigma_{(\beta,t_0 H)}$-semi-stable, if
$\phi_{(\beta,t_0 H)}(E_1) \leq \phi_{(\beta,t_0 H)}(E)$
for any proper subobject $E_1 \ne 0$ of $E$ with 
$Z_{(\beta,t_0 H)}(E_1) \ne 0$.
If 
$\phi_{(\beta,t_0 H)}(E_1) < \phi_{(\beta,t_0 H)}(E)$
for any proper subobject $E_1 \ne 0$ of $E$ with 
$Z_{(\beta,t_0 H)}(E_1) \ne 0$, then
$E$ is $\sigma_{(\beta,t_0 H)}$-stable.
\item[(2)]
Let ${\cal M}_{(\beta,t_0 H)}(v)$ (resp. 
${\cal M}_{(\beta,t_0 H)}(v)^s$) be the moduli stack
of $\sigma_{(\beta,t_0 H)}$-semi-stable objects (resp. 
$\sigma_{(\beta,t_0 H)}$-stable objects) $E$ with
$v(E)=v$. 
\end{enumerate}
\end{defn}

\begin{rem}
If there is a homomorphism 
$\psi: E \to G_i$ for a
$\sigma_{(\beta,t_0 H)}$-semi-stable object $E$, then $\psi$ is surjective and 
$\phi_{(\beta,t_0 H)}(\ker \psi)=\phi_{(\beta,t_0 H)}(E)$.
Hence $\Hom(E,G_i)=0$ for a $\sigma_{(\beta,t_0 H)}$-stable object $E$.
\end{rem}

In order to relate $\sigma_{(\beta,t_0 H)}$-semi-stabilty with
Bridgeland semi-stability, we first prove the following.
\begin{lem}\label{lem:wall-stable}
Assume that $Z_{(\beta,t_0 H)}(v) \in {\Bbb R}_{> 0}e^{\pi \sqrt{-1} \phi}$,
$0<\phi<1$.
Then
\begin{equation}
{\cal M}_{(\beta,t_+ H)}(v)^s \cap {\cal M}_{(\beta,t_- H)}(v)^s 
={\cal M}_{(\beta,t_0 H)}(v)^s.
\end{equation}
\end{lem}

\begin{proof}

\begin{NB}
Assume that $Z_{(\beta,tH)}(v)=0$.
Let $E$ be an object of ${\bf D}(X)$ with $v(E)=v$.
If $E \in {\cal A}_{(\beta,t_- H)}$, then
$E \in {\frak S}$, and
if $E \in {\cal A}_{(\beta,t_+ H)}$, then
 $E \in {\frak S}[1]$ 
by \cite[Lem. 4.2.3]{MYY:2011:1}.
Hence ${\cal M}_{(\beta,t H)}(v)^s \not = \emptyset$
if and only if there is a stable sheaf $G_i$ with $v(G_i)=v$,
and  ${\cal M}_{(\beta,t H)}(v)^s=\{G_i \}$.
We have ${\cal M}_{(\beta,t_- H)}(v)=\{ E \in {\frak S} \mid v(E)=v \}$
and ${\cal M}_{(\beta,t_+ H)}(v)=\{ E \in {\frak S} \mid v(E)=v \}$.
Since $Z_{(\beta,t_\pm H)}(E) \in {\Bbb R}$,
$E$ is $\sigma_{(\beta,t_\pm H)}$-stable if and only if $E$ is 
$\beta$-twisted stable.
Thus we also have ${\cal M}_{(\beta,t_\pm H)}(v)^s=\{G_i \}$.
Hence the claim holds. 
\end{NB}

Assume that $E \in 
{\cal M}_{(\beta,t_+ H)}(v)^s \cap {\cal M}_{(\beta,t_- H)}(v)^s$. 
If $E$ is not $\sigma_{(\beta,t_0 H)}$-stable, then
there is a subobject $F$ in ${\cal A}_{(\beta,t_- H)}$
such that 
$\phi_{(\beta,t_0 H)}(F)=\phi_{(\beta,t_0 H)}(E)$.
We take an exact sequence in ${\cal A}_{(\beta,t_- H)}$
\begin{equation}
0 \to F_1 \to F \to F_2 \to 0
\end{equation} 
such that $\pH^{-1}(F_1)=\pH^{-1}(F) \in 
{\cal F}_{(\beta,t_- H)} (\subset {\cal F}_{(\beta,t_+ H)})$,
$\pH^0(F_1) \in {\cal T}_{(\beta,t_+ H)}$ and
$F_2 \in {\frak S} (\subset {\cal T}_{(\beta,t_- H)})$.
Then $E/F_1 \in {\cal A}_{(\beta,t_- H)}$.
Since $\Hom(E,G)=0$ for $G \in {\frak S}$,
we have $\Hom(E/F_1,G)=0$ for $G \in {\frak S}$,
which implies $E/F_1 \in  {\cal A}_{(\beta,t_+ H)}$
by \cite[Lem. 4.2.2]{MYY:2011:1}.
Since $F_1 \in {\cal A}_{(\beta,t_+ H)}$,
we have an exact sequence in ${\cal A}_{(\beta,t_\pm H)}$:
\begin{equation}
0 \to F_1 \to E \to E/F_1 \to 0.
\end{equation}
By the stability of $E$,
we have $\phi_{(\beta,t_\pm H)}(F_1) \leq 
\phi_{(\beta,t_\pm H)}(E)$.
Since $\phi_{(\beta,t_0 H)}(F_1)=\phi_{(\beta,t_0 H)}(F)=
\phi_{(\beta,t_0 H)}(E)$,
we have $\phi_{(\beta,t_\pm H)}(F_1)= 
\phi_{(\beta,t_\pm H)}(E)$, which means
$E$ is properly $\sigma_{(\beta,t_\pm H)}$-semi-stable.
Therefore $E \in {\cal M}_{(\beta,t_0 H)}(v)^s$.

Conversely for $E \in {\cal M}_{(\beta,t_0 H)}(v)^s$, assume that 
$E \not \in 
{\cal M}_{(\beta,t_+ H)}(v)^s \cap {\cal M}_{(\beta,t_- H)}(v)^s$.
If $E \not \in {\cal M}_{(\beta,t_+ H)}(v)^s$, then
there is a subobject $F$ of $E$ in ${\cal A}_{(\beta,t_+ H)}$
such that $\phi_{(\beta,t_+ H)}(F) \geq \phi_{(\beta,t_+ H)}(E)$.
Then $\phi_{(\beta,t_0 H)}(F) \geq \phi_{(\beta,t_0 H)}(E)$.
For $E/F \in {\cal A}_{(\beta,t_+ H)}$, we have an exact sequence
in ${\cal A}_{(\beta,t_+ H)}$
\begin{equation}
0 \to F' \to E/F \to F_2 \to 0
\end{equation}
such that 
$F' \in {\frak S}[1]$, $\pH^{-1}(F_2) \in {\cal F}_{(\beta,t_- H)}$
and 
$\pH^0(E/F)=\pH^0(F_2) \in {\cal T}_{(\beta,t_+ H)}
(\subset {\cal T}_{(\beta,t_- H)})$.
Then $E \to F_2$ is surjective in ${\cal A}_{(\beta,t_+ H)}$.
We set $F_1=\ker(E \to F_2) \in {\cal A}_{(\beta,t_+ H)}$.
By the construction of $F_2$, 
we have $F_2 \in {\cal A}_{(\beta,t_- H)}$.
Since $\Hom(G[1],E)=0$ for $G \in {\frak S}$,
we have $\Hom(G[1],F_1)=0$.  
Hence $F_1 \in {\cal A}_{(\beta,t_- H)}$.
Thus we have an exact sequence
$$
0 \to F_1 \to E \to F_2 \to 0
$$
in ${\cal A}_{(\beta,t_\pm H)}$.
Then $\phi_{(\beta,t_0 H)}(F_1)=\phi_{(\beta,t_0 H)}(F) \geq
\phi_{(\beta,t_0 H)}(E)$, which shows that
$E$ is not $\sigma_{(\beta,t_0 H)}$-stable.

If $E \not \in {\cal M}_{(\beta,t_- H)}(v)^s$, then
there is a subobject $F$ of $E$ in ${\cal A}_{(\beta,t_- H)}$
such that $\phi_{(\beta,t_- H)}(F) \geq \phi_{(\beta,t_- H)}(E)$.
Then we have 
$\phi_{(\beta,t_0 H)}(F) \geq \phi_{(\beta,t_0 H)}(E)$, which shows
that
$E$ is not $\sigma_{(\beta,t_0 H)}$-stable.
Therefore we get our claim.
\end{proof}

\begin{NB}
For a small positive number $\epsilon$,
$\sigma_{(\beta+sH,tH)}$ ($-\epsilon<s<0, |t-t_0|<\epsilon$)
is a stability condition such that
all irreducible objects $A$ of ${\frak C}$ are stable and
$\sigma_{(\beta,t_0H)}$ is its limit. 
\end{NB}

\begin{prop}\label{prop:exceptional-case}
${\cal M}_{(\beta',\omega')}(v)={\cal M}_{(\beta,t_0 H)}(v)$
for $(\beta',\omega') \in \xi^{-1}(\xi(\beta,t_0 H))$ with
$(\beta',\omega')=(\beta+sH,tH)$, $-\epsilon < s <0$.
\end{prop}

\begin{proof}
For $E \in {\cal M}_{(\beta,t_0 H)}(v)$,
we have a filtration 
\begin{equation}\label{eq:JHF-H}
0 \subset F_1 \subset F_2 \subset \cdots \subset F_s=E
\end{equation}
such that $E_i:=F_i/F_{i-1}$ are $\sigma_{(\beta,t_0 H)}$-stable
and $Z_{(\beta,t_0 H)}(E_i)=\lambda_i Z_{(\beta,t_0 H)}(E)$
with $0 \leq \lambda_i \leq 1$.
If $\lambda_i=0$, then $E_i \in {\frak E}$.
\begin{NB}
Since $d_\beta(E)>0$ and $d_\beta(\bullet)$ is a discrete invariant 
for $\beta \in \NS(X)_{\Bbb Q}$,
there is a filtration \eqref{eq:JHF-H} such that
$E_i:=F_i/F_{i-1}$ does not have a subobject 
$E'$ with $Z_{(\beta,tH)}(E')=\lambda Z_{(\beta,tH)}(E_i)$
and $0<\lambda<1$.
Then applying Corollary \ref{cor:irreducible},
we get a desired decomposition.
\end{NB}
We set $v_i:=v(E_i)$.
Then $\langle v_i^2 \rangle \geq -2$ and $\xi(\beta,t_0 H) \in v_i^\perp$.
\eqref{eq:JHF-H} is the Jordan-H\"{o}lder filtration of
$E$ with respect to $\sigma_{(\beta,t_0 H)}$. 
For the proof of our claim, 
it is sufficient to prove that
\begin{equation}\label{eq:wall-stability}
{\cal M}_{(\beta',\omega')}(v_i)^s={\cal M}_{(\beta,t_0 H)}(v_i)^s
\end{equation}
for $(\beta',\omega') \in \xi^{-1}(\xi(\beta,t_0 H))$ 
with $\beta'=\beta+sH$, $s<0$ and 
any decomposition $v=\sum_i v_i$ of $v$ such that
$\xi(\beta,t_0 H) \in v_i^\perp$ with $\langle v_i^2 \rangle \geq -2$.
Indeed \eqref{eq:wall-stability} means that
semi-stability and its $S$-equivalence class with respect to 
$\sigma_{(\beta',\omega')}$
is the same as those for $\sigma_{(\beta,t_0 H)}$.
We take $t_\pm$ such that $t_-<t_0<t_+$ and 
$t_+ - t_- \ll 1$. 
By Lemma \ref{lem:wall-stable}, we have
\begin{equation}
{\cal M}_{(\beta,t_+ H)}(v_i)^s \cap {\cal M}_{(\beta,t_- H)}(v_i)^s 
={\cal M}_{(\beta,t_0 H)}(v_i)^s.
\end{equation}
We first assume that $\lambda_i \ne 0$.
As in \cite[Rem. 3.6]{Y:wall}, 
we set $\xi_{v_i}(\beta',\omega'):=
\mathrm{Im}(Z_{(\beta',\omega')}(v_i)^{-1}e^{\beta'+i \omega'})$.
Since $\mathrm{Im}(Z_{(\beta,t_0 H)}(v_i)^{-1}e^{\beta+i t_0 H})
=\lambda_i^{-1}
\mathrm{Im}(Z_{(\beta,t_0 H)}(v)^{-1}e^{\beta+i t_0 H})$,
 we get
$\xi_{v_i}(\beta,t_0 H) \in {\Bbb R} \xi(\beta,t_0 H)$.
\begin{NB}
By \cite[Rem. 3.6]{Y:wall},
$Z_{(\beta',\omega')}(x) \in {\Bbb R}Z_{(\beta',\omega')}(v_i)$
if and only if 
$x \in \xi_{v_i}(\beta',\omega')^\perp$.
\end{NB}
By \cite[Rem. 3.6]{Y:wall},
\begin{NB}
Remark \ref{rem:fiber-property2},
\end{NB}
$\xi_{v_i}^{-1}(\xi_{v_i}(\beta,t_0 H))=
\xi^{-1}(\xi(\beta,t_0 H))$.
We take $(\beta'_\pm,\omega'_\pm) \in 
\xi_{v_i}^{-1}(\xi_{v_i}(\beta,t_\pm H))$
which are in a neighborhood of $(\beta',\omega')$.
Then
${\cal M}_{(\beta,t_\pm H)}(v_i)^s=
{\cal M}_{(\beta'_\pm,\omega'_\pm)}(v_i)^s$.
Since 
$$
{\cal M}_{(\beta'_-,\omega'_-)}(v_i)^s 
\cap {\cal M}_{(\beta'_+,\omega'_+)}(v_i)^s=
{\cal M}_{(\beta',\omega')}(v_i)^s,
$$
\begin{NB}
$E_i$ is $\sigma_{(\beta',\omega')}$-stable
with $Z_{(\beta',\omega')}(E_i) \in {\Bbb R}_{>0}Z_{(\beta',\omega')}(v)$.
\end{NB}
we get \eqref{eq:wall-stability}.
If $E_i=G_j$, then
it is also $\sigma_{(\beta',\omega')}$-stable
with $Z_{(\beta',\omega')}(E_i) \in {\Bbb R}_{>0}Z_{(\beta',\omega')}(v)$
by Lemma \ref{lem:G_i}.
\begin{NB}
We don't need anymore.
Indeed we also have ${\cal M}_{(\beta'_\pm,\omega'_\pm)}(v_i)
={\cal M}_{(\beta,t_\pm H)}(v_i)=
\{ G_j \}$
and
${\cal M}_{(\beta',\omega')}(v_i)
={\cal M}_{(\beta,tH)}(v_i)=
\{ G_j \}$.
\end{NB}
Hence $E$ is $\sigma_{(\beta',\omega')}$-semi-stable with a
Jordan-H\"{o}lder filtration \eqref{eq:JHF-H}.

\begin{NB}
Let $v=\sum_i v_i$ be a decomposition of $v$ such that
$$
v_i=e^\beta(r_i+d_i H+D_i +a_i \varrho_X),\;d_i=r_i \frac{d}{r},
a_i=r_i \frac{a}{r}.
$$
We note that $v_i=e^\beta D_i$, if $r_i=0$.
We first prove that
\begin{equation}\label{eq:wall-stability}
{\cal M}_{(\gamma,\omega)}(v_i)^s={\cal M}_{(\beta,tH)}(v_i)^s=
{\cal M}_H^{\beta-\frac{1}{2}K_X}(v_i)^s
\end{equation}
for $(\gamma,\omega) \in \xi^{-1}(\xi(\beta,tH))$.

As in Remark \ref{rem:fiber-property2},
we set $\xi_{v_i}(\gamma,\omega):=
\mathrm{Im}(Z_{(\gamma,\omega)}(v_i)^{-1}e^{\gamma+\sqrt{-1}\omega})$.
By our assumption, we see that
$\xi_{v_i}(\beta,tH) \in {\Bbb R} \xi(\beta,tH)$.
\begin{NB2}
By Remark \ref{rem:fiber-property},
$Z_{(\gamma,\omega)}(x) \in {\Bbb R}Z_{(\gamma,\omega)}(v_i)$
if and only if 
$x \in \xi_{v_i}(\gamma,\omega)^\perp$.
\end{NB2}
By Remark \ref{rem:fiber-property2},
$\xi_{v_i}^{-1}(\xi_{v_i}(\beta,tH))=
\xi^{-1}(\xi(\beta,tH))$.
We first assume that $\rk v_i \ne 0$.
By Lemma \ref{lem:stable-on-wall},
we have 
$$
{\cal M}_H^{\beta-\frac{1}{2}K_X}(v_i)^s=
{\cal M}_H^{\beta_--\frac{1}{2}K_X}(v_i)^s 
\cap {\cal M}_H^{\beta_+-\frac{1}{2}K_X}(v_i)^s.
$$
We take $(\gamma_\pm,\omega_\pm) \in 
\xi_{v_i}^{-1}(\xi_{v_i}(\beta_\pm,tH))$
which are in a neighborhood of $(\gamma,\omega)$.
Then
${\cal M}_H^{\beta_\pm-\frac{1}{2}K_X}(v_i)^s
={\cal M}_{(\beta_\pm,tH)}(v_i)^s=
{\cal M}_{(\gamma_\pm,\omega_\pm)}(v_i)^s$, where $t$ is sufficiently large.
Since 
$$
{\cal M}_{(\gamma_-,\omega_-)}(v_i)^s 
\cap {\cal M}_{(\gamma_+,\omega_+)}(v_i)^s=
{\cal M}_{(\gamma,\omega)}(v_i)^s,
$$
we have ${\cal M}_H^{\beta-\frac{1}{2}K_X}(v_i)^s=
{\cal M}_{(\gamma,\omega)}(v_i)^s$.

Let $E \in {\bf D}(X)$ be an object which is a successive 
extension of $E_i \in {\bf D}(X)$ with
$v(E_i)=v_i$.   
Then 
$E$ is $\beta$-twisted semi-stable if
and only if $E$ is $\sigma_{(\gamma,\beta)}$-semi-stable.  

If $\rk v_i=0$, then 
we also have ${\cal M}_{(\gamma_\pm,\omega_\pm)}(v_i)
={\cal M}_{(\beta_\pm,tH)}(v_i)=
\{ {\cal O}_C(l-1)^{\oplus n} \}$
where $n$ is determined by $v_i=nv({\cal O}_C(l-1))$.
Hence 
${\cal M}_{(\gamma,\omega)}(v_i)
={\cal M}_{(\beta,tH)}(v_i)=
\{ {\cal O}_C(l-1)^{\oplus n} \}$, which implies that 
\eqref{eq:wall-stability} also holds for this case.
\end{NB}
\begin{NB}
If $r_i=0$, then we also have $M_{(\beta,H)}(v_i)=\{{\cal O}_C(l-1) \}$
by Lemma \ref{lem:C-stable}. 
\end{NB}
\end{proof}

\begin{NB}
As far as $(\beta+sH,tH)$ and $(\beta',\omega')$ can be connected 
by a connected curves in $U$,
semi-stability is constant. Hence we have the following.
\begin{cor}
Assume that $\xi^{-1}(\xi(\beta,t_0 H)) \cap V$ is connected in an open subset
$V$ of $U$ containing $(\beta+sH,tH)$ in 
Proposition \ref{prop:exceptional-case}. 
Then
$(\beta',\omega') \in \xi^{-1}(\xi(\beta,t_0 H)) \cap V$ 
belongs to a wall if and only
of $(\beta,t_0 H)$ belongs to a wall.
\end{cor}
\end{NB}

\begin{cor}
If $(\beta,t_0,H)$ belongs to any wall with respect to , that is,
there is $E_1 \oplus E_2 \in {\cal M}_{(\beta,t_0 H)}(v)$
such that $v(E_1) \not \in {\Bbb Q}v(E_2)$,
then $(\beta,t_+ H)$ and $(\beta,t_- H)$ belong to the same 
chamber, where $t_+>t_0>t_-$ and $t_+-t_- \ll 1$. 
\end{cor}

Assume that $v=r+\xi+b \varrho_X \in v(K(X))$ satisfies
$$
((\xi-r\beta) \cdot H)=\min\{((c_1(F)-\rk F \beta)\cdot H) 
>0 \mid F \in K(X) \}.
$$

Then $(\beta,t_0 H)$ lies on a wall for $v$ if and only if
there is a decomposition
$v=\sum_i v_i$ such that
$Z_{(\beta,t_0 H)}(v_i) \in {\Bbb R}_{\geq 0}Z_{(\beta,t_0 H)}(v)$
and $\langle v_i^2 \rangle \geq -2$.
We set
$v:=e^\beta(r+dH+D+a \varrho_X)$ and
$v_i:=e^\beta(r_i+d_i H+D_i+a_i \varrho_X)$ $(D_i \in H^\perp)$.
Then $d_i \geq 0$ for all $i$ and 
$\sum_i d_i=d$.
Hence we may assume that $d_1=d$ and
$d_i=0$ for $i \geq 2$.
Then $Z_{(\beta,t_0 H)}(v_i)=0$ for $i \geq 2$.
By $\langle v_i^2 \rangle \geq -2$, 
$v_i$ $(i \geq 2)$ satisfy
\begin{equation}\label{eq:min}
\begin{split}
& \langle v_i^2 \rangle= -2, \;\; (i \geq 2)\\
& \langle v^2 \rangle
-2\langle v,\sum_{i \geq 2} v_i \rangle+
\langle (\sum_{i \geq 2} v_i)^2 \rangle \geq -2.
\end{split}
\end{equation}

\begin{ex}
Let $X$ be a K3 surface with $\Pic(X)={\Bbb Z}H$.
Let $U$ be an exceptional vector bundle with
$v(U)=e^\beta(r_0+\frac{1}{r_0}\varrho_X)$.
If $v=-e^\beta(r-d H+a \varrho_X)$,
then \eqref{eq:min} is 
$r/r_0+a r_0 \leq \langle v^2 \rangle/2$, where
$v_1=v-v(U)$ and $v_2=v(U)$.
In particular if $r/r_0+a r_0 > \langle v^2 \rangle/2$, then
there is no wall in $\{(\beta+sH,tH) \mid s \leq 0 \}$. 

We shall see the wall by the computation in \cite{MYY:2011:1}.
We set $t_0:=\frac{1}{r_0}\sqrt{\frac{2}{(H^2)}}$.
Then $t=t_0$ is the candidate of a unique wall on the half line
$\{(\beta+sH,tH) \mid s=0\}$.
Let $t_-,t_+$ be numbers with $t_-<t_0<t_+$. 
We take $E:=F[1] \in {\cal M}_{(\beta,t_+ H)}(v)$.
We note that $F^{\vee}$ is a stable sheaf by
\cite[Cor. 3.2.1]{MYY:2011:1}.
If $\Hom(U,E) \ne 0$, then
we have a stable sheaf
$(F')^{\vee}$ fitting in the extension
$$
0 \to U^{\vee} \to (F')^{\vee} \to F^{\vee} \to 0,
$$
which gives an exact sequence
$$
0 \to U \to F[1] \to F'[1] \to 0
$$ 
in ${\cal A}_{(\beta+sH,tH)}$ for $s<0$ (and $t \gg 0$).
The condition for the existence of $F'$ is 
$\langle v(F')^2 \rangle=\langle v^2 \rangle-2(r/r_0+a r_0)-2 \geq -2$.
Therefore $U$ defines a wall in $s<0$ if 
$r/r_0+a r_0 \leq  \langle v^2 \rangle/2$.
\end{ex}

\begin{rem}
If $\langle v,v(U) \rangle=
r/r_0+a r_0 \leq \langle v^2 \rangle/2$,
then $U[1]$ defines a wall 
in $\{(\beta+sH,tH) \mid s > 0 \}$. 
Indeed for $E \in {\cal M}_{(\beta,t_+ H)}(v)$,
we have $\Hom(E,U[1]) \ne 0$. Hence we have an exact sequence
in ${\cal A}_{(\beta,t_+ H)}$
\begin{equation}\label{eq:U-2}
0 \to E' \to E \to U[1] \to 0,
\end{equation}
where $E'$ is also $\sigma_{(\beta,t_+ H)}$-stable.
In the region $\{(\beta+sH,tH) \mid s > 0 \}$,
\eqref{eq:U-2} gives a wall for $v$.
\end{rem}

 
\begin{NB}
For a Mukai vector $v=e^\beta(-r+dH+D-a \varrho_X)$ ($r,d>0$),
we shall study wall and chamber in the subspace 
$\{(\beta+sH,tH) \mid t>0, -\epsilon<s \leq 0 \}$,
where $\epsilon>0$ is sufficiently small.
We assume that $s=0$ is not a wall.
For this purpose, 
we shall study ${\cal M}_{(\beta,tH)}(v)$.
Let $E=\oplus_i E_i \in {\cal M}_{(\beta,tH)}(v)$  
be an object such that $E_i$ are $\sigma_{(\beta,tH)}$-stable objects
with $Z_{(\beta,tH)}(E_i) \in {\Bbb R}_{\geq 0}Z_{(\beta,tH)}(v)$.
We set
$v(E_i):=e^\beta(-r_i+d_i H+D_i-a_i \varrho_X)$.
Then $d_i \geq 0$, $\sum_i d_i=d$,
$\langle v(E_i)^2 \rangle \geq -2$ and
$(d r_i-d_i r)t^2(H^2)/2+(d_i a-d a_i)=0$.
If $d r_i-d_i r=0$ for all $i$, then
$d_i a-d a_i=0$ for all $i$.
Since $s=0$ is not a wall,
$v(E_i) \in {\Bbb Q}v$ for all $i$. 

If $d r_i-d_i r \ne 0$, there is $j$ such that 
$d r_j-d_j r>0$.
Since $d_i \geq 0$ for all $i$,
$0 \leq d_j \leq d$.
If $d_j=0$, then $d>0$ implies 
$r_j>0$.
By $d_j=0$ and $\Ima Z_{(\beta+sH,tH)}(E_j)>0$
for $s<0$, 
$E_j$ is a $\beta$-stable sheaf with $Z_{(\beta,tH)}(E_j)=0$.
In particular we have $-r_j=\rk E_j>0$, which is a contradiction.
Therefore $d_j>0$.
If $d_j a-d a_j \geq 0$, then $E_j$ does not define a wall.

Thus we need the following condition:

If $d r_1-d_1 r>0$, then $a_1 \leq d_1 a/d$.

If $d >N$, then this condition holds.
\end{NB}

\end{document}